\documentclass{amsart}
\usepackage{fullpage}
\usepackage{amsmath, amsthm, amsfonts}
\usepackage{amssymb}
\usepackage{latexsym}
\usepackage{verbatim}
\usepackage{a4wide}
\usepackage[all,cmtip]{xy}
\usepackage{color}
\usepackage[backref,breaklinks]{hyperref}
\usepackage{breakurl}
\definecolor{mylinkcolor}{rgb}{0.2,0.5,0.2}
\definecolor{myurlcolor}{rgb}{0.0,0.0,0.75}
\hypersetup{colorlinks=true,urlcolor=myurlcolor,citecolor=myurlcolor,linkcolor=mylinkcolor,linktoc=page,breaklinks=true}

\usepackage{enumerate}
\usepackage{amssymb}
\usepackage{array}
\usepackage{nicefrac}
\usepackage[section]{placeins}
\usepackage{mathtools}
\usepackage{booktabs}
\usepackage{subcaption}

\DeclareMathAlphabet{\mathfr}{U}{euf}{m}{n}
\newtheorem{maintheorem}{Theorem}
\newtheorem{theorem}{Theorem}[section]

\newtheorem{proposition}[theorem]{Proposition}
\newtheorem{corollary}[theorem]{Corollary}

\newtheorem{lemma}[theorem]{Lemma}

\newtheorem{question}[theorem]{Question}

\theoremstyle{definition}
\newtheorem{remark}[theorem]{Remark}
\newtheorem{definition}[theorem]{Definition}

\newcommand{\cO}{{\mathcal O}}

\newcommand{\Q}{\mathbb Q}
\newcommand{\Qbar}{{\overline{\mathbb Q}}}
\newcommand{\Gal}{\mathrm{Gal}}

\newcommand{\R}{\mathbb R}
\newcommand{\Z}{\mathbb Z}

\newcommand{\F}{\mathbb F}

\newcommand{\C}{\mathbb C}
\newcommand{\SL}{\mathrm{SL}}
\newcommand{\GL}{\mathrm{GL}}
\newcommand{\M}{\mathrm{M }}
\newcommand{\PGL}{\mathrm{PGL}}
\newcommand{\PSL}{\mathrm{PSL}}

\newcommand{\id}{\operatorname{id}}

\newcommand{\End}{\operatorname{End}}
\newcommand{\Hom}{\operatorname{Hom}}

\newcommand{\Frob}{\operatorname{Frob}}
\newcommand{\Aut}{\operatorname{Aut}}

\newcommand{\Jac}{\operatorname{Jac}}

\newcommand{\ord}{\operatorname{ord}}
\newcommand{\Rr}{\operatorname{Re}}
\newcommand{\Ii}{\operatorname{Im}}

\newcommand{\G}{\mathrm{G}}

\newcommand{\p}{\mathfrak{p}}
\newcommand{\Res}{\operatorname{Res}}

\newcommand{\Tr}{\operatorname{Tr}}

\newcommand{\Ker}{\operatorname{Ker}}
\newcommand{\im}{\mathrm{Im}}

\newcommand{\GSp}{\mathrm{GSp}}

\newcommand{\Sp}{\mathrm{Sp}}

\newcommand{\USp}{\mathrm{USp}}

\newcommand{\ST}{\mathrm{ST}}

\newcommand{\Lef}{\operatorname{L}}

\newcommand{\AST}{\operatorname{AST}}
\newcommand{\topo}{\operatorname{top}}

\newcommand{\sym}[1]{{\mathrm{S}_#1}}

\newcommand{\T}{\mathrm{T}}

\numberwithin{equation}{section}

\usepackage{url}

\newcommand{\Unitary}{\mathrm{U}}
\DeclareMathOperator{\Trace}{Trace}

\begin{document}
\title{Sato-Tate distributions of twists of\\the Fermat and the Klein quartics}

%\author[Fit\'e]{Francesc Fit\'e}
%\address{Departament de Matem\`atiques, Universitat Polit\`ecnica de Catalunya/BGSmath\\Edifici Omega\,\\ C/Jordi Girona 1--3\\
%08034 Barcelona\\ Catalonia
%}
%\email{francesc.fite@gmail.com}
%\urladdr{https://mat-web.upc.edu/people/francesc.fite/}

\author[Fit\'e]{Francesc Fit\'e}
\address{Institute for Advanced Study\\Fuld Hall\\ 1 Einstein Drive\\
Princeton \\ New Jersey 08540\\USA}
\email{ffite@ias.edu}
\urladdr{https://mat-web.upc.edu/people/francesc.fite/}

\author[Lorenzo]{Elisa Lorenzo Garc\'ia} 
\address{Laboratoire IRMAR, 
	Universit\'e de Rennes 1\\
	Campus de Beaulieu, 
	35042, Rennes Cedex\\
	France}
\email{elisa.lorenzogarcia@univ-rennes1.fr}
\urladdr{https://sites.google.com/site/elisalorenzo/home}

\author[Sutherland]{Andrew V. Sutherland}
\address{Department of Mathematics\\Massachusetts Institute of Technology\\77 Massachusetts Avenue\\Cambridge, Massachusetts  02139\\USA}
\email{drew@math.mit.edu}
\urladdr{https://math.mit.edu/~drew}

\begin{abstract}
We determine the limiting distribution of the normalized Euler factors of an abelian threefold $A$ defined over a number field $k$ when $A$ is $\Qbar$-isogenous to the cube of a CM elliptic curve defined over $k$.
As an application, we classify the Sato--Tate distributions of the Jacobians of twists of the Fermat and Klein quartics, obtaining 54 and 23, respectively, and 60 in total. We encounter a new phenomenon not visible in dimensions~$1$ or~$2$: the limiting distribution of the normalized Euler factors is not determined by the limiting distributions of their coefficients.
\end{abstract}
\maketitle
\tableofcontents

\section{Introduction}\label{section: introduction}

Let $A$ be an abelian variety of dimension $g\geq 1$ defined over a number field $k$. For a prime~$\ell$, let~$V_\ell(A)\coloneqq \Q \otimes \varprojlim_n\!\!\! A[\ell^n]$ be the (rational) $\ell$-adic Tate module of~$A$, and let
$$
\varrho_{A}\colon G_k\rightarrow \Aut(V_\ell(A))
$$
be the $\ell$-adic representation arising from the action of the absolute Galois group $G_k$ on $V_\ell(A)$.
Let~$\p$ be a prime of~$k$ (a nonzero prime ideal of the ring of integers $\cO_k$) not lying above the rational prime $\ell$. The $L$-\emph{polynomial} of~$A$ at the prime $\p$ is defined by 
$$
L_\p(A,T)\coloneqq\det(1-\varrho_A(\Frob_\p)T;\,V_\ell(A)^{I_\p})\in\Z[T]\,,
$$
where $\Frob_\p$ denotes a Frobenius element at $\p$ and $I_\p$ is the inertia subgroup at $\p$; it does not depend on the choice of $\ell$.
Let~$S$ be a finite set of primes of $k$ that includes all primes of bad reduction for $A$ and all primes lying above $\ell$.
For $\p\not\in S$ the polynomial $L_\p(A,T)$ has degree~$2g$ and coincides with the numerator of the zeta function of the reduction of $A$ modulo $\p$. The $L$-function of $A$ is defined as the Euler product
$$
L(A,s)\coloneqq \prod_\p L_\p(A,N(\p)^{-s})^{-1},
$$
where $N(\p)\coloneqq [\cO_k\!:\!\p]$ is the (absolute) norm of $\p$.  The \emph{normalized $L$-polynomial} of $A$ at $\p$ is the monic polynomial $\overline L_\p(A,T)\coloneqq L_ \p(A,N(\p)^{-1/2}T)\in \R[T]$; its roots come in complex conjugate pairs and lie on the unit circle, as shown by Weil in \cite{Weil}.

As constructed by Serre in \cite{Ser12}, the Sato--Tate group $\ST(A)$ is a compact real Lie subgroup of $\USp(2g)$, defined up to conjugacy in $\GL_{2g}(\C)$, that comes equipped with a map that assigns to each prime $\p \not\in  S$ a semisimple conjugacy class $s(\p)$ of $\ST(A)$ for which
$$
\det(1-s(\p)T)=\overline L_\p(A,T).
$$
Let $\mu$ be the pushforward of the Haar measure of $\ST(A)$ to its set of conjugacy classes $X$, 
and let $\{s(\p)\}_\p$ denote the sequence of conjugacy classes $s(\p)$ arranged in an order compatible with the partial ordering of primes $\p$ by norm.
The generalized Sato--Tate conjecture predicts that:
\smallskip

\begin{enumerate}
\item[(ST)] The sequence $\{s(\p)\}_\p$ is equidistributed on $X$ with respect to the measure $\mu$.
\end{enumerate}
\smallskip

This conjecture has been proved for abelian varieties of dimension one (elliptic curves) over a totally real \cite{HSBT10} or CM number field \cite{ACCGHLNSTT}, and in several special cases for abelian varieties of higher dimension, including abelian varieties with potential CM \cite{Joh17}.

For each $s\in X$, we write $\det(1-sT)\eqqcolon\sum_{j=0}^{2g} a_jT^j$, and define
$$
 I_j\coloneqq\left[ -\binom{2g}{j},\binom{2g}{j}\right]\,,\qquad \text{and}\qquad I:=\prod_{j=1}^g I_j\,.
$$
For $0\leq j \leq 2g$ we have $a_j\in I_j$ and $a_j=a_{2g-j}$ for $0 \leq j\leq 2g$, since the eigenvalues of any conjugacy class of $\USp(2g)$ come in complex conjugate pairs on the unit circle. Consider the maps
$$
\Phi\colon X {\longrightarrow} I\,,\qquad \Phi_j:=X \stackrel{\Phi}{\longrightarrow} I \stackrel{\varpi_j}{\longrightarrow} I_j\,,
$$
where $\Phi$ is defined by $\Phi(s)=(a_1,\dots,a_g)$ and $\varpi_j$ is the projection to the $j$th component. Let $\mu_I$ (resp.\ $\mu_{I_j}$) denote the projection of the measure $\mu$ by the map $\Phi$ (resp.\ $\Phi_j$). We will call $\mu_I$ the \emph{joint coefficient measure} and the set of measures $\{\mu_{I_j}\}_j$, the \emph{independent coefficient measures}.

The measures $\mu_I$ and $\mu_{I_j}$ are respectively determined by their moments
\begin{equation}\label{equation: moments}
\M_{n_1,\dots,n_g}[\mu_I]\coloneqq\int_I a_1^{n_1}\cdots a_g^{n_g}\mu_I(a_1,\dots,a_g)\,,\qquad \M_{n}[\mu_{I_j}]\coloneqq\int_{I_j} a_j^{n}\mu_{I_j}(a_j)\,,
\end{equation}
for $n_1,\dots,n_g\geq 0$ and $n\geq 0$.
We denote by $a_j(A)(\p)$, or simply $a_j(\p)$, the $j$th coefficient $\Phi_j(s(\p))$ of the normalized $L$-polynomial, and by $a(A)(\p)$, or simply $a(\p)$, the $g$-tuple $\Phi(s(\p))$ of coefficients of the normalized $L$-polynomial. We can now consider the following successively weaker versions of the generalized Sato--Tate conjecture:
\smallskip

\begin{enumerate}
\item[(ST${}^\prime$)] The sequence $\{a(\p)\}_\p$ is equidistributed on $I$ with respect to $\mu_I$.
\medskip

\item[(ST${}^{\prime\prime}$)] The sequences $\{a_j(\p)\}_\p$ are equidistributed on $I_j$ with respect to $\mu_{I_j}$, for $1\leq j \leq g$.
\end{enumerate}
\smallskip

Let $\pi(x)$ count the number of primes $\p\not\in S$ for which $N(\p)\leq x$. If we define

\begin{equation}\label{equation: defgenmom}
\M_{n_1,\dots,n_g}[a]\coloneqq\lim _{x\rightarrow \infty} \frac{1}{\pi(x)}\!\!\sum_{N(\p)\leq x}\!\! a_1(\p)^{n_1} \cdots a_g(\p)^{n_g}\,,
\quad
\M_n[a_j]\coloneqq\lim _{x\rightarrow \infty} \frac{1}{\pi(x)}\!\!\sum_{N(\p)\leq x}\!\! a_j(\p)^{n}\,,
\end{equation}
then (ST${}^\prime$) holds if and only if $\M_{n_1,\dots,n_g}[\mu_I]=\M_{n_1,\dots,n_g}[a]$ for every $n_1,\dots,n_g\geq 0$, while (ST${}^{\prime\prime}$) holds if and only if  $\M_{n}[\mu_{I_j}]=\M_n[a_j]$ for every $n\geq 0$ and $1\leq j\leq g$.

Let $A'$ be an abelian variety defined over a number field $k'$ also of dimension $g$, and let $X'$, $\mu'$, $\mu_I'$, $\mu_{I_j}'$ be the data associated to $A'$ corresponding to $X$, $\mu$, $\mu_I$, $\mu_{I_j}$, respectively. The following implications are immediate:
\begin{equation}\label{equation: implication}
\ST(A)=\ST(A') \quad \Rightarrow \quad  \mu_I=\mu_I'\quad \Rightarrow \quad  \{\mu_{I_j}\}_j=\{\mu_{I_j}'\}_j\,.
\end{equation}
The classification of Sato--Tate groups of elliptic curves and abelian surfaces together with the explicit computation of their Haar measures implies that for $g\leq 2$ the converses of the implications in \eqref{equation: implication} both hold; see \cite{FKRS12}. In this article, we show that for $g=3$, there are cases in which the converse of the second implication of \eqref{equation: implication} fails to hold.\footnote{Using Gassmann triples one can construct examples (of large dimension) where the converse of the first implication in \eqref{equation: implication} also fails to hold, but we will not pursue this here.}
\smallskip

{\textbf{Main result.}}
In this article we obtain the counterexamples alluded to in the previous paragraph by searching among abelian threefolds defined over a number field that are $\Qbar$-isogenous to the cube of an elliptic curve with complex multiplication (CM).
More precisely, we obtain a complete classification of the Sato--Tate groups, the joint coefficient measures, and the independent coefficient measures of the Jacobians of twists of the Fermat and the Klein quartics (which are both $\Qbar$-isogenous to the cube of a CM elliptic curve). The Fermat and Klein quartics are respectively given by the equations
\begin{equation}\label{equation: Fermatandklein}
\tilde C^0_1\colon x^4+y^4+z^4=0\,,
\qquad
\tilde C^0_7\colon x^3y+y^3z+z^3x=0\,,
\end{equation}
and they have the two largest automorphism groups among all genus 3 curves, of sizes 96 and 168, respectively.
Our main result is summarized in the following theorem.

\begin{maintheorem}\label{theorem: Main}
The following hold:
\begin{enumerate}[{\rm (i)}]
\item There are $54$ distinct Sato--Tate groups of twists of the Fermat quartic.  These give rise to $54$ (resp.\ $48$) distinct joint (resp.\ independent) coefficient measures.
\item There are $23$ distinct Sato--Tate groups of twists of the Klein quartic.  These give rise to $23$ (resp.\ $22$) distinct joint (resp.\ independent) coefficient measures.
\item There are $60$ distinct Sato--Tate groups of twists of the Fermat or the Klein quartics.  These give rise to $60$ (resp.\ $54$) distinct joint (resp.\ independent) coefficient measures.
\end{enumerate}
\end{maintheorem}

One motivation for our work is a desire to extend the classification of Sato-Tate groups that is known for dimensions $g\le 2$ to dimension $3$.
Of the 52 Sato-Tate groups that arise for abelian surfaces (see \cite[Table 10]{FKRS12} for a list), 32 can be realized as the Sato-Tate group of the Jacobian of a twist of one of the two genus 2 curves with the largest automorphism groups, as shown in \cite{FS14}; these groups were the most difficult to treat in \cite{FKRS12} and notably include cases missing from the candidate list of trace distributions identified in \cite[Table 13]{KS09}.
While the classification of Sato-Tate groups in dimension 3 remains open, the 60 Sato-Tate groups identified in Theorem~\ref{theorem: Main} and explicitly described in \S\ref{section: ST groups} are likely to include many of the most delicate cases and represent significant progress toward this goal.
\smallskip

{\textbf{Overview of the paper.}} This article can be viewed as a genus $3$ analog of \cite{FS14}, where the Sato--Tate groups of the Jacobians of twists of the curves $y^2=x^5-x$ and $y^2=x^6+1$ were computed. However, there are two important differences in the techniques we use here; these are highlighted in the paragraphs below that outline our approach.  We also note \cite{FS16}, where the Sato-Tate groups of the Jacobians of certain twists of the genus 3 curves $y^2=x^7-x$ and $y^2=x^8+1$ are computed, and \cite{ACLLM}, where the Sato-Tate groups of the Jacobians of twists of the curve $y^2=x^8-14x^4+1$ are determined.
Like the Fermat and Klein quartics we consider here, these three curves represent extremal points in the moduli space of genus 3 curves, but they are all hyperelliptic, and their automorphism groups are smaller (of order 24, 32, 48, respectively).

As noted above, the Sato--Tate conjecture is known for abelian varieties that are $\Qbar$-isogenous to a product of CM abelian varieties \cite[Cor.\,15]{Joh17}. It follows that we can determine the set of independent coefficient measures $\{\mu_{I_j}\}_j$ by computing the sequences $\{\M_n[a_j]\}_{j,n}$, and similarly for~$\mu_I$ and the sequences $\{M_{n_1,\ldots,n_g}[a]\}_{n_1,\ldots,n_g}$. Closed formulas for these sequences are determined in \S\ref{section: cubes} in the more general setting of abelian threefolds defined over a number field~$k$ that are $\Qbar$-isogenous to the cube of an elliptic curve defined over~$k$; see Proposition~\ref{proposition: import} and Corollary~\ref{corollary: import}. This analysis closely follows the techniques developed in \cite[\S3]{FS14}.

In \S\ref{section: twists}, we specialize to the case of  Jacobians of twists of the Fermat and Klein quartics. In \S\ref{section: momsequence}, we obtain a complete list of possibilities for $\{\M_n[a_j]\}_{j,n}$: there are 48 in the Fermat case, 22 in the Klein case, and 54 when combined; see Corollary~\ref{corollary: det mom seq}. We also compute lower bounds on the number of possibilities for $\{\M_{n_1,n_2,n_3}[a]\}_{n_1,n_2,n_3}$ by computing the number of possibilities for the first several terms (up to a certain conveniently chosen bound) of this sequence. These lower bounds are 54 in the Fermat case, 23 in the Klein case, and 60 when combined; see Proposition~\ref{proposition: lowernumb}.
 
The first main difference with \cite{FS14} arises in \S\ref{section: ST groups}, where we compute the Sato--Tate groups of the twists of the Fermat and Klein quartics using the results of \cite{BK15}. Such an analysis would have been redundant in \cite{FS14}, since a complete classification of Sato--Tate groups of abelian surfaces was already available from \cite{FKRS12}. We show that there are at most 54 in the Fermat case, and at most 23 in the Klein case; see Corollaries~\ref{corollary: uppernumb1} and~\ref{corollary: uppernumb2}.
Combining the implications in \eqref{equation: implication} together with the lower bounds of \S\ref{section: momsequence} and upper bounds of \S\ref{section: ST groups} yields Theorem~\ref{theorem: Main} above. 

The second main difference with \cite{FS14} arises in \S\ref{section: curves}, where we provide explicit equations of twists of the Fermat and Klein quartics that realize each of the possible Sato--Tate groups. Here, the computational search used in \cite{FS14} is replaced by techniques developed in \cite{Lor17,Lor18} that involve the resolution of certain Galois embedding problems, and a moduli interpretation of certain twists $X_E(7)$ of the Klein quartic as twists of the modular curve $X(7)$, following \cite{HK00}.
In order to apply the latter approach, which also plays a key role in \cite{PSS07}, we obtain a computationally effective description of the minimal field over which the automorphisms of $X_E(7)$ are defined (see Propositions~\ref{proposition: modularK} and~\ref{proposition: phi7}), a result that may have other applications.

Finally, in \S\ref{section: computations}, we give an algorithm for the efficient computation of the $L$-polynomials of twists of the Fermat and Klein quartics
This algorithm combines an average polynomial-time for computing Hasse-Witt matrices of smooth plane quartics \cite{HS} with a result specific to our setting that allows us to easily derive the full $L$-polynomial at $\p$ from the Frobenius trace using the splitting behavior of $\p$ in certain extensions; see Proposition~\ref{proposition: trthetaMEA}.
Our theoretical results do not depend on this algorithm, but it played a crucial role in our work by allowing us to check our computations and may be of independent interest.
\smallskip

{\textbf{Acknowledgements.}} We thank Josep Gonz\'alez for his help with Lemma~\ref{lemma: embeddings}, and we are grateful to the Banff International Research Station for hosting a May 2017 workshop on Arithmetic Aspects of Explicit Moduli Problems where we worked on this article.
Fit\'e is grateful to the University of California at San Diego for hosting his visit in spring 2012, the period in which this project was conceived. Fit\'e received financial support from the German Research Council (CRC 701), the Excellence Program Mar\'ia de Maeztu MDM-2014-0445, and MTM2015-63829-P. Sutherland was supported by NSF grants DMS-1115455 and DMS-1522526. 
This project has received funding from the European Research Council (ERC) under the European Union’s Horizon 2020 research and innovation programme (grant agreement No 682152), and from the Simons Foundation~ (grant~\#550033).
\smallskip

{\textbf{Notation.}} Throughout this paper, $k$ denotes a number field contained in a fixed algebraic closure $\overline \Q$ of $\Q$. All the field extensions of $k$ we consider are algebraic and assumed to lie in $\Qbar$.
We denote by $G_k$ the absolute Galois group $\Gal(\Qbar/k)$. For an algebraic variety~$X$ defined over~$k$ and a field extension $L/k$, write $X_L$ for the algebraic variety defined over $L$ obtained from~$X$ by base change from~$k$ to~$L$. For abelian varieties~$A$ and~$B$ defined over~$k$, we write $A\sim B$ if there is an isogeny from $A$ to $B$ that is defined over~$k$. We use $M^\T$ to denote the transpose of a matrix~$M$. We label the isomorphism class ID$(H)=\langle n, m\rangle$ of a finite group $H$ according to the Small Groups Library \cite{SGL}, in which $n$ is the order of~$H$ and~$m$ distinguishes the isomorphism class of $H$ from all other isomorphism classes of groups of order~$n$.

\section{Equidistribution results for cubes of CM elliptic curves}\label{section: cubes}

Let $A$ be an abelian variety over $k$ of dimension $3$ such that $A_\Qbar\sim E_\Qbar^3$, where $E$ is an elliptic curve defined over $k$ with complex multiplication (CM) by an imaginary quadratic field $M$.
Let $L/k$ be the minimal extension over which all the homomorphisms from $E_\Qbar$  to $A_\Qbar$ are defined.
We note that $kM\subseteq L$, and we have $\Hom(E_\Qbar,A_\Qbar)\simeq\Hom(E_{L} ,A_{L} )$ and $A_{L} \sim_{L}  E_{L} ^3$.

Let $\sigma$ and $\overline\sigma$ denote the two embeddings of $M$ into $\Qbar$.
Consider
$$
\Hom(E_{L} ,A_{L} )\otimes_{M,\sigma}\Qbar \qquad \text{(resp.\ $\End(A_{L} )\otimes_{M,\sigma}\Qbar$)}\,,
$$
where the tensor product is taken via the embedding $\sigma\colon M\hookrightarrow\Qbar$. Letting $\Gal(L/kM)$ act trivially on $\Qbar$, it acquires the structure of a $\overline \Q[\Gal(L/kM)]$-module of dimension 3 (resp.~9) over~$\Qbar$, and similarly for $\overline\sigma$.

\begin{definition}\label{definition: thetaMsigma}
Let $\theta\coloneqq\theta_{M,\sigma}(E,A)$ (resp.\ $\theta_{M,\sigma}(A)$) denote the representation afforded by the module $\Hom(E_{L} ,A_{L} )\otimes_{M,\sigma}\Qbar$ (resp.\ $\End(A_{L} )\otimes_{M,\sigma}\Qbar$), and similarly define $\overline\theta \coloneqq \theta_{M,\overline\sigma}(E,A)$ and $\theta_{M,\overline\sigma}(A)$.
Let $\theta_\Q\coloneqq\theta_\Q(E,A)$ (resp.\ $\theta_\Q(A)$) denote the representation afforded by the $\Q[\Gal({L} /k)]$-module $\Hom(E_{L},A_{L})\otimes\Q$ (resp.\ $\End(A_L)\otimes\Q$).
\end{definition}
For each $\tau \in \Gal({L} /kM)$, we write
$$\det(1-\theta(\tau)T)=1+a_1(\theta)(\tau)T+a_2(\theta)(\tau)T^2+a_3(\theta)(\tau)T^3\,, $$
so that $a_1(\theta)=-\Tr\theta$ and $a_3(\theta)=-\det(\theta)$.

Fix a subextension $F/kM$ of $L/kM$, and let $S$ be the set of primes of~$F$ for which $A_F$ or $E_F$ has bad reduction. Note that by \cite[Thm. 4.1]{Sil92} the set $S$ contains the primes of $F$ ramified in $L$. For $z\in M$, write $|z|\coloneqq\sqrt{\sigma(z)\cdot\overline\sigma(z)}$. For $\p \not \in S$, there exists $\alpha(\p)\in M$, such that $|\alpha(\p)|=N(\p)^{1/2}$ and
\begin{equation}\label{equation: alpha}
a_1(E_F)(\p)=-\frac{\sigma(\alpha(\p))+\overline\sigma(\alpha(\p))}{N(\p)^{1/2}}\,.
\end{equation}

\begin{proposition}\label{proposition: import}
Let $A$ be an abelian variety of dimension $3$ defined over $k$ such that $A_L\sim_L E_L^3$, where $E$ is an elliptic curve defined over $k$ with CM by the imaginary quadratic field $M$.
Suppose that $a_3(\theta)(\tau)$ is rational for every $\tau\in G:=\Gal(L/kM)$.
Then for $i=1,2,3$, the sequence $a_i(A_{kM})$ is equidistributed on $I_i=\left[-\binom{2g}{i},\binom{2g}{i}\right]$  with respect to a measure that is continuous up to a finite number of points and therefore uniquely determined by its moments. For $n\geq 1$, we have $\M_{2n-1}[a_1(A_{kM})]=\M_{2n-1}[a_3(A_{kM})]=0$ and:
$$
\begin{array}{lll}
\M_{2n}[a_1(A_{kM})] & = &\frac{1}{|G|}\sum_{\tau\in G}|a_1(\theta)(\tau)|^{2n} \binom{2n}{n}\,,\\[6pt]

\M_n[ a_2(A_{kM})] & = &\frac{1}{|G|}\sum_{\tau\in G}\sum_{i=0}^{n}\binom{n}{i}\binom{2i}{i} |a_2(\theta)(\tau)|^i \left(|a_1(\theta)(\tau)|^2-2\cdot |a_2(\theta)(\tau)|\right)^{n-i}\,,\\[6pt]

\M_{2n}[a_3(A_{kM})] &= & \frac{1}{|G|}\left(\sum_{\tau\in G}\sum_{i=0}^{n}\binom{2n}{2i}\sum_{j=0}^{2i}\sum_{k=0}^{n-i}\binom{2i}{j}(r_1(\tau)-3)^{2i-j}\right.\\[5pt]
&& \,\,\,\,\,\,\,\,\,\,\,\left.\cdot r_2(\tau)^{2n-2i}\binom{n-i}{k}4^k(-1)^{n-i-k}\binom{2j+2n-2k}{j+n-k}\right)\,.
\end{array}
$$
Here $r_1(\tau)$ and $r_2(\tau)$ are the real and imaginary parts of $a_3(\theta)(\tau)a_2(\theta)(\tau)\overline a_1(\theta)(\tau)$, respectively.
\end{proposition}
\begin{proof}
The proof follows the steps of \cite[\S3.3]{FS14}. Define
$$V_\sigma(A)=V_\ell(A_{kM})\otimes_{M\otimes \Q_\ell}\Qbar_\ell\,,$$
where the tensor product is taken relative to the map of $\Q_\ell$-algebras $M\otimes\Q_\ell\rightarrow\Qbar_\ell$ induced by $\sigma$; similarly define $V_{\overline\sigma}(A)$, $V_{\sigma}(E)$, and $V_{\overline\sigma}(E)$. We then have isomorphisms of $\overline \Q_\ell[G_{kM}]$-modules
$$
V_\ell(A_{kM})\simeq V_\sigma(A)\oplus V_{\overline\sigma}(A)\,,\qquad
V_\ell(E_{kM})\simeq V_\sigma(E)\oplus V_{\overline\sigma}(E)\,.
$$
It follows from Theorem 3.1 in \cite{Fit10}, that
$$
V_\sigma(A)\simeq \theta_{M,\sigma}(E,A)\otimes V_\sigma(E)\,,\qquad V_{\overline\sigma}(A)\simeq \theta_{M,\overline\sigma}(E,A)\otimes V_{\overline\sigma}(E)\,.
$$
We thus have an isomorphism of $\overline \Q_\ell[G_{kM}]$-modules
\begin{equation}\label{equation:flip}
V_\ell(A_{kM})\simeq \bigl(\theta_{M,\sigma}(E,A)\otimes V_\sigma(E)\bigr)\,\oplus\,\bigl(\theta_{M,\overline\sigma}(E,A)\otimes V_{\overline\sigma}(E)\bigr)\,.
\end{equation}
For each prime $\p\not \in S$, let us define
$$
\alpha_1(\p)\coloneqq\frac{\sigma(\alpha(\p))}{N(\p)^{1/2}}\,,\qquad
\overline\alpha_1(\p)\coloneqq\frac{\overline\sigma(\alpha(\p))}{N(\p)^{1/2}}\,,
$$
where $\sigma(\alpha(\p))$, as in equation~(\ref{equation: alpha}), gives the action of $\Frob_\p$ on $V_\sigma(E)$. It follows from (\ref{equation:flip}) that
\begin{equation}\label{equation: icoeffs}
\begin{array}{l}
a_1(A_{kM})(\p)=a_1(\p)\alpha_1(\p)+\overline a_1(\p)\overline\alpha_1(\p)\,,\\[4pt]

a_2(A_{kM})(\p)= a_2(\p) \alpha_1(\p)^2+\overline a_2(\theta)\overline\alpha_1(\p)^2+
a_1(\p)\overline a_1(\p) \,,\\[4pt]

a_3(A_{kM})(\p)= a_3(\p)\alpha_1(\p) ^3+\overline a_3(\p)\overline\alpha_1(\p)^3+\overline a_1(\p)a_2(\p)\alpha_1(\p)+a_1(\p)\overline a_2(\p)\overline\alpha_1(\p)\,,
\end{array}
\end{equation}
where to simplify notation we have written $a_i(\mathfrak p):=a_i(\theta)(\Frob_\mathfrak p)$ and $\overline a_i(\mathfrak p):=a_i(\overline \theta)(\Frob_\mathfrak p)$. Let $r_1(\p)$ and $r_2(\p)$ denote the real and imaginary parts of $a_3(\p)a_2(\p)\overline a_1(\p)$, respectively.  We have $a_3(\p)^2=1$, since $a_3(\p)$ is a rational root of unity, and we can rewrite the above expressions as
$$
\begin{array}{lll}
a_1(A_{kM})(\p)&=&|a_1(\p)|\left(z_1(\p)\alpha_1(\p)+\overline z_1(\p)\overline\alpha_1(\p)\right)\,,\\[6pt]

a_2(A_{kM})(\p)&= &|a_2(\p)|\left(z_2(\p)\alpha_1(\p)+\overline z_2(\p)\overline\alpha_1(\p)\right)^2-2|a_2(\p)|+|a_1(\p)|^2\,,\\[6pt]

a_3(A_{kM})(\p)&= &a_3(\p)\big(\left(\alpha_1(\p)+\overline\alpha_1(\p)\right)^3+(r_1(\p)-3)\left(\alpha_1(\p)+\overline\alpha_1(\p)\right)\\[2pt]
&&\pm r_2(\p)\sqrt{4-(\alpha_1(\p)+\overline\alpha_1(\p))^2}\big)\,,
\end{array}
$$
where
$$
z_1(\p)=\frac{a_1(\p)}{|a_1(\p)|},\quad z_2(\p)=\left(\frac{a_2(\p)}{|a_2(\p)|}\right)^ {1/2} \in \Unitary(1)\,.
$$

Let $\alpha_1$ denote the sequence $\{\alpha_1(\p_i)\}_{i\geq 1}$ and, for each conjugacy class $c$ of $\Gal(L/kM)$, let $\alpha_{1,c}$ denote the subsequence of $\alpha_1$ obtained by restricting to primes $\p$ of $\not \in S$ such that $\Frob_\p=c$. By the translation invariance of the Haar measure and \cite[Prop. 3.6]{FS14}, for $z\in\Unitary(1)$ and $i\geq 1$ we have
\begin{equation}\label{equation: 36FS}
\M_i[z\alpha_{1,c}+\overline z\, \overline\alpha_{1,c}]=\M_i[\alpha_{1,c}+\overline\alpha_{1,c}]=
\begin{cases}
\binom{i}{i/2} & \text{if $i$ is even,}\\
0 & \text{if $i$ is odd}.
\end{cases}
\end{equation}
The formulas for $\M_{2n}[a_1(A_{kM})]$, $\M_n[a_2(A_{kM})]$, $\M_{2n}[a_3(A_{kM})]$ follow immediately from (\ref{equation: 36FS}) and the Chebotarev Density Theorem (see \cite[Prop. 3.10]{FS14} for a detailed explanation of a similar calculation).
\end{proof}

\begin{remark}
In the statement of the proposition, we included the hypothesis that $a_3(\theta)$ is rational, which is satisfied for Jacobians of twists of the Fermat and Klein curves (see Section \ref{section: arittwist}), because this makes the formulas considerably simpler. This hypothesis is not strictly necessary; one can similarly derive a more general formula without it.
\end{remark}

\begin{corollary}\label{corollary: import} Let $A$ be an abelian variety of dimension $3$ defined over $k$ such that $A_L\sim_L E_L^3$, where $E$ is an elliptic curve defined over $k$ with CM by the quadratic imaginary field $M$, with $k\ne kM$. Suppose that $a_3(\theta)(\tau)$ is rational for $\tau\in G:=\Gal(L/kM)$. Then for $i=1,2,3$, the sequence $a_i(A_{k})$ is equidistributed on $I_i=\left[-\binom{2g}{i},\binom{2g}{i}\right]$  with respect to a measure that is continuous up to a finite number of points and therefore uniquely determined by its moments. For $n\geq 1$, we have $\M_{2n-1}[a_1(A_{k})]=\M_{2n-1}[a_3(A_{k})]=0$ and:
$$
\begin{array}{lll}
\M_{2n}[ a_1(A_{k})] & = &\frac{1}{2|G|}\sum_{\tau\in G}|a_1(\theta)(\tau)|^{2n} \binom{2n}{n}\,,\\[6pt]

\M_n[a_2(A_{k})] & = &\frac{1}{2|G|}\Bigl(\sum_{\tau\in G}\sum_{i=0}^{n}\binom{n}{i}\binom{2i}{i} |a_2(\theta)(\tau)|^i \left(|a_1(\theta)(\tau)|^2-2\cdot |a_2(\theta)(\tau)|\right)^{n-i}\\[5pt]

& & \,\,\,\,\,\,\,\,\,\,\,\,\,+\ \overline o(2)3^n+\overline o(4)(-1)^n+\overline o(8)+\overline o(12)2^ n\Bigr),\\[6pt]

\M_{2n}[a_3(A_{k})] &= & \frac{1}{2|G|}\Bigl(\sum_{\tau\in G}\sum_{i=0}^{n}\binom{2n}{2i}\sum_{j=0}^{2i}\sum_{k=0}^{n-i}\binom{2i}{j}\cdot(r_1(\tau)-3)^{2i-j}\\[5pt]

&& \,\,\,\,\,\,\,\,\,\,\,\,\,\,\cdot\  r_2(\tau)^{2n-2i}\binom{n-i}{k}4^k(-1)^{n-i-k}\binom{2j+2n-2k}{j+n-k}\Bigr)\,.
\end{array}
$$
Here $\overline o(n)$ denotes the number of elements in $\Gal(L/k)$ not in $G$ of order $n$,
and $r_1(\tau)$ and $r_2(\tau)$ are the real and imaginary parts of $a_3(\theta)(\tau)a_2(\theta)(\tau)\overline a_1(\theta)(\tau)$, respectively.
\end{corollary}

\begin{proof} For primes of $k$ that split in $kM$, we invoke Proposition \ref{proposition: import}. Let $N$ be such that $\tau^N= 1$ for every $\tau\in \Gal(L/k)\setminus G$ (the present proof shows \emph{a posteriori} that one can take $N=24$, but for the moment it is enough to know that such an $N$ exists). For primes~$\p$ of $k$ that are inert in $kM$, we will restrict our analysis to those also satisfy:
\begin{enumerate}[(a)] 
\item $\p$ has absolute residue degree $1$, that is, $N(\p)=p$ is prime. 
\item $\p$ is of good reduction for both $A$ and $E$. 
\item $\Q(\sqrt p)\cap \Q(\zeta_{4N})=\Q$, where $\zeta_{4N}$ is a primitive $4N$-th root of unity.
\end{enumerate}
These conditions exclude only a density zero set of primes and thus do not affect the computation of moments.

Now define $D(T,\tau)\coloneqq\det(1-\theta_{\Q}(E,A)(\tau)T)$ for $\tau\in\Gal(L/k)\setminus G$, and let~$\p$ be such that $\Frob_\p=\tau$. In the course of the proof of \cite[Cor.\ 3.12]{FS14} it is shown that:
\begin{enumerate} [(i)]
\item The polynomial $\overline L_\p(A,T)$ divides the Rankin-Selberg polynomial $\overline L_\p(E,\theta_\Q(E,A),T)$.
\item The roots of $D(T,\tau)$ are quotients of roots of $\overline L_\p(A,T)$ and $\overline L_\p(E,T)$.
\end{enumerate}

Since $\overline L_\p(E,T)=1+T^2$ and the roots of $D(T,\tau)$ are $N$-th roots of unity, it follows from (i) that the roots of $\overline L_\p(A,T)$ are $4N$-th roots of unity. In particular $a_1(A)(\p)$, $a_2(A)(\p)$, $a_3(A)(\p) \in \Z[\zeta_{4N}]$. But (a) implies that 
\begin{equation}\label{equation: zerorationality}
\sqrt p \cdot a_1(A)(\p) \in \Z, \qquad p\cdot a_2(A)(\p) \in \Z,\qquad p\sqrt p \cdot a_3(A)(\p) \in \Z\,, 	
\end{equation}
which combined with (c) implies $a_1(A)(\p)=a_3(A)(\p)=0$. This yields the desired moment formulas for $a_1(A_k)$ and $a_3(A_k)$, leaving only $a_2(A_k)$ to consider.

From \eqref{equation: zerorationality} we see that $\overline L(A,T)$ has rational coefficients.  Both $\overline L_\p(E,T)$ and $D(T,\tau)$ have integer coefficients, hence so does $\overline L_\p(E,\theta_\Q(E,A),T)$. Moreover, $\overline L_\p(E,\theta_\Q(E,A),T)$ is also primitive, which by (i) and Gauss' Lemma implies that $\overline L(A,T)$ has integer coefficients.
The Weil bounds then imply (see \cite[Prop.\ 4]{KS08}, for example), that the polynomial $\overline L_{\p}(A,T)$ has the form
\begin{equation}\label{equation: polout}
1+aT^2+aT^4+T^6\eqqcolon P_a(T)\,,
\end{equation}
form some $a\in\{-1,0,1,2,3\}$. To compute the moments of $a_2(A_k)$, it remains only to to determine how often each value of $a$ occurs as $\tau$ ranges over $G$.
We have
\begin{equation}\label{equation: supersingular factors}
\begin{array}{l}
P_{-1}(T)=(1-T)^2(1+T)^2(1+T^2),\\[4pt]
P_{0}(T)=(1+T^2)(1-T^2+T^4),\\[4pt]
P_{1}(T)=(1+T^2)(1+T^4),\\[4pt]
P_{2}(T)=(1-T+T^2)(1+T+T^2)(1+T^2),\\[4pt]
P_{3}(T)=(1+T^2)^3.\\
\end{array}
\end{equation}
 The integer $\ord(\tau)$ is even and then condition (ii) and the specific shape of the $P_a(T)$ imply that the only possible orders of $\tau$ are $2$, $4$, $6$, $8$, $12$. If $\ord(\tau)=2$, then all the roots of $\overline L_\p(E,\theta_\Q(E,A),T)$ are of order $4$. By (i), so are the roots of $P_a(T)$, and (\ref{equation: supersingular factors}) implies that $a=3$.

If $\ord(\tau)=4$, then all the roots of $\overline L_\p(E,\theta_\Q(E,A),T)$ are of order dividing $4$. Thus so are the roots of $P_a(T)$, which leaves the two possibilities $a=-1$ or $a=3$. But (ii) implies that the latter is not possible: if $a=3$, then the roots of $D(T,\tau)$ would all be of order dividing $2$ and this contradicts the fact that $\tau$ has order $4$. Thus $a=-1$.

If $\ord(\tau)=6$, then by (ii) we have $a\not=-1,1,3$ (otherwise the order of $\tau$ would not be divisible by 3). If $a=2$, then again by (ii) the polynomial $D(T,\tau)$ would have a root of order at least 12, which is impossible for $\ord(\tau)=6$. Thus $a=0$.

If $\ord(\tau)=8$, then $\overline L_\p(E,\theta_\Q(E,A),T)$ has at least 8 roots of order 8 and thus $P_a(T)$ has at least a root of order $8$. Thus $a=1$.

If $\ord(\tau)=12$, then  by (ii) we have $a\not=-1,1,3$ (otherwise the order of $\tau$ would not be divisible by 3). If $a=0$, then again by (ii) the polynomial $D(T,\tau)$ would only have roots of orders $1$, $2$, $3$ or $6$, which is incompatible for $\ord(\tau)=12$. Thus $a=2$.
\end{proof}

\section{The Fermat and Klein quartics}\label{section: twists}

The Fermat and the Klein quartics admit models over~$\Q$ given by the equations~$\tilde C^0_1$ and~$\tilde  C^0_7$ of \eqref{equation: Fermatandklein}, respectively.
The Jacobian of $\tilde C^0_1$ is $\Q$-isogenous to the cube of an elliptic curve defined over~$\Q$ (see Proposition~\ref{proposition: decomposition}), but this is not true for $\tilde C^0_7$, which leads us to choose a different model for the Klein quartic.
Let us define $C^0_1:=\tilde C^0_1$ and
$$
C^0_7\colon x^4+y^4+z^4+6(xy^3+yz^3+zx^3)-3(x^2y^2+y^2z^2+z^2x^2)+3xyz(x+y+z)=0\,.
$$
The model $C^0_7$ is taken from \cite[(1.22)]{Elk99}, and its Jacobian is $\Q$-isogenous to the cube of an elliptic curve defined over~$\Q$, as we will prove below.
One can explicitly verify that the curve $C^0_7$ is $\Qbar$-isomorphic to the Klein quartic by using \eqref{equation: autos7} below to show that
$$
\text{ID}(\Aut((C^0_7)_{\Q(\sqrt{-7})}))=\langle 168, 42\rangle.
$$
One similarly verifies that $C_1^0$ is $\Qbar$-isomorphic to the Fermat quartic by using  \eqref{equation: autos1} below to show
$$
\text{ID}(\Aut((C^0_1)_{\Q(i)}))\simeq \langle 96,64\rangle.
$$

Let $E^0_1$ and $E^0 _7$ be the elliptic curves over $\Q$ given by the equations
\begin{equation}\label{eq:E0}
E^0_1\colon y^2z=x^3+xz^2\,,
\qquad
E^0_7\colon y^2z = x^3 - 1715xz^2 + 33614z^3\,,
\end{equation}
with Cremona labels \href{http://www.lmfdb.org/EllipticCurve/Q/64a4}{\texttt{64a4}} and \href{http://www.lmfdb.org/EllipticCurve/Q/49a3}{\texttt{49a3}}, respectively.
We note that $j(E^0_1)=2^6\cdot 3^3$ and $j(E^0_7)=-3^3\cdot 5^3$, thus $E^0_1$ has CM by the ring of integers of $\Q(i)$ and $E^0_7$ has CM by the ring of integers of $\Q(\sqrt{-7})$. For future reference, let us fix some notation. The automorphisms
\begin{equation}\label{equation: autos1}
\begin{cases}
s_1([x:y:z])=[z:x:y]\,,\\
t_1([x:y:z])=[-y:x:z]\,,\\
u_1([x:y:z])=[ix:y:z]
\end{cases}
\end{equation}
generate $\Aut((C^0_1)_\Qbar)$, whereas the automorphisms
\begin{equation}\label{equation: autos7}
\begin{cases}
s_7([x:y:z])=[y : z : x]\,,\\
t_7([x:y:z])=[-3x-6y+2z :-6x+2y-3z : 2x-3y-6z]\,,\\
u_7([x:y:z])=[-2x+ay-z : a x-y+(1-a)z :-x+(1-a)y-(1+a)z]\,,
\end{cases}
\end{equation}
with
$$
a\coloneqq\frac{-1+\sqrt{-7}}{2}=\zeta_7+\zeta_7^2+\zeta_7^4\,,
$$
generate $\Aut((C^0_7)_\Qbar)$.

\begin{proposition}\label{proposition: decomposition} For $d=1$ or $7$, the Jacobian of $C^0_d$ is $\Q$-isogenous to the cube of $E_d^0$.

\end{proposition}

\begin{proof}
For $d=1$, we have a nonconstant map $\varphi_1\colon C^0_1\rightarrow E^0_1$, given by
$$
\varphi_1([x: y: z])=[-x^3zy : x^2 z^3: z x y^3]\,.
$$
Therefore, there exists an abelian surface $B$ defined over $\Q$ such that $\Jac(C^0_1)\sim B\times E^0_1$. Suppose that $E^0_1$ was not a $\Q$-factor of~$B$. Then, the subgroup $\langle s_1,t_1\rangle\subseteq \Aut(C^0_1)$, isomorphic to the symmetric group on~$4$ letters $\sym 4$, would inject into $(\End(\Jac(C^0_1))\otimes \C)^\times$. There are two options for this $\C$-algebra: it is either $\GL_1(\C)^r$ or $\GL_2(\C)\times \GL_1(\C)^ s$, with $r,s\in \Z_{\ge 0}$. In either case we reach a contradiction with the fact that $\sym 4$ has no faithful representations whose irreducible constituents have degrees at most~$2$. Thus $E^0_1$ is a $\Q$-factor of $B$ and $\Jac(C_1^0)\sim (E^0_1)^2\times E$, where $E$ is an elliptic curve defined over $\Q$. Applying the previous argument again shows  $E\sim E^0_1$, so $\Jac(C_1^0)\sim (E^0_1)^3$.

For $d=7$, we have a nonconstant map $\varphi_7\colon C^0_7\rightarrow E^0_7$, given by
$$
\varphi_7([x:y:z])=\left[-7(x+y+z)(3x-y-9z) :
 2^2\cdot7^2 (-x^2-3xy-xz+2z^2) : (x+y+z)^2\right]\,,
$$
thus $\Jac(C^0_7)\sim B\times E^0_7$ for some abelian surface $B$ defined over $\Q$.
Since $\sym 4$ is contained in $\Aut((C^0_7)_M)\simeq \PSL_2(\F_7)$, where $M=\Q(\sqrt{-7})$, we may reproduce the argument above to show that $\Jac(C_7^0)_M\sim (E^0_7)^3_M$. It follows that 
$$
\Jac(C_7^0)\sim E \times E'\times E''\,.
$$
where $E$, $E'$, and $E''$ are either $E^0_7$ or $E^0_7\otimes\chi$, where $\chi$ is the quadratic character of $M$. But $E^0_7\otimes \chi\sim E^0_7$, since $E^0_7$ has CM by~$M$, and the result follows.
\end{proof}

\begin{remark} To simplify notation, for the remainder of this article $d$ is either $1$ or $7$, and we write $C^0$ for $C^0_ d$, $E^0$ for $E^0_d$, $M$ for $\Q(\sqrt{-d})$, and $s$, $t$, $u$ for $s_d$, $t_d$, $u_d$.
When $d$ is not specified it means we are considering both values of $d$ simultaneously.
\end{remark}

\subsection{Twists}\label{section: arittwist}

Let $C$ be a $k$-twist\footnote{When we need not specify the number field $k$ over which $C$ is defined, we will simply say that~$C$ is a twist of~$C^0$. Thus, by saying that~$C$ is a twist of~$C^0$, we do not necessarily mean that~$C$ is defined over~$\Q$.} of $C^0$, a curve defined over $k$ that is $\Qbar$-isomorphic to $C^0$. The set of $k$-twists of $C^0$, up to $k$-isomorphism, is in one-to-one correspondence with $H^1(G_k,\Aut(C^0_M))$. Given an isomorphism $\phi\colon C_\Qbar\overset{\sim}{\to} C^0_\Qbar$, the $1$-cocycle defined by $\xi(\sigma)\coloneqq \phi({}^\sigma\phi)^{-1}$, for $\sigma\in G_k$, is a representative of the cohomology class corresponding to $C$.

Let $K/k$ (resp.\ $L/k$) denote the minimal extension over which all endomorphisms of $\Jac(C)_\Qbar$ (resp.\ all homomorphisms from $\Jac(C)_\Qbar$ to $E^0_\Qbar$) are defined. Let $\tilde K/k$ (resp.\ $\tilde L/k$) denote the minimal extension over which all automorphisms of $C_\Qbar$ (resp.\ all isomorphisms from $C_\Qbar$ to $C^0_\Qbar$) are defined. 

\begin{lemma}\label{lemma: K in L}
We have the following inclusions and equalities of fields:
$$
M\subseteq \tilde K=K \subseteq \tilde L=L\,.
$$
\end{lemma}
\begin{proof}
The inclusion $M\subseteq  \tilde K$ follows from the fact that $\Tr(A_u)\in M\setminus \Q$, where $A_u$ is as in (\ref{equation: emb Fermat}) and (\ref{equation: emb Klein}). From the proof of Proposition \ref{proposition: decomposition}, we know that $\Jac(C)_{\tilde K}\sim E^3$, where $E$ is an elliptic curve defined over $\tilde K$ with CM by $M$. This implies $K=\tilde K M$ and $L=\tilde LM$, as in the proof of \cite[Lem.\,4.2]{FS14}.
\end{proof}

We now associate to $C^0$ a finite group $G^0$ that will play a key role in rest of the article.

\begin{definition}\label{definition :G}
Let $G_{C^0}\coloneqq \Aut(C^0_{M})\rtimes \Gal(M/\Q)$, where $\Gal(M/\Q)$ acts on $\Aut(C^0_{M})$ in the obvious way (coefficient-wise action on rational maps).
It is straightforward to verify that
\begin{equation}\label{equation: twist groups}
\text{ID}(G_{C^0_1})=\langle 192, 956\rangle\,,\qquad \text{ID}(G_{C^0_7})=\langle 336, 208\rangle\,.
\end{equation}
We remark that $G_{C^0_7}\simeq \PGL(\F_7)$.
\end{definition}

As in \cite[\S4.2]{FS14}, we have a monomorphism of groups
$$
\lambda_\phi\colon \Gal(L/k)=\Gal(\tilde L/k)\rightarrow G_{C^0}\,,\qquad\lambda_\phi(\sigma)=(\xi(\sigma),\pi(\sigma))\,,
$$
where $\pi\colon \Gal(L/k)\rightarrow \Gal(M/\Q)$ is the natural projection (which by Lemma~\ref{lemma: K in L} is well defined).
For each $\alpha\in\Aut(C^0_M)$, let $\tilde\alpha$ denote its image by the embedding $\Aut(C^0_M)\hookrightarrow \End((E^0_M)^3)$. The $3$-dimensional representation
$$\theta_{E^0,C^0}\colon\Aut(C^0_M)\rightarrow\Aut_\Qbar( \Hom(E^0_M,\Jac(C^0_M))\otimes_{M,\sigma}\Qbar)\,,$$
defined by $\theta_{E^0,C^0}(\alpha)(\psi)\coloneqq\tilde\alpha\circ\psi$ satisfies
\begin{equation}\label{equation: iso theta}
\theta_{E^0,C^0}\circ\Res^{k}_{kM}\lambda_\phi\simeq \theta_{M,\sigma}(E^0,\Jac(C))\,,
\end{equation}
where $\Res^{k}_{kM}\lambda_\phi$ denotes the restriction of $\lambda_\phi$ from $\Gal(L/k)$ to $\Gal(L/kM)$.

\begin{lemma}\label{lemma: twisting rep} Let $C$ be a twist of $C^0$. Then:
$$
\Tr \theta_{E^0,C^0}=\begin{cases}
\chi_8 &\text{ if $C^0=C^0_1$ (see Table \ref{table: charactertable1}),}\\
\chi_3 &\text{ if $C^0=C^0_7$ (see Table \ref{table: charactertable7}).}
\end{cases}
$$
\end{lemma}

\begin{proof} 
In the proof of Lemma \ref{lemma: embeddings}, we will construct an explicit embedding
$$
\Aut(C^0_M)\hookrightarrow \End((E^0_M)^3)\otimes\Q,\qquad\alpha\mapsto \tilde\alpha\,.
$$
Fix the basis $B=\{\id\times 0\times 0, 0\times\id\times 0,0\times 0\times \id \}$ for $\Hom(E_M^0,\Jac(C_M^0))$. In this basis, with the above embedding the representation $\theta_{E^0,C^0}$ is given by
$$
\theta_{E^0,C^0}(s)=A_s^{-1},\quad \theta_{E^0,C^0}(t)=A_t^{-1},\quad \theta_{E^0,C^0}(u)=A_u^{-1}\,,
$$
where $A_s$, $A_t$, and $A_u$ are as in (\ref{equation: emb Fermat}) and (\ref{equation: emb Klein}). The lemma follows.
\end{proof}

\begin{remark}
Observe that since $\det(\theta_{E^0,C^0})$ is a rational character of $\Aut(C^0_M)$, by \eqref{equation: iso theta} so is $a_3(\theta)=\det\theta_{M,\sigma}(E^0,\Jac(C))$. Thus Corollary \ref{corollary: import} can be used to compute the moments of $a_i(\Jac(C))$ for $i=1,2,3$.
\end{remark}

\begin{proposition}\label{proposition: K=L}
The fields $K$ and $L$ coincide.\footnote{Note that this does not hold for the hyperelliptic curves considered in \cite{FS14} where $[L\colon\! K]$ may be $1$ or $2$.}
\end{proposition}

\begin{proof}
Note that $L/K$ is the minimal extension over which an isomorphism between $E^0_\Qbar$ and $E_\Qbar$ is defined. 
It follows that $L=K(\gamma^{\nicefrac{1}{n}})$ for some $\gamma\in K$, with $n=4$ for $d=1$ and $n=2$ for $d=7$; see \cite[Prop.\,X.5.4]{Sil09}.
In either case, $\Gal(L/K)$ is cyclic of order dividing $4$ (note that $\Q(\zeta_n)\subseteq K$). Suppose that $L\not=K$, let $\omega$ denote the element in $\Gal(L/K)$ of order $2$, and write $K^0=L^{\langle\omega\rangle}$.
Fix an isomorphism $\psi_1\colon E^0_L\rightarrow  E_L$ and an isogeny $\psi_2\colon (E_{K^0})^3 \rightarrow \Jac(C)_{K^0}$.
For $i=1,2,3$, let $\iota_i\colon E_{K^0}\rightarrow (E_{K^0})^3$ denote the natural injection to the $i$th factor. Then $\{\psi_2\circ\iota_i\circ\psi_1\}_{i=1,2,3}$  constitute a basis of the $\Qbar[\Gal(L/M)]$-module $\Hom(E^0_L,\Jac(C)_L)\otimes_{M,\sigma}\Qbar$.
Since ${}^{\omega}\psi_1=-\psi_1$, ${}^{\omega}\psi_2=\psi_2$, and ${}^{\omega}\iota_i=\iota_i$, we have $\Trace\theta_{M,\sigma}(E^0,\Jac(C))(\omega)=-3$.
But this contradicts \eqref{equation: iso theta}, because there is no $\alpha$ in $\Aut(C^0_M)$ for which $\Trace \theta_{E^0,C^0}(\alpha)=-3$.
\end{proof}

\begin{remark}
By Proposition~\ref{proposition: K=L}, and the identities \eqref{equation: icoeffs} and \eqref{equation: iso theta}, the independent and joint coefficient measures of $\Jac(C)$ depend only on the conjugacy class of $\lambda_\phi(\Gal(K/k))$ in $G_{C^0}$. In Proposition~\ref{proposition: conjSTgruops1}, we will see that this also applies to the Sato--Tate group of $\Jac(C)$. For this reason, henceforth, subgroups $H\subseteq G_{C^0}$ will be considered only up to conjugacy.
\end{remark}

\begin{definition}\label{definiton: H0}
Let $G_0\coloneqq \Aut(C_M^0)\times\langle 1\rangle\subseteq G_{C^0}$, and for subgroups $H\subseteq G_{C^0}$, let $H_0:=H \cap \G_0$.
We may view $H_0$ as a subgroup of $\Aut(C_M^0)\simeq G_0$ whenever it is convenient to do so.
\end{definition}

\noindent
Noting that $[G_{C^0}:G_0]=2$, for any subgroup $H$ of $G_{C^0}$ there are two possibilities:
\begin{enumerate}
\item[(c$_1$)] $H\subseteq G_0$, in which case $[H:H_0]=1$;
\item[(c$_2$)] $H\not\subseteq G_0$, in which case $[H:H_0]=2$.
\end{enumerate}

\begin{remark}\label{remark: pairs}
We make the following observations regarding $H\subseteq G_{C^0}$ and cases (c$_1$) and (c$_2$):
\begin{enumerate}[(i)]
\item In \S\ref{section: curves}, we will show that for each subgroup $H\subseteq G_{C^0}$, there is a twist $C$ of $C^0$ such that $H=\lambda_\phi(\Gal(K/k))$. From the definition of $\lambda_\phi$, we must then have $H_0=\lambda_\phi(\Gal(K/kM))$. The case (c$_1$) corresponds to $kM = k$, and the case (c$_2$) corresponds to $k\not= kM$. 
\item There are $83$ subgroups $H \subseteq G_{C^0_1}$ up to conjugacy, of which $24$ correspond to case (c$_1$) and~$59$ correspond to case (c$_2$). On Table~\ref{table: example curves Fermat} we list the subgroups in case (c$_2$). For any subgroup $H$ in case (c$_1$), there exists a subgroup $H'$ in case (c$_2$) such that $H'_0=H$; thus the subgroups in case (c$_1$) can be recovered from Table~\ref{table: example curves Fermat} by looking at the column for~$H_0$.
\item There are $23$ subgroups $H \subseteq G_{C^0_7}$ up to conjugacy, of which $12$ correspond to case (c$_1$) and~ $11$ correspond to case (c$_2$).  For all but 3 \emph{exceptional} subgroups $H$ in case (c$_1$) there exists a subgroup $H'$ in case (c$_2$) such that $H'_0=H$. In Table~\ref{table: example curves Klein} we list the subgroups in case (c$_2$) as well as the 3 exceptional subgroups, which appear in rows \#3, \#8, and \#12 of Table~\ref{table: example curves Klein}.  As in (ii) above, the non-exceptional subgroups in case (c$_1$) can be recovered from Table \ref{table: example curves Klein} by looking at the column for $H_0$, which for the exceptional groups is equal to $H$.
\item The subgroups $H\subseteq G_{C^0}$ in Tables~\ref{table: example curves Fermat} and~\ref{table: example curves Klein} are presented as follows. First, generators of $H_0\subseteq G_0\simeq \Aut(C^0_M)$ are given in terms of the generators $s,t,u$  for $\Aut(C^0_M)$ listed in \eqref{equation: autos1} and \eqref{equation: autos7}. For the 3 exceptional subgroups of Table~\ref{table: example curves Klein} we necessarily have $H=H_0$, and for the others, $H$ is identified by listing an element $h\in \Aut(C^0_M)$ such that
\begin{equation}\label{equation: descH}
H=H_0\cup H_0\cdot (h,\tau) \subseteq G_{C^0}\,,
\end{equation}
where $\tau$ is the generator of $\Gal(M/\Q)$. 
\end{enumerate}
\end{remark}

\subsection{Moment sequences}\label{section: momsequence}

We continue with the notation of \S\ref{section: arittwist}. If $C$ is a $k$-twist of $C^0$, we define the joint and independent coefficient moment sequences
$$
\M_{\rm joint}(C)\coloneqq \{\M_{n_1,n_2,n_3}[a(C)]\}_{n_1,n_2,n_3}\,,\qquad \M_{\rm indep}(C)\coloneqq \{\M_n[a_i(C)]\}_{j,n}\,,
$$
where $a(C)$ and $a_j(C)$ denote $a(\Jac(C))$ and $a_j(\Jac(C))$, respectively, as defined in \S\ref{section: introduction};
recall that these moment sequences are defined by and uniquely determine corresponding measures $\mu_I$ and~$\mu_{I_j}$, respectively.

Using Lemma \ref{lemma: twisting rep} and (\ref{equation: iso theta}), we can apply Corollary~\ref{corollary: import} to compute the moments $\M_n[a_j(C)]$ for any $n$, and as explained in \S\ref{sec: joint moments}, it is easy to compute $M_{n_1,n_2,n_3}[a(C)]$ for any particular values of $n_1,n_2,n_3$.  Magma scripts \cite{Magma} to perform these computations are available at \cite{FLS17}, which we note depend only on the pairs $(H,H_0)$ (or just $H_0$ when $k=kM$) listed in Tables~\ref{table: example curves Fermat} and~\ref{table: example curves Klein}, and are otherwise independent of the choice of $C$.

\subsubsection{\emph{\textbf{Independent coefficient moment sequences}}}
We now show that for any twist of the Fermat or Klein quartic, each of the independent coefficient moment sequences (and hence the corresponding measures) is determined by the first several moments.

\begin{proposition}\label{proposition: bounds}
Let $C$ and $C'$ be $k$-twists of $C^0$. For each $i=1,2,3$ there exists a positive integer $N_i$ such that if
$$
\M_{n}[a_i(C)]=\M_{n}[a_i(C')] \quad \text{for } 1 \leq n \leq N_i\\
$$
then in fact
$$
\M_{n}[a_i(C)]=\M_{n}[a_i(C')] \quad \text{for all } n \geq 1.\\
$$
Moreover, one can take $N_1=6$, $N_2=6$, $N_3=10$.
\end{proposition}

\begin{proof}
For the sake of brevity we assume $k\not =kM$ (the case $k=kM$ is analogous and easier). It follows from Corollary~\ref{corollary: import} that, for $i=1,2,3$, the sequence $\{\M_{n}[a_i(C)]\}_{n\geq 0}$ is determined by $|a_1(C)|$, $|a_2(C)|$, $a_3(C) a_2(C)\bar a_1(C)$, and $\bar o(2)$, $\bar o(4)$, $\bar o(8)$ (note that $G_{C_1^0}$ and $G_{C_7^0}$ contain no elements of order 12, so we ignore the $\bar o(12)$ term in the formula for $\M_n[a_2(C)]$). We consider the Fermat and Klein cases separately.

For the Fermat case, with the notation for conjugacy classes as in Table~\ref{table: charactertable1}, let $x_1$ (resp.\ $x_2$, $x_3$, $x_4$, $x_5$) denote the proportion of elements in $\Gal(L/k)$ lying in the conjugacy class $1a$ (resp.\ $2a \cup 2b \cup 4c \cup 4d$, $3a$, $4a \cup 4b$, $8a\cup 8b$); note that by Lemma~\ref{lemma: twisting rep}, we are interested in the representation with character $\chi_8$ listed in Table~\ref{table: charactertable1}, which motivates this partitioning of conjugacy classes.

Let $y_1$ (resp.\ $y_2$, $y_3$) denote the proportion of elements in $\Gal(L/k)$ which do not lie in $\Gal(L/kM)$ and have order $2$ (resp.\ $4$, $8$).  Applying Corollary~\ref{corollary: import}, one finds that for $n\ge 1$ we have
\begin{equation}\label{equation: linear syst}
\M_{2n}[a_1(C)]=x_1\cdot 9^n\binom{2n}{n}+(x_2+x_5)\binom{2n}{n}+x_4(-1)^n\binom{2n}{n}\,.
\end{equation}
Evaluating \eqref{equation: linear syst} at $n=1,2,3$ yields an invertible linear system in $x_1$, $x_2+x_5$, $x_4$ of dimension 3.
The moments $\M_{2n}[a_1(C)]$ for $n=1,2,3$ thus determine $x_1$, $x_2+x_5$, $x_4$, and therefore determine all the $\M_{2n}[a_1(C)]$.

For $\M_n[a_2(C)]$ one similarly obtains an invertible linear system in $x_1$, $x_2+x_5$, $x_4$, $y_1$, $y_2$, $y_3$ of dimension 6, and it follows that the moments $\M_n[a_2(C)]$ for $n\le 6$ determine all the $\M_{n}[a_2(C)]$.

For $\M_{2n}[a_3(C)]$ one obtains an invertible linear system in $x_1$, $x_2$, $x_3$, $x_4$, $x_5$ of dimension 5, and it follows that the moments $\M_{2n}[a_2(C)]$ for $n\le 5$ determine all the $\M_{2n}[a_3(C)]$.

In the Klein case one proceeds analogously. With the notation of Table~\ref{table: charactertable7}, let $x_1$ (resp.\ $x_2$, $x_3$, $x_4$) denote the proportion of elements in $\Gal(L/k)$ lying in the conjugacy class $1a$ (resp.\ $2a \cup 4a$, $3a$, $7a\cup 7b$), and let $y_1$, $y_2$, $y_3$ be as in the Fermat case. Now $x_1$, $x_2$, $x_4$ determine $\M_n[a_1(C)]$; $x_1$, $x_2$, $x_4$ and $y_1$, $y_2$, $y_3$ determine $\M_{n}[a_2(C)]$; and $x_1$, $x_2$, $x_3$, $x_4$ determine $\M_{n}[a_3(C)]$. These proportions are, as before, determined by the first several moments (never more than are needed in the Fermat case), and the result follows. We spare the reader the lengthy details.
\end{proof}

With Proposition \ref{proposition: bounds} in hand we can completely determine the moment sequences $\M_n[a_i(C)]$ that arise among $k$-twists $C$ of $C^0$ by computing the moments $\M_n[a_i(C)]$ for $n\le N_i$ for the 59 pairs $(H,H_0)$ listed in Table~\ref{table: example curves Fermat} in the case $C^0=C^0_1$, and for the 14 pairs $(H,H_0)$ listed in Table~\ref{table: example curves Klein} (as described in Remark \ref{remark: pairs}).  Note that each pair $(H,H_0)$ with $H\ne H_0$ gives rise to two moments sequences $\M_n[a_i(C)]$ for each $i$, one with $k\ne kM$ and one with $k=kM$.

After doing so, one finds that in fact Proposition \ref{proposition: bounds} remains true with $N_2=4$ and $N_3=4$.  The value of $N_1$ cannot be improved, but one also finds that the sequences $\M_n[a_2(C)]$ and $\M_n[a_3(C)]$ together determine the sequence $\M_n[a_1(C)]$, and in fact just two well chosen moments suffice.

\begin{corollary}\label{corollary: det mom seq}
There are $48$ (resp.\ $22$) independent coefficient measures among twists of $C^0_1$ (resp.\ $C^0_7$).
In total there are $54$ independent coefficient measures among twists $C$ of either $C_1^0$ or $C_7^0$, each of which is uniquely distinguished by the moments $\M_3[a_2(C)]$ and $\M_4[a_3(C)]$. 
\end{corollary}

The moments $\M_3[a_2(C)]$ and $\M_4[a_3(C)]$ correspond to the joint moments $\M_{0,3,0}[a(C)]$ and $\M_{0,0,4}[a(C)]$ whose values are listed in Table~\ref{table: distributions} for each of the 60 distinct joint coefficient moment measures obtained in the next section (this includes all the independent coefficient measures).

\subsubsection{\emph{\textbf{Joint coefficient moments}}}\label{sec: joint moments}

Instead of giving closed formulas for $\M_{n_1,n_2,n_3}[a(C)]$, analogous to those derived for $\M_n[a_j(C)]$ in Corollary~\ref{corollary: import}, let us explain how to compute $\M_{n_1,n_2,n_3}[a(C)]$ for specific values of $n_1,n_2,n_3$ (this will suffice for our purposes). Suppose $k\not = kM$ (the other case is similar). By \eqref{equation: icoeffs}, we may naturally regard the quantity
\begin{equation}\label{equation: productterm}
a_1(C)^{n_1}a_2(C)^{n_2}a_3(C)^{n_3}
\end{equation}
as an element of the formal polynomial ring $\Qbar[\alpha_1,\bar\alpha_1]/(\alpha_1 \bar\alpha_1-1)$. The moment $\M_{n_1,n_2,n_3}[a(C)]$ is simply the constant term of \eqref{equation: productterm}. For $N\geq 0$ and each pair $(H,H_0)$ on Tables~\ref{table: example curves Fermat} and~\ref{table: example curves Klein}, one can then compute truncated joint moment sequences
\[
\M_{\rm joint}^{\le N}(C)\coloneqq \{\M_{n_1,n_2,n_3}[a(C)]:n_1,n_2,n_3\ge 0,\ n_1+n_2+n_3\leq N\}.
\]

By explicitly computing $\M_{\rm joint}^{\le 4}(C)$ for all the pairs $(H,H_0)$ listed in Tables~\ref{table: example curves Fermat} and~\ref{table: example curves Klein}, we obtain the following proposition.

\begin{proposition}\label{proposition: lowernumb}
There are at least $54$ (resp.\ $23$) joint coefficient measures (and hence Sato--Tate groups) of twists of the Fermat (resp.\ Klein) quartic, and at least $60$ in total.
These $60$ joint coefficient measures are listed in Table~\ref{table: distributions}, in which each is uniquely distinguished by the three moments $\M_{1,0,1}[a(C)]$, $\M_{0,3,0}[a(C)]$, and $\M_{2,0,2}[a(C)]$.
\end{proposition}

Computing $\M_{\rm joint}^{\le N}(C)$ with $N=5,6,7,8,$ does not increase any of the lower bounds in Proposition~\ref{proposition: lowernumb}, leading one to believe they are tight.
We will prove this in the next section, but an affirmative answer to the following question would make it easy to directly verify such a claim.

\begin{question}\label{question: Ng bound}
Recall the setting of the first paragraph of \S\ref{section: introduction}. In particular,~$A$ is an abelian variety defined over a number field $k$ of dimension $g\geq 1$, and $\mu_I$ is the measure induced on~$I$. By \cite[Prop. 3.2]{FKRS12} and \cite[Rem. 3.3]{FKRS12}, one expects a finite list of possibilities for the Sato--Tate group of $A$. One thus expects a finite number of possibilities for the sequence $\{\M_{n_1,\dots,n_g}[\mu_I]\}_{n_1,\dots,n_g}$. In particular, one expects that there exists $N_g\geq 1$, depending only on $g$, such that for any abelian variety $A'$ defined over a number field $k'$, if
\begin{equation}\label{equation: momupto}
\{\M_{n_1,\dots,n_g}[\mu_I]\}_{n_1,\dots,n_g}=\{\M_{n_1,\dots,n_g}[\mu_I']\}_{n_1,\dots,n_g}
\end{equation}
for all $n_1+\dots+n_g\leq N_g$, then \eqref{equation: momupto} holds for all $n_1,\dots,n_g\geq 0$. Is there an explicit and effectively computable upper bound for $N_g$? 
\end{question}

Lacking an answer to Question,~\ref{question: Ng bound}, in order to determine the exact number of distinct joint coefficient measures, we take a different approach. In the next section, we will classify the possible Sato--Tate groups of twists of the Fermat and Klein quartics.
This classification yields an upper bound that coincides with the lower bound of Proposition~\ref{proposition: lowernumb}.

Tables \ref{table: distributions} also lists $z_1=[z_{1,0}]$, $z_2=[z_{2,-1},z_{2,0},z_{2,1},z_{2,2},z_{2,3}]$, and $z_3=[z_{3,0}]$, where, $z_{i,j}$ denotes the density of the set of primes $\p$ for which $a_i(C)(\p)=j$.  For these we record the following lemma.

\begin{lemma}\label{lemma: zvalues}
Let $C$ be a $k$-twist of $C^0$ and let $H_0:=\lambda_\phi(\Gal(K/kM))$. Then
$$
z_1(\Jac(C))=\begin{cases}
\frac{o(3)}{|H_0|} & \text{ if $M\subseteq k$}, \\[4pt]
\frac{1}{2}+\frac{o(3)}{2|H_0|} & \text{ if $M\not\subseteq k$},
\end{cases}
\qquad 
z_3(\Jac(C))=\begin{cases}
0 & \text{ if $M\subseteq k$}, \\[4pt]
\frac{1}{2} & \text{ if $M\not\subseteq k$}, 
\end{cases}
$$
$$
z_2(\Jac(C))=\begin{cases}
\frac{1}{|H_0|}[0,o(3),0,0,0] & \text{ if $M\subseteq k$}, \\[4pt]
\frac{1}{2|H_0|}[\bar o(4),o(3)+\bar o(6),\bar o(8),0,\bar o(2)] & \text{ if $M\not\subseteq k$}.
\end{cases}
$$
Here $\bar o (n)$ is as in Corollary~\ref{corollary: import} and $o(n)$ denotes the number of elements of order $n$ in $H_0$.
\end{lemma}

\begin{proof}
The formula for $z_1$ is immediate from \eqref{equation: icoeffs} and the study of the polynomial \eqref{equation: polout} in the proof of Corollary~\ref{corollary: import}, together from the fact that $\tau \in \Gal(L/k)$ satisfies $a_1(\theta)(\tau)=0$ if and only if $\tau$ has order $3$ (as can be seen from Table \ref{table: charactertable}).

The formula for $z_2$ follows from a similar reasoning, once one observes that again $a_2(\theta)(\tau)=0$ if and only if $\tau$ has order $3$, and the discussion of the end of Corollary~\ref{corollary: import}. Note also that, as $G_{C^0}$ contains no elements of order~$12$, we have $\bar o(12)=0$.

For $z_3$, it suffices to note that $a_3(\tau)$ does not vanish.  
\end{proof}

\subsection{Sato-Tate groups}\label{section: ST groups}

In this section, for any twist $C$ of $C^0$, we explicitly construct $\ST(\Jac(C))$, which to simplify notation we denote by $\ST(C)$. The first step is to compute a (non-canonical) embedding 
$$
\iota\colon \Aut(C^0_M)\rightarrow\USp(6)
$$ 
(see \cite{MT10} for a very similar approach). Let $\Omega^1(E^0_M)$ (resp.\ $\Omega^1(C^0_M)$) denote the $M$-vector space of regular differentials of $E^0_M$ (resp.\ $C^0_M$). Define
$$
\iota_1\colon\Aut(C^0_{M})\rightarrow \Aut(\Omega^1(C^0_{M})),\qquad \iota_1(\alpha)=(\alpha^*)^{-1}\,,
$$
where $\alpha^* \colon \Omega^1(C^0_{M}) \rightarrow \Omega^1(C^0_{M})$ is the map induced by $\alpha$.

\begin{remark}\label{remark: reg dif}
Let $f(X,Y)=0$ be an affine model of the plane quartic $C^0_M$. Then
\begin{equation}\label{equation: base1}
\left\{W_1:=X\frac{dX}{f_Y},W_2:=Y\frac{dX}{f_Y},W_3:=\frac{dX}{f_Y}\right\}\,,
\end{equation}
with $f_Y=\frac{\partial f}{\partial Y}$, is a basis of the regular differentials $\Omega^1(C^0_ M)$.
If we denote by $\omega_{i}$ the regular differential of the $i$th copy of $E^0_M$ in $(E^0_M)^3$, then
\begin{equation}\label{equation: base2}
\{\omega_1,\,\omega_2,\,\omega_3 \} 
\end{equation}
is a basis of of the regular differentials $\Omega^1(E^0_M)^3$.
\end{remark}
Consider the isomorphism
$$
\iota_2\colon\End(\Omega^1(C^0_M))\rightarrow\End(\Omega^1(E^0_M)^3)
$$
induced by the isomorphism $\Omega^1(C^0_M) \simeq \Omega^1(E^0_M)^3$ that sends $W_i$ to $\omega_i$.

Fix an isomorphism $[\,]\colon M\rightarrow\End(E^0_M)\otimes\Q$ such that for any regular differential $\omega\in\Omega^1(E^0_M)$, one has $[m]^*(\omega)=m\omega$ for every $m\in M$ (see \cite[Chap. II, Prop. 1.1]{Sil94}) and then define
$$
\iota_3\colon\End(\Omega^1(E^0_M)^3)\rightarrow\End((E^0_M)^3)\otimes \Q,\qquad \iota_3((m_{jk}))=([m_{jk}])\,,
$$
where $m_{ij}\in M$. Let $f_1=i$ and let $f_7=a$. For $d=1,7$, let $\gamma_d\in H_1((E^0_d)^{\topo}_{\C}, \Q)$ be such that $\{\gamma_d, [f_d]_*\gamma_d\}$ is a symplectic basis of $H_1((E^0_d)^{\topo}_{\C}, \Q)$ with respect to the cup product, and use this basis to obtain an isomorphism
$$
\Theta_d\colon\End(H_1((E^0_d)^{\topo}_{\C}, \Q))\rightarrow \GSp_2(\Q)\,.
$$
Then define
$$
\iota_4\colon\End((E^0_d)^3_M)\rightarrow\GSp_6(\Q)\,,\qquad ([m_{jk}])\rightarrow (\Theta_d([m_{jk}]_*))\,.
$$
Finally, define the matrices
$$
I_2=\begin{pmatrix}
1 & 0\\ 0&1
\end{pmatrix}
\,,\qquad
J_2=\begin{pmatrix}
0 & -1\\1 & 0
\end{pmatrix}\,,\qquad
K_2=\begin{pmatrix}
0 & -2\\1 & -1
\end{pmatrix}\,.
$$

\begin{remark}
From now on, we fix the following notation: denote by $A_3$ the 3-diagonal embedding of a subset $A$ of $\GL_2$ in $\GL_6$. Throughout this section we consider the general symplectic group $\GSp_6/\Q$, the symplectic group $\Sp_6/\Q$, and the unitary symplectic group $\USp(6)$ with respect to the symplectic form given by the $3$-diagonal embedding $(J_2)_3$.
\end{remark}

\begin{lemma}\label{lemma: embeddings}
The map
$$
\iota\colon\Aut(C^0_M)\hookrightarrow^{\iota_1} \End(\Omega^1(C^0_M))\simeq^{\iota_2}\End(\Omega^1(E^0_M)^3 )\simeq^{\iota_3}\End((E^0_M)^3)\otimes \Q\hookrightarrow^{\iota_4}\GSp_6(\Q)
$$
is a monomorphism of groups that for $C^0=C^0_1$ is explicitly given by
\small
$$
\iota(s_1)=\begin{pmatrix}
0 & 0 & I_2\\
I_2 & 0 & 0\\
0 & I_2 & 0
\end{pmatrix}
\,,\quad
\iota(t_1)=\begin{pmatrix}
0 & -I_2 & 0\\
I_2 & 0 & 0\\
0 & 0 & I_2
\end{pmatrix}
\,,\quad
\iota(u_1)=\begin{pmatrix}
-I_2 & 0 & 0\\
0 & -J_2 & 0\\
0 & 0 & -J_2
\end{pmatrix}
\,,
$$
\normalsize
and for $C^0=C^0_7$ is explicitly given by $\iota(s_7)=\iota(s_1)^{\T}$ and
\small
$$
\iota(t_7)=\frac{1}{7}\begin{pmatrix}
-3I_2 & -6I_2 & 2I_2\\
-6I_2 & 2I_2 & -3I_2\\
2I_2 & -3I_2 & -6I_2
\end{pmatrix}
\,,\quad
\iota(u_7)=\frac{1}{7}\begin{pmatrix}
-2I_2-4K_2 & 3I_2-K_2 & -I_2-2K_2\\
3I_2-K_2 & -I_2-2K_2 & -2I_2+3K_2\\
-I_2-2K_2 & -2I_2+3K_2 & -4I_2-K_2
\end{pmatrix}
\,.
$$
\normalsize
\end{lemma}

\begin{proof}
We first consider the case $C^0=C^0_1$. In the basis of \eqref{equation: base1}, the elements $s_1^*$, $t_1^*$, and $u_1^*$ of $\End(\Omega^1(C^0_M))$ are given by the matrices
\begin{equation}\label{equation: emb Fermat}
A_{s_1}=\begin{pmatrix}
0 & 1 & 0\\
0 & 0 &1\\
1 & 0 & 0
\end{pmatrix}
\,,\quad
A_{t_1}=\begin{pmatrix}
0 & 1 & 0\\
-1 & 0 & 0\\
0 & 0 & 1
\end{pmatrix}
\,,\quad
A_{u_1}=\begin{pmatrix}
-1 & 0 & 0\\
0 & i & 0\\
0 & 0 & i
\end{pmatrix}
\,.
\end{equation}
Thus, in the basis of \eqref{equation: base2}, the elements $\iota_2\iota_1(s_1)$, $\iota_2\iota_1(t_1)$, $\iota_2\iota_1(u_1)$ of $\End(\Omega^1(E^0_M)^3)$ are given by the matrices $A_{s_1}^{-1}$, $A_{t_1}^{-1}$, $A_{u_1}^{-1}$.
It is then enough to check that in the basis $\{\gamma_1,[i]_*\gamma_1\}$ we have $\Theta_1(1)=I_2$ and $\Theta_1(i)=J_2$.

We now assume $C^0=C^0_7$. In the basis of \eqref{equation: base1}, the matrices associated to $s_7^*$, $t_7^*$, $u_7^*$ are:
\begin{equation}\label{equation: emb Klein}
A_{s_7}=\begin{pmatrix}
0 & 0 & 1\\
1 & 0 & 0\\
0 & 1 & 0
\end{pmatrix},
\quad
A_{t_7}=\frac{1}{7}\begin{pmatrix}
-3 & -6 & 2\\
-6 & 2 & -3\\
2 & -3 &-6
\end{pmatrix},
\quad
A_{u_7}=\frac{1}{7}\begin{pmatrix}
2+4a & 4+a & 1+2a\\
4+a & 1+2a & -5-3a\\
1+2a & -5-3a & -3+a
\end{pmatrix}.
\end{equation}
Thus, in the basis of \eqref{equation: base1}, the elements $\iota_2\iota_1(s_7)$, $\iota_2\iota_1(t_7)$, $\iota_2\iota_1(u_7)$ of $\End(\Omega^1_M(E^0_7)^3)$ are given by the matrices $A_{s_7}^{-1}$, $A_{t_7}^{-1}$, $A_{u_7}^{-1}$. It is then enough to check that in the basis $\{\gamma_7,[a]_*\gamma_7\}$ we have $\Theta_7(1)=I_2$ and $\Theta_7(a)=K_2$. But this is clear, since
$$
\begin{array}{l}
[a]_*(\gamma_7)=0\cdot\gamma_7+1\cdot([a]_*\gamma_7)\,,\\[6pt]
[a]_*([a]_*\gamma_7)=[a^2]_*\gamma_7=[-a-2]_*\gamma_7=-2\cdot\gamma_7-1\cdot([a]_*\gamma_7)\,,
\end{array}
$$
and this completes the proof.
\end{proof}

\begin{remark}
Note that since $\Aut(C^0_M)$ has finite order, the image of $\iota$ is contained in $\USp_6(\Q)$.
\end{remark}

\begin{remark}
It is easy to check that the matrices $\iota(s_1)$, $\iota(t_1)$, $\iota(u_1)$ (resp.\ $\iota(s_7)$, $\iota(t_7)$, $\iota(u_7)$) are symplectic with respect to $J:=(J_2)_3$.
\end{remark}

The following theorem gives an explicit description of the Sato-Tate group of a twist of the Fermat or Klein quartic corresponding to a subgroup $H$ of the group
\[
G_{C_0}\coloneqq \Aut(C^0_M)\times \Gal(M/\Q)
\]
associated to $C^0$ (see Definition~\ref{definition :G}).

\begin{theorem}\label{theorem: ST groups}
The following hold:
\begin{enumerate}[{\rm (i)}]
\item The monomorphism of Lemma \ref{lemma: embeddings} extends to a monomorphism
$$
\iota\colon G_{C^0}\hookrightarrow \USp(6)/\langle -1\rangle\,
$$
by defining
$$
\iota((1,\tau))\coloneqq
\begin{cases}
\frac{1}{\sqrt 2}
\begin{pmatrix}
i & i\\i & -i
\end{pmatrix}_3
& \text{ if } C^0=C^0_1,\\[6pt]

\begin{pmatrix}
i & -i\\0 & -i
\end{pmatrix}_3
 & \text{ if } C^0=C^0_7,
\end{cases}
$$
where $\tau$ denotes the non-trivial element of $\Gal(M/\Q)$.  
\item Let $\phi\colon C_\Qbar\rightarrow C^ 0_\Qbar$ denote a $k$-twist of $C^0$ and write $H:=\lambda_\phi(\Gal(K/k))\subseteq G_{C^0}$. The Sato-Tate group of $\Jac(C)$ is given by
\begin{align*}
\ST(C)=\ST(E^0,1)_3\cdot\iota(H)\,,
\end{align*}
where
\begin{align*}
\ST(E^0,1)=&
\left\{
\begin{pmatrix}
\cos(2\pi r) & \sin(2\pi r)\\-\sin(2\pi r) & \cos(2\pi r)
\end{pmatrix}\,|\, r\in [0,1]
\right\}
\,
\end{align*}
if $C^0=C^0_1$, and
\begin{align*}
\ST(E^0,1)=
\left\{
\begin{pmatrix}
\cos(2\pi r)-\frac{1}{\sqrt 7}\sin(2\pi r) & \frac{4}{\sqrt 7}\sin(2\pi r) \\-\frac{2}{\sqrt 7} \sin(2\pi r) &  \cos(2\pi r)+\frac{1}{\sqrt 7}\sin(2\pi r)
\end{pmatrix}\,|\,r\in[0,1]
\right\}
\end{align*}
if $C^0=C^0_7$.
\end{enumerate}
\end{theorem}

\begin{proof}
To prove (i) it is enough to note that $\iota((1,\tau))^2=1$ in $\USp(6)/\langle -1\rangle$ and that $\iota((1,\tau))$ acts by matrix conjugation on $\iota(\Aut(C^0_M))$ as $\tau$ acts by Galois conjugation on $\Aut(C^0_M)$. If $C^0= C^0_1$ (resp.\ $C^0=C^0_7$), the latter is equivalent to
$$
\iota((1,\tau))^{-1} J_2\iota((1,\tau))= -J_2\,,\qquad \big(\, \text{resp.\ }\iota((1,\tau))^{-1} K_2\iota((1,\tau))= -I_2-K_2 \,\big)\,,
$$
which is straightforward to check.

For (ii) we consider only the case $kM\ne k$, since the case $k=kM$ can be easily deduced from the case $kM\ne k$.
Recall from \cite{BK15} that $\ST(E^0)$ is a maximal compact subgroup of the algebraic Sato--Tate group $\AST(E^0)\otimes \C$ attached to $E^0$. Recall that $\AST(E^0)=\Lef(E^0,1)\cup\Lef(E^0,\tau)$, where for $\sigma \in\Gal(kM/k)$ one has
\begin{equation}\label{equation: Lef}
\Lef(E^0,\sigma) \coloneqq \{\gamma \in \Sp_{2}\,|\,\gamma^{-1} \alpha \gamma = {}^\sigma\alpha \mbox{ for all $\alpha \in \End(E^0_\Qbar)\otimes{\Q}$}\}.
\end{equation}
This induces a decomposition $\ST(E^0)=\ST(E^0,1)\cup\ST(E^0,\tau)$ that can be explicitly determined.

For the case $C^0=C^0_1$,  we have
\begin{align*}
\Lef(E^0,1)(\C) &=\{A\in M_2(\C)|A^\T J_2A=J_2,\, A^{-1}J_2A=J_2 \} \\
& =\left\{
\begin{pmatrix}
c & b \\-b & c
\end{pmatrix}\,|\,c,b\in \C,\,c^2+b^2=1
\right\}\,.
\end{align*}
Thus, a maximal compact subgroup of $\Lef(E^0,1)(\C)$ is
$$
\ST(E^0,1) =\left\{
\begin{pmatrix}
\cos(2\pi r) & \sin(2\pi r) \\- \sin(2\pi r) &  \cos(2\pi r)
\end{pmatrix}\,|\,r\in[0,1]
\right\}\,.
$$
Analogously,
\begin{align*}
\Lef(E^0,\tau)(\C) &=\{A\in M_2(\C)|A^\T J_2A=J_2,\, A^{-1}J_2A=-J_2 \} \\
& =\left\{
\begin{pmatrix}
ic & ib \\ib &-i c
\end{pmatrix}\,|\,c,b\in\C,\,c^2+b^2=1
\right\}\,.
\end{align*}
Thus, a maximal compact subgroup of $\Lef(E^0,\tau)(\C)$ is
\begin{equation}\label{equation: keysplit}
\ST(E^0,\tau) =\ST(E^0,1)\cdot\frac{1}{\sqrt 2}
\begin{pmatrix}
i & i \\i &-i
\end{pmatrix}
\,.
\end{equation}
There is a relation between the algebraic Sato--Tate groups $\AST(C)$ and $\AST(C^0)$ attached to $\Jac(C)$ and $\Jac(C^0)$, respectively, given by \cite[Lemma 2.3]{FS14}. If we put $H_0\coloneqq\lambda_\phi(\Gal(K/kM))$, this relation implies that
$$
\AST(C)=\Lef(E^0,1)_3\cdot\iota(H_0)\cup \Lef(E^0,\tau)_3\cdot\iota((1,\tau)^{-1}(H\setminus H_0))\,.
$$
Then \eqref{equation: keysplit} implies
\begin{equation}\label{equation: expST}
\ST(C)=\ST(E^0,1)_3\big(\iota(H_0) \cup \iota(H\setminus H_0)\big)=\ST(E^0,1)_3\cdot\iota(H)\,.
\end{equation}
Note that, even if $\iota(H)$ is only defined as an element of $\USp(6)/\langle -1\rangle$, the product $\ST(E^0,1)_3\cdot\iota(H)$ is well defined inside $\USp(6)$, provided that $-1\in \ST(E^0,1)_3$.

For the case $C^0=C^0_7$, we have
\begin{align*}
\Lef(E^0,1)(\C) &=\{A\in M_2(\C)|A^\T J_2A=J_2,\, A^{-1}K_2A=K_2\} \\
& =\left\{
\begin{pmatrix}
c-b & 4b \\-2b & c+b
\end{pmatrix}\,|\,c,b\in\C,\,c^2+7b^2=1
\right\}\,.
\end{align*}
Thus, a maximal compact subgroup of $\Lef(E^0,1)(\C)$ is
$$
\ST(E^0,1) =\left\{
\begin{pmatrix}
\cos(2\pi r)-\frac{1}{\sqrt 7}\sin(2\pi r) & \frac{4}{\sqrt 7}\sin(2\pi r) \\-\frac{2}{\sqrt 7} \sin(2\pi r) &  \cos(2\pi r)+\frac{1}{\sqrt 7}\sin(2\pi r)
\end{pmatrix}\,|\,r\in[0,1]
\right\}\,.
$$
Analogously,
\begin{align*}
\ST(E^0,\tau) &=\{A\in M_2(\C)|A^\T J_2A=J_2,\, A^{-1}K_2A=-I_2-K_2 \} \\
& =\left\{
\begin{pmatrix}
ic-ib & 4ib \\ \frac{ic}{2}+\frac{3ib}{2} & ib-ic
\end{pmatrix}\,|\,c,b\in\C,\,c^2+7b^2=1
\right\}\,.
\end{align*}
Thus, a maximal compact subgroup of $\Lef(E^0,\tau)(\C)$ is
$$
\ST(E^0,\tau) =\ST(E^0,1)\cdot
\begin{pmatrix}
i & -i \\0 &-i
\end{pmatrix}
\,.
$$
We can now apply \cite[Lemma 2.3]{FS14} exactly as in the case $C^0=C^0_1$ to complete the proof.
\end{proof}

The previous theorem describes the Sato--Tate group of a twist $C$ of $C^0$.
Now suppose that~$C$ and~$C'$ are both twists of~$C^0$.
The next proposition gives an effective criterion to determine when $\ST(C)$ and $\ST(C')$ coincide. Let $H\subseteq G_{C^0}$ (resp.\ $H'$) be attached to~$C$ (resp.\ $C'$) as in Remark~\ref{remark: pairs}.

\begin{proposition}\label{proposition: conjSTgruops1}
If $H$ and $H'$ are conjugate in $G_{C^0}$, then $\ST(C)$ and $\ST(C')$ coincide.
\end{proposition}

\begin{proof}
Since the Sato--Tate group is defined only up to conjugacy, is suffices to exhibit $A\in \GL_6(\C)$ such that $A^{-1}\ST(C)A=\ST(C')$. Let $g\in G_{C^0}$ be such that $H'=g^{-1}H g$. It is straightforward to check that $\iota(G_{C^0})$ normalizes the group $\ST(E^0,1)_3$. In particular, by Theorem~\ref{theorem: ST groups}~(ii), we have
\begin{equation}\label{equation: conjgroup}
\ST(C')=\ST(E^0,1)_3\iota(H')=\iota(g)^{-1}\ST(E^0,1)_3\iota(H)\iota(g)=\iota(g)^{-1}\ST(C)\iota(g)\,.\qedhere
\end{equation} 
\end{proof}

\begin{corollary}\label{corollary: uppernumb1}
There are at most $23$ Sato--Tate groups of twists of the Klein quartic $C_7^0$.
\end{corollary}

\begin{proof}
There are 23 subgroups of $G_{C^0_7}$, up to conjugacy.
\end{proof}

In the Fermat case, $\ST(C)$ and $\ST(C')$ may coincide when~$H$ and~$H'$ are not conjugate in $G_{C^0}$. We thus require a sharper criterion.

\begin{definition} Let $C$ and $C'$ be twists of $C^0$ and $C^{0\prime}$, respectively (here $C^0$ and $C^{0\prime}$ both denote one of $C^0_1$ or $C^0_7$, but possibly not the same curve in both cases), and let $H$ and $H'$ be the corresponding attached groups. We say that $H$ and $H'$ are \emph{equivalent} if there exists an isomorphism 
\begin{equation}\label{equation: isofin}
\Psi\colon H \rightarrow H'
\end{equation}  
such that $\Psi(H_0)=H_0'$ and for every $h \in H_0$, we have
\begin{equation}\label{equation: propcru}
\Tr(j(h))= \Tr(j(\Psi(h)))\,,
\end{equation}
where $H_0$ and $H_0'$ are defined as in Definition~\ref{definiton: H0} and $j$ denotes compositions $\iota_2\circ \iota_1$ of the embeddings defined in Lemma~\ref{lemma: embeddings} for $C^0$ and $C^{0\prime}$ (two different maps $j$ if $C^0\ne C^{0\prime}$).
\end{definition}

\begin{proposition}\label{proposition: conjSTgruops2} Let $C$ and $C'$ be twists of $C^0$ and $C^{0\prime}$. If $H$ and $H'$ are equivalent, then $\ST(C)$ and $\ST(C')$ coincide.  
\end{proposition}

\begin{proof}
Let us first assume that $C^0=C^{0\prime}$. By Theorem~\ref{theorem: ST groups} (ii), we can consider the group isomorphism
$$
\Phi\colon\ST(C)=\ST(E^0,1)_3\cdot \iota(H)\simeq  \ST(C')=\ST(E^0,1)_3\cdot \iota(H')
$$
defined by sending an element of the form $g=g_0\iota(h)$ to $g_0\iota(\Psi(h))$.
We aim to show that $\ST(C)$ and $\ST(C')$ are conjugate inside $\GL_6(\C)$. This amounts to showing that $\ST(C)$ and $\ST(C')$ are equivalent representations of the same abstract group, for which it suffices to prove the following claim: for every $g\in \ST(C)$, we have $\Tr(g)=\Tr(\Phi(g))$. To prove the claim distinguish the cases: (a) $h\in H_0$, and (b) $h\in H\setminus H_0$.

Suppose we are in case (a). By \eqref{equation: propcru} there exists $A\in \GL_3(M)$ such that $Aj(h)A^{-1}=j(\Psi(h))$ for every $h\in H_0$. Moreover, if we let $r$ denote the composition $\iota_4\circ \iota_3$ of the embeddings defined in Lemma~\ref{lemma: embeddings}, the fact that $A$ has entries in $M$ easily implies that $r(A)$ centralizes $\ST(E^0,1)$, and thus we have
$$
\Phi(g)=g_0r(A)\iota(h) r(A)^{-1}=r(A)g_0\iota(h) r(A)^{-1}=r(A)gr(A)^{-1}\,,
$$
from which the claim follows. In case (b), we have that both $g$ and $\Phi(g)$ have trace $0$, as follows for example from the proof of Corollary~\ref{corollary: import} and the Chebotarev Density Theorem. The claim follows immediately.

If $C^0\ne C^{0\prime}$ then we may assume without loss of generality that~$C$ is a twist of $C^0_7$ and $C'$ is a twist of $C^0_1$.
Now consider the isomorphism
$$
\Phi\colon\ST(C)=\ST(E^0_7,1)_3\cdot \iota(H)\simeq  \ST(C')=\ST(E^0_1,1)_3\cdot \iota(H')
$$
defined by sending an element of the form $g=g_0\iota(h)$ to $Tg_0T^{-1}\iota(\Psi(h))$, where 
$$
T=\begin{pmatrix}
1 & 0\\
-1/\sqrt 7 & 4/\sqrt 7
\end{pmatrix}\,.
$$
We now note that $T\ST(E^0_1,1) T^{-1}=\ST(E^0_7,1)$, and the proof then proceeds exactly as above;
the hypothesis $C^0\not =C^{0\prime}$ implies that $\Tr(j(h))=\Tr(j(\Psi(h)))\in \Q(\sqrt{-1})\cap \Q(\sqrt{-7})=\Q$ for every $h\in H_0$, thus the matrix $A$ from above can be taken in $\GL_3(\Q)$.
\end{proof}

\begin{corollary}\label{corollary: uppernumb2}
The following hold:
\begin{enumerate}[{\rm (i)}]
\item There are at most $54$ distinct Sato--Tate groups of twists of the Fermat quartic.
\item There are at most $60$ distinct Sato--Tate groups of twists of the Fermat and Klein quartics.
\end{enumerate}
\end{corollary}

\begin{proof}
Determining whether two subgroups $H$ and $H'$ are equivalent is a finite problem. Using the computer algebra program \cite{Magma}, one can determine a set of representatives for equivalence classes of subgroups $H$ that turn out to have size $54$ in case (i), and of size~60 in case~(ii).  For the benefit of the reader, here we give a direct proof of (ii), assuming (i).

The $6$ Sato--Tate groups of a twist of the Klein quartic that do not show up as the Sato--Tate group of a twist of the Fermat quartic are precisely those ruled out by the fact that $H_0$ contains an element of order~$7$ (those in rows \#10, \#13, \#14 of Table~\ref{table: example curves Klein}), since $7$ does not divide $\#G_{C_1^0}$.

To show that the other 17 Sato--Tate groups of twists of the Klein quartic also arise for twists of the Fermat quartic, we proceed as follows.
Let $H\subseteq G_{C^0_7}$ correspond to a twist $C$ of $C^0_7$ such that $H_0$ does not contain an element of order~$7$, and let $H'\subseteq C^0_1$ correspond to a twist~$C'$ of~$C^0_1$. In this case, from Tables~\ref{table: charactertable1} and~\ref{table: charactertable7}, to ensure that $H$ and $H'$ are equivalent it suffices to check that:
\begin{enumerate} 
\item There exists an isomorphism $\Psi\colon H\rightarrow H'$ such that $\Psi(H_0)=H_0'$; 
\item $\Tr(j(h))=1$ for every $h\in H_0'$ such that $\ord(h)=4$.
\end{enumerate} 
From Tables~\ref{table: example curves Fermat} and~\ref{table: example curves Klein}, it is trivial to check that for every $H$ as above, one can always find a subgroup $H'$ such that condition (1) is satisfied. Condition (2) is vacuous except for rows \#6, \#7, \#11, and \#12 of Table~\ref{table: example curves Klein}. In these cases, a subgroup $H'$ for which condition (2) is also satisfied can be found by noting that both $j(t_1)$ and $j(t_1^3u_1t_1u_1^3)$ have trace~$1$. More precisely, one finds that the Sato--Tate groups corresponding to these cases coincide with the Sato--Tate groups of rows \#13, \#20, \#34, and \#55 of Table~\ref{table: example curves Fermat}, respectively.
\end{proof}

Combining the lower and upper bounds proved in this section yields our main theorem, which we restate for convenience.

\setcounter{maintheorem}{0}
\begin{maintheorem}
The following hold:
\begin{enumerate}[{\rm (i)}]
\item There are $54$ distinct Sato--Tate groups of twists of the Fermat quartic.  These give rise to $54$ (resp.\ $48$) distinct joint (resp.\ independent) coefficient measures.
\item There are $23$ distinct Sato--Tate groups of twists of the Klein quartic.  These give rise to $23$ (resp.\ $22$) distinct joint (resp.\ independent) coefficient measures.
\item There are $60$ distinct Sato--Tate groups of twists of the Fermat or the Klein quartics.  These give rise to $60$ (resp.\ $54$) distinct joint (resp.\ independent) coefficient measures.
\end{enumerate}
\end{maintheorem}
\begin{proof}
This follows immediately from Corollaries~ \ref{corollary: det mom seq}, \ref{corollary: uppernumb1}, \ref{corollary: uppernumb2}, and Proposition~\ref{proposition: lowernumb}.
\end{proof}

\begin{corollary}\label{corollary: isoSTeqH}
If $C$ and $C'$ are twists of $C^0$ corresponding to $H$ and $H'$, respectively, then $\ST(C)$ and $\ST(C')$ coincide if and only if $H$ and $H'$ are equivalent.
\end{corollary}

\begin{remark} 
One could have obtained the lower bounds of Proposition~\ref{proposition: lowernumb} by computing the joint coefficient measures $\mu_I$ of the Sato--Tate groups explicitly described in Theorem~\ref{theorem: ST groups}. This is a lengthy but feasible task that we will not inflict on the reader.   We note that this procedure also allows for case-by-case verifications of the equalities $\M_{n_1,n_2,n_3}[\mu_I]=\M_{n_1,n_2,n_3}[a]$, and thus of the Sato-Tate conjecture in the cases considered.
\end{remark}

\begin{remark}
Let $\mathfrak X_d$ denote the set of Sato--Tate groups of twists of $C^0_d$. Theorem~\ref{theorem: ST groups} gives a map from the set of subgroups of $G_{C^0_d}$ to $\mathfrak{X}_d$ that assigns to a subgroup $H\subseteq G_{C^0}$ the Sato-Tate group $\ST(E^0,1)_3\cdot \iota(H)$. It also shows that $\mathfrak X_d$ is endowed with a lattice structure compatible with this map and the lattice structure on the set of subgroups of $G_{C^0_d}$. Moreover, Proposition~\ref{proposition: conjSTgruops1} says that this map factors via
$$
\varepsilon_d \colon \mathfrak C_d \rightarrow \mathfrak X_d\,,
$$
where $\mathfrak C_d$ denotes the lattice of subgroups of $G_{C^0_d}$ up to conjugation. Parts (i) and (ii) of Theorem~\ref{theorem: Main} imply that while the map $\varepsilon_7$ is a lattice isomorphism, the map $\varepsilon_1$ is far from being injective. Corollary~\ref{corollary: isoSTeqH} can now be reformulated by saying that two subgroups $H,\,H'\in \mathfrak C_1$ lie in the same fiber of $\varepsilon_1$ if and only if they are equivalent. 
\end{remark}

In virtue of the above remark, one might ask about conditions on twists $C$ and $C'$ corresponding to distinct but equivalent groups $H$ and $H'$ that ensure their Jacobians have the same Sato-Tate group.
One such condition is that $\Jac(C)$ and $\Jac(C')$ are isogenous (recall that the Sato-Tate group of an abelian variety is an isogeny invariant). The next proposition shows that, under the additional hypothesis that $K$ and $K'$ coincide, the previous statement admits a converse.

\begin{proposition}\label{proposition: isojac}
Let $C$ and $C'$ be $k$-twists of $C^0$. Suppose that the corresponding subgroups $H$ and $H'$ of $G_{C^0}$ are equivalent, and that the corresponding fields $K$ and~$K'$ coincide. Then $\Jac(C)$ and $\Jac(C')$ are isogenous.
\end{proposition}

\begin{proof}
Let $S$ be the set of primes of $k$ which are of bad reduction for either $\Jac(C)$ or $\Jac(C')$ or lie over the fixed prime $\ell$. Note that by \cite[Thm. 4.1]{Sil92} the set $S$ contains the primes of $k$ ramified in $K$ or $K'$. By Faltings' Isogeny Theorem \cite[Korollar~2]{Fal83}, it suffices to show that for every $\p\not\in S$, we have
\begin{equation}\label{equation: locfac}
L_\p(\Jac(C),T)=L_\p(\Jac(C'),T)\,.
\end{equation}
If $\Frob_\p\not \in G_{kM}$, by the proof of Corollary~\ref{corollary: import}, both polynomials of \eqref{equation: locfac} have the same expression, which depends only on the order of (the projection of) $\Frob_\p$ in $\Gal(K/kM)$. To obtain \eqref{equation: locfac} for those $\p$ such that $\Frob_\p\in G_{kM}$, we will show that $V_\ell(\Jac(C)_{kM})$ and $V_\ell(\Jac(C')_{kM})$ are isomorphic as $\Q_\ell[G_{kM}]$-modules. Indeed, the fact that $H$ and $H'$ are equivalent pairs implies that the restrictions from $\Aut(C^0_M)$ to $H_0$ and $H_0'$ of the representations $\theta_{E^0,C^0}$ attached to $C$ and $C'$ are equivalent. Together with \eqref{equation: iso theta}, this shows that $\theta_{M,\sigma}(E^0,\Jac(C))$ and $\theta_{M,\sigma}(E^0,\Jac(C'))$ are equivalent representations. The desired $\Q_{\ell}[G_{kM}]$-module isomorphism follows now from \eqref{equation:flip}.    
\end{proof}

\begin{remark}\label{remark: isojac}
Let $H$ and $H'$ be any two equivalent pairs attached to twists $C$ and $C'$ of the same curve $C^0$. As one can read from Tables \ref{table: example curves Fermat} and \ref{table: example curves Klein} (and as we will see in the next section), one can choose $C$ and $C'$ so that $K$ and $K'$ coincide. It follows from Proposition~\ref{proposition: isojac} that on Table~\ref{table: example curves Fermat} (resp.\ Table~\ref{table: example curves Klein}) two curves $C$ and $C'$ satisfy $\ST(C)=\ST(C')$ if and only if $\Jac(C)\sim \Jac(C')$.
\end{remark}

We conclude this section with an observation that is not directly relevant to our results but illustrates a curious phenomenon arising among the twists of the Fermat quartic in Table~\ref{table: example curves Fermat}. Let 
$$
C_5\colon 9x^4+9y^4-4z^4=0 \qquad\text{and}\qquad C_8 \colon 9x^4-4 y^4+z^4=0
$$
be the curves listed in rows \#5 and \#8 of Table~\ref{table: example curves Fermat}. As can be seen in Table~\ref{table: example curves Fermat}, the groups $\ST(C_5)$ and $\ST(C_8)$ coincide, as do the respective fields $K$.
Proposition~\ref{proposition: isojac} thus implies that the Jacobians of $C_5$ and $C_8$ are isogenous, but in fact more is true.

\begin{proposition}
The curves $C_5$ and $C_8$ are not isomorphic (over $\Q$), but their reductions $\tilde C_5$ and $\tilde C_8$ modulo $p$ are isomorphic (over $\F_p$) for every prime $p>3$.
\end{proposition}

\begin{proof}
The twists $C_5$ are $C_8$ of $C_1^0$ are not isomorphic because they arise from non-conjugate subgroups $H$ of $G_{C_1^0}$.
For the reductions $\tilde C_5$ and $\tilde C_8$ we first consider $p\equiv 1 \pmod 4$. We claim that $-4$ is a fourth power modulo $p$; this follows from the factorization
$$
x^4+4=(x^2-2x+2)(x^2+2x+2) \qquad \text{in $\Q[x]$}
$$
together with the fact that $x^2-2x+2$ and $x^2+2x+2$ have discriminant $-4$. It follows that $\tilde C_5$ and $\tilde C_8$ are both isomorphic to $9x^4+9y^4+z^4=0$ (over $\F_p$).

Suppose now that $p\equiv -1 \pmod 4$. Then $9$ is a fourth power modulo $p$ since
$$
x^4-9=(x^2-3)(x^2+3) \qquad \text{and} \qquad \binom{-3}{p}=-\binom{3}{p}\,.
$$
It follows that $\tilde C_5$ and $\tilde C_8$ are both isomorphic to $x^4+y^4-4z^4=0$ (over $\F_p$).
\end{proof}

\subsection{Curve equations}\label{section: curves}

In this section we construct explicit twists of the Fermat and the Klein quartics realizing each of the subgroups $H\subseteq G_{C^0}$ described in Remark~\ref{remark: pairs}.
Recall that each $H$ has an associated subgroup $H_0\coloneqq H\cap G_0$ of index at most 2 (see Definition~\ref{definiton: H0}), and there exists a twist corresponding to $H$ with $k=kM$ if and only if $H=H_0$, where, as always, $M$ denotes the CM field of $E^0$ (the elliptic curve for which $\Jac(C^0)\sim (E^0)^3$).

Equations for these twists are listed in Tables~\ref{table: example curves Fermat} and \ref{table: example curves Klein} in~\S\ref{section: tables}.
As explained in Remark~\ref{remark: pairs}, in the Fermat case every subgroup $H\subseteq G_{C^0_1}$ with $[H:H_0]=1$ (case (c$_1$) of Definition~\ref{definiton: H0}) arises as $H_0'$ for some subgroup $H'\subseteq G_{C^0_1}$ for which $[H':H_0']=2$ (case (c$_2$) of Definition~\ref{definiton: H0}), and a twist corresponding to $H$ can thus be obtained as the base change to $kM$ of a twist corresponding to~$H'$.  We thus only list twists for the 59 subgroups $H$ in case (c$_2$), since base changes of these twists to $kM$ then address the 24 subgroups $H$ in case (c$_1$).
In the Klein case we list twists for the 11 subgroups $H$ in case (c$_2$) and also the 3 exceptional subgroups $H$ in case (c$_1$) that cannot be obtained as base changes of twists corresponding to subgroups in case (c$_2$); see Remark~\ref{remark: pairs}.

Our twists are all defined over base fields $k$ of minimal possible degree, never exceeding~$2$.  For the 3 exceptional subgroups $H$ in the Klein case noted above, we must have $k=kM$, and we use $k=M=\Q(\sqrt{-7})$.  In all but 5 of the remaining cases with $[H:H_0]=2$ we use $k=\Q$.  These 5 exceptions are all explained by Lemma~\ref{lemma: minimal} below (the second of the 4 pairs listed in Lemma~\ref{lemma: minimal} arises in both the Fermat and Klein cases, leading to 5 exceptions in total).
In each of these 5 exceptions with $[H:H_0]=2$, the subgroup $H_0$ also arises as $H_0'$ for some subgroup $H'\subseteq G_{C_1^0}$ with $[H':H'_0]=2$ that is realized by a twist with $k=\Q$, allowing $H_0$ to be realized over a quadratic field as the base change to $M$ of a twist defined over $\Q$.

\begin{lemma}\label{lemma: minimal}
Twists of the Fermat or Klein quartics corresponding to pairs $(H,H_0)$ with the following pairs of $\mathrm{GAP}$ identifiers cannot be defined over a totally real field:
\[
(\langle 4,1\rangle,\langle 2,1\rangle),\quad (\langle 8,1\rangle,\langle 4,1\rangle),\quad (\langle 8,4\rangle,\langle 4,1\rangle), \quad(\langle 16,6\rangle,\langle(8,2\rangle).
\]
\end{lemma}
\begin{proof}
If $k$ is totally real then complex conjugation acts trivially on $k$ but not on $kM$, giving an involution in $H=\Gal(K/k)$ with non-trivial image in $H/H_0=\Gal(K/k)/\Gal(K/kM)$.  For the four pairs $(H,H_0)$ listed in the lemma, no such involution exists.
\end{proof}

In addition to listing equations and a field of definition $k$ for a twist $C$ associated to each subgroup $H$, in Tables~\ref{table: example curves Fermat} and \ref{table: example curves Klein} we also list the minimal field $K$ over which all the endomorphisms of $\Jac(C)$ are defined, and we identify the conjugacy class of $\ST(C)$ and $\ST(C_{kM})$, which depends only on $H$, not the particular choice of $C$.  As noted in Remark~\ref{remark: isojac}, we have chosen twists $C$ so that twists with the same Sato-Tate group have the same fields $K$ and thus have isogenous Jacobians, by Proposition~\ref{proposition: isojac} (thereby demonstrating that the hypotheses of the proposition can always be satisfied).

\subsubsection{\emph{\textbf{Constructing the Fermat twists}}}

The twists of the Fermat curve over any number field are parametrized in \cite{Lor18}, and we specialize the parameters in Theorems 4.1, 4.2, 4.5 of \cite{Lor18} to obtain the desired  examples.  In every case we are able to obtain equations with coefficients in $\mathbb{Q}$, but as explained above, we cannot always take $k=\Q$; the exceptions can be found in rows \#4, \#13, \#27, \#33 of Table~\ref{table: example curves Fermat}.
The parameterizations in \cite{Lor18} also allow us to determine the field $L$ over which all the isomorphisms to $(C_7^0)_\Qbar$ are defined, which by Lemma~\ref{lemma: K in L} and Proposition~\ref{proposition: K=L}, this is the same as the field $K$ over which all the endomorphisms of $\Jac(C_7^0)_\Qbar$ are defined.

Specializing the parameters for each of the $59$ cases with $k\not= kM$ involves a lot of easy but tedious computations.  The resulting equations are typically not particularly pleasing to the eye or easy to format in a table; in order to make them more presentable we used the algorithm in \cite{Sut12} to simplify the equations.  To give just one example, for the unique subgroup $H$ with ID$(H) = \langle 24,13\rangle$, the equation we obtain from specializing the parameterizations in \cite{Lor18} is
\begin{align*}
14x^4 &- 84x^3y + 392x^3z + 588x^2y^2 - 2940x^2yz + 4998x^2z^2 - 980xy^3 + 9996xy^2z\\& - 28812xyz^2 + 30184xz^3 + 833y^4 - 9604y^3z + 45276y^2z^2 - 90552yz^3 + 69629z^4 = 0,
\end{align*}
but the equation listed for this curve in row \#48 of Table~\ref{table: example curves Fermat} is
\[
3x^4 + 4x^3y + 4x^3z + 6x^2y^2 + 6x^2z^2 + 8xy^3 + 12xyz^2 + 5y^4 + 4y^3z + 12y^2z^2 + z^4 = 0.
\]
We used of the number field functionality in \cite{Magma} and \cite{PARI} to minimize the presentation of the fields $K$ listed in the tables (in particular, the function \texttt{polredabs} in PARI/GP).

\subsubsection{\emph{\textbf{Constructing the Klein twists}}}

Twists of the Klein curve over arbitrary number fields are parametrized in Theorems~6.1 and 6.8 of \cite{Lor18}, following the method described in \cite{Lor17}, which is based on the resolution of certain Galois embedding problems.
However, in the most difficult case, in which $H=G_{C_7^0}$ has order 336, this Galois embedding problem is computationally difficult to resolve explicitly.
This led us to pursue an alternative approach that exploits the moduli interpretation of twists of the Klein curve as twists of the modular curve $X(7)$.
As described in \cite[\S3]{HK00} and \cite[\S4]{PSS07}, associated to each elliptic curve $E/\Q$ is a twist $X_E(7)$ of the Klein quartic defined over $\Q$ that parameterizes isomorphism classes of 7-torsion Galois modules isomorphic to $E[7]$, as we recall below.
With this approach we can easily treat the case $H=G_{C_7^0}$, and we often obtain twists with nicer equations.  In one case, we also obtain a better field of definition~$k$, allowing us to achieve the minimal possible degree $[k:\Q]$ in every case.

However, as noted in \cite[\S 4.5]{PSS07}, not every twist of the Klein curve can be written as $X_E(7)$ for some elliptic curve $E/\Q$, and there are several subgroups $H\subseteq G_{C_7^0}$
for which the parameterizations in \cite{Lor18} yield a twist of the Klein quartic defined over $\Q$ but no twists of the form $X_E(7)$ corresponding to $H$ exist.  We are thus forced to use a combination of the two approaches.  For twists of the form $X_E(7)$ we need to determine the minimal field over which the endomorphisms of $\Jac(X_E(7))$ are defined; this is addressed by Propositions \ref{proposition: modularK} and \ref{proposition: phi7} below.

Let $E/\Q$ be an elliptic curve and let $E[7]$ denote the $\F_7[G_\Q]$-module of $\Qbar$-valued points of the kernel of the multiplication-by-$7$ map $[7]\colon E\rightarrow E$. The Weil pairing gives a $G_\Q$-equivariant isomorphism $\bigwedge ^2 E[7]\simeq \mu_7$, where $\mu_7$ denotes the $\F_7[G_\Q]$-module of $7$th roots of unity.
Let $Y_E(7)$ be the curve defined over $\Q$ described in \cite[\S3]{HK00} and \cite[\S4]{PSS07}. For any field extension $L/\Q$, the $L$-valued points of $Y_E(7)$ parametrize isomorphism classes of pairs $(E',\phi)$, where $E'/L$ is an elliptic curve and $\phi\colon E[7]\rightarrow E'[7]$ is a symplectic isomorphism. By a symplectic isomorphism we mean a $G_L$-equivariant isomorphism $\phi\colon E[7]\rightarrow E'[7]$ such that the diagram
\begin{equation}\label{equation: comm symp}
\xymatrix{
\bigwedge ^2  E[7]\ar[r]^{\simeq}\ar[d]_{\bigwedge ^2 \phi} & \mu_7 \ar[d]^{\id} \\
\bigwedge ^2 E'[7] \ar[r]^{\simeq}& \mu_7\,,}
\end{equation}
commutes, where the horizontal arrows are Weil pairings. Two pairs $(E',\phi)$ and $(\tilde E', \tilde \phi)$ are isomorphic whenever there exists an isomorphism $\varepsilon\colon E'\rightarrow \tilde E'$ such that $\tilde \phi=\varepsilon \circ \phi$.  

In \cite{HK00} it is shown that $X_E(7)$, the compactification of $Y_E(7)$, is a twist of $C^0_7$, and an explicit model for $X_E(7)$ is given by \cite[Thm.\ 2.1]{HK00}, which states that if $E$ has the Weierstrass model $y^2=x^3+ax+b$ with $a,b\in \Q$, then
\begin{equation}\label{eq: XE7}
ax^{4}+7bx^{3}z+3x^{2}y^{2}-3a^{2}x^{2}z^{2}-6bxyz^{2}-5abxz^{3}+2y^{3}z+3ay^{2}z^{2}+2a^{2}yz^{3}-4b^{2}z^{4}=0\,
\end{equation}
is a model for $X_E(7)$ defined over $\Q$.

We will use the moduli interpretation of $X_E(7)$ to determine the minimal field over which all of its automorphisms are defined. Recall that the action of $G_\Q$ on $E[7]$ gives rise to a Galois representation
\[
\varrho_{E,7}\colon G_\Q\to \Aut(E[7])\simeq \GL_2(\F_7)\,.
\]
Let $\overline \varrho_{E,7}$ denote the composition $\pi\circ \varrho_{E,7}$, where $\pi\colon\GL_2(\F_7)\rightarrow \PGL_2(\F_7)$ is the natural projection.
\begin{proposition}\label{proposition: modularK} The following field extensions of $\Q$ coincide:
\begin{enumerate}[{\rm (i)}]
\item The minimal extension over which all endomorphisms of $\Jac(X_E(7))$ are defined;
\item The minimal extension over which all automorphisms of $X_E(7)$ are defined;
\item The field $\overline \Q^{\ker \overline \varrho_{E,7}}$;
\item The minimal extension over which all $7$-isogenies of $E$ are defined.  
\end{enumerate}
In particular, if $K$ is the field determined by these equivalent conditions, then
$$
\Gal(K/\Q)\simeq \im (\overline \varrho_{E,7}) \simeq \im (\varrho_{E,7}) / (\im \varrho_{E,7} \cap \F_7^\times )\,.
$$
\end{proposition}

\begin{proof}
The equivalence of (i) and (ii) follows from \ref{lemma: K in L}, since $X_E(7)$ is a twist of $C^0_7$.

Following \cite{PSS07}, let $\Aut_{\wedge}(E[7])$ denote the group of symplectic automorphisms of $E[7]$. Given a field extension $F/\Q$, let us write $E[7]_F$ for the $\F_7[G_F]$-module obtained from $E[7]$ by restriction from $G_\Q$ to $G_F$. 
Note that $E[7]_\Qbar\simeq (\Z/7\Z)^2$. Under this isomorphism, for any $g\in \Aut_{\wedge}(E[7]_\Qbar)$, diagram \eqref{equation: comm symp} becomes
$$
\xymatrix{
\bigwedge ^2  (\Z/7\Z)^2\ar[r]^{\simeq}\ar[d]_{\det(g)} & \Z/7\Z \ar[d]^{\id} \\
\bigwedge ^2 (\Z/7\Z)^2 \ar[r]^{\simeq}& \Z/7\Z\,,}
$$
from which we deduce
\begin{equation}\label{equation: nonequiv}
\Aut_{\wedge}(E[7]_\Qbar)\simeq \SL_2(\F_7)\,.
\end{equation}
Each $g\in \Aut_{\wedge}(E[7]_\Qbar)$ acts on $Y_E(7)_\Qbar$ via $(E',\phi)\mapsto (E',\phi\circ g^{-1})$. This action extends to $X_E(7)$, from which we obtain a homomorphism
\begin{equation}\label{equation: autaut}
\Aut_\wedge(E[7]_\Qbar)\rightarrow \Aut(X_E(7)_\Qbar)\,.
\end{equation}
This homomorphism is non-trivial, since elements of its kernel induce automorphisms of $E_\Qbar$, but $\Aut(E_\Qbar)$ is abelian and $\Aut_\wedge (E[7])_\Qbar\simeq \SL_2(\F_7)$ is not, and it cannot be injective, since the group on the right has cardinality $168 < \#\SL_2(\F_7)=336$.
The only non-trivial proper normal subgroup of $\SL_2(\F_7)$ is $\langle \pm 1\rangle$, thus \eqref{equation: autaut} induces a $G_\Q$-equivariant isomorphism  
\begin{equation}\label{equation: autaut2}
\Aut_\wedge(E[7])/\langle \pm 1\rangle \rightarrow \Aut(X_E(7)_\Qbar)\,.
\end{equation}
By transport of structure, we now endow $\SL_2(\F_7)$ with a $G_\Q$-module structure that turns \eqref{equation: nonequiv} into a $G_\Q$-equivariant isomorphism $\varphi\colon \Aut_\wedge(E[7]_\Qbar)\overset{\sim}{\rightarrow}\SL_2(\F_7)$.
Let $\varrho\colon G_\Q\rightarrow \Aut (\SL_2(\F_7))$ be the representation associated to this $G_\Q$-module structure. Since the action of $\sigma \in G_\Q$ on each $g \in \Aut_{\wedge}(E[7]_\Qbar)$ is defined by
$$
({}^\sigma g)(P)={}^\sigma (g({}^{\sigma^{-1}}P))\,, 
$$
for $P\in E[7]$, we have $\varrho(\sigma) (\varphi(g))= \varrho_{E,7}(\sigma)\cdot\varphi(g)\cdot \varrho_{E,7}(\sigma)^{-1}$. The $\G_\Q$-action is trivial on $\langle \pm 1\rangle$ and thus  descends to $\PSL_2(\F_7)$.
Let us write $\PSL_2(\F_7)^{\varrho}$ for $\PSL_2(\F_7)$ endowed with the $G_\Q$-action given by conjugation by $\varrho_{E,7}$. Then \eqref{equation: nonequiv} with \eqref{equation: autaut2} yield a  $G_\Q$-equivariant isomorphism
$$
\PSL_2(\F_7)^{\varrho}\simeq \Aut(X_E(7)_\Qbar) \,.
$$
This implies that the field described in (ii) coincides with $\Qbar^{\Ker(\varrho)}$, and we now note that
\begin{align*}
\ker(\varrho) &=\{\sigma \in G_\Q\,|\,\varrho_{E,7}(\sigma)\cdot \alpha\cdot \varrho_{E,7}(\sigma)^{-1}=\alpha\,,\,\text{for all }\alpha\in \PSL_2(\F_7)\}\\
 &=\{\sigma \in G_\Q\,|\,\varrho_{E,7}(\sigma)\in \F_7^\times\}\\
 &=\ker(\overline \varrho_{E,7})\,,
\end{align*}
thus $\Qbar^{\ker(\varrho)}=\Qbar^{\ker(\varrho_{E,7})}$ is the field described in (iii).

Now let $\mathbb P(E[7])$ denote the projective space over $E[7]$, consisting of its 8 linear $\F_7$-subspaces, equivalently, its 8 cyclic subgroups of order 7.
The $G_\Q$-action on $\mathbb P(E[7])$ gives rise to the projective Galois representation
$$
\overline \varrho_{E,7}\colon G_\Q \rightarrow \Aut(\mathbb P(E[7]))\simeq \PGL_2(\F_7)\,.
$$
The minimal field extension $K$ over which the $G_K$-action on $\mathbb P(E[7])$ becomes trivial is precisely the minimal field over which the cyclic subgroups of $E[7]$ of order 7 all become Galois stable, equivalently, the minimal field over which all the $7$-isogenies of $E$ are defined.
It follows that the fixed field of $\ker\overline\varrho_{E,7}$ identified in (iii) is also the field described in (iv).
\end{proof}

To explicitly determine the field over which all the $7$-isogenies of $E$ are defined, we rely on Proposition~\ref{proposition: phi7} below, in which $\Phi_7(X,Y)\in \Z[X,Y]$ denotes the classical modular polynomial; the equation for $\Phi_7(X,Y)$ is too large to print here, but it is available in \cite{Magma} and can be found in the tables of modular polynomials listed in \cite{Sut} that were computed via \cite{BLS12}; it is a symmetric in $X$ and $Y$, and has degree $8$ in both variables.

The equation $\Phi_7(X,Y)=0$ is a canonical (singular) model for the modular curve $Y_0(7)$ that parameterizes $7$-isogenies.
If $E_1$ and $E_2$ are elliptic curves related by a 7-isogeny then $\Phi_7(j(E_1),j(E_2))=0$, and if $j_1,j_2\in F$ satisfy $\Phi_7(j_1,j_2)=0$, then there exist elliptic curves $E_1$ and $E_2$ with $j(E_1)=j_1$ and $j(E_2)=j_2$ that are related by a 7-isogeny.
However, this $7$-isogeny need not be defined over~$F$!  The following proposition characterizes the relationship between $F$ and the minimal field $K$ over which all the 7-isogenies of $E$ are defined.

\begin{proposition}\label{proposition: phi7}
Let $E$ be an elliptic curve over a number field $k$ with $j(E)\ne 0,1728$.
Let $F$ be the splitting field of $\Phi_7(j(E),Y)\in k[X]$, and let $K$ be the minimal field over which all the $7$-isogenies of $E$ are defined.
The fields $K$ and $F$ coincide.
\end{proposition}
\begin{proof}
Let $S$ be the multiset of roots of $\Phi_7(j(E),Y)$ in $\Qbar$, viewed as a $G_k$-set in which the action of $\sigma\in G_k$ preserves multiplicities: we have $m(\sigma(r))=m(r)$ for all $\sigma\in G_k$, where $m(r)$ denotes the multiplicity of $r$ in $S$.
Let $\mathbb{P}(E[7])$ be the $G_k$-set  of cyclic subgroups $\langle P\rangle$ of $E[7]$ of order~$7$.
In characteristic zero every isogeny is separable, hence determined by its kernel up to composition with automorphisms; this yields
a surjective morphism of $G_k$-sets $\varphi\colon \mathbb{P}(E[7])\to S$ defined by $\langle P\rangle \mapsto j(E/\langle P\rangle)$ with $m(r)=\#\varphi^{-1}(r)$ for all $r\in S$ (note $\#\mathbb{P}(E[7])=8=\sum_{r\in S}m(r)$).
The $G_k$-action on $S$ factors through the $G_k$-action on $\mathbb{P}(E[7])$, and we thus have group homomorphisms
\[
G_k\overset{\bar\varrho_{E,7}}{\longrightarrow} \Aut(\mathbb{P}(E[7])\overset{\phi}{\longrightarrow} \Aut(S),
\]
where $\phi\colon \bar\varrho_{E,7}(G_k)\to \Aut(S)$ is defined by $\phi(\sigma)(\varphi(\langle P\rangle)\coloneqq\varphi(\sigma(\langle P\rangle))$ for each $\sigma\in \bar\varrho_{E,7}(G_k)$.
We then have $K=\Qbar^{\ker\bar\varrho_{E,7}}$ and $F=\Qbar^{\ker(\phi\circ\bar\varrho_{E,7})}$, so $F\subseteq K$.

If $E$ does not have complex multiplication, then $m(r)=1$ for all $r\in S$, since otherwise over $\Qbar$ we would have two 7-isogenies $\alpha,\beta\colon E_{\Qbar}\to E'$ with distinct kernels, and then $(\hat \alpha\circ\beta)\in \End(E_{\Qbar})$ is an endomorphism of degree 49 which is not $\pm 7$, contradicting $\End(E_{\Qbar})\simeq\Z$.
It follows that $\varphi$ and therefore $\phi$ is injective, so  $\ker\bar\varrho_{E,7}=\ker(\phi\circ\bar\varrho_{E,7})$ and $K=F$.

If $E$ does have complex multiplication, then $\End(E_{\Qbar})$ is isomorphic to an order in an imaginary quadratic field $M$.
We now consider the isogeny graph whose vertices are $j$-invariants of elliptic curves $E'/FM$ with edges $(j_1,j_2)$ present with multiplicity equal to the multiplicity of $j_2$ as a root of $\Phi_7(j_1,Y)$.  Since $j(E)\ne 0,1728$, the component of $j(E_{FM})$ in this graph is an isogeny volcano, as defined in \cite{Sut13}.
In particular, there are at least 6 distinct edges $(j(E_{FM}),j_2)$ (edges with multiplicity greater than~$1$ can occur only at the surface of an isogeny volcano and the subgraph on the surface is regular of degree at most $2$).  It follows that $m(r)>1$ for at most one $r\in S$.

The image of $\bar\varrho_{E,7}$ is isomorphic to a subgroup of $\PGL_2(\F_7)$, and this implies that if $\bar\varrho_{E,7}(\sigma)$ fixes more than $2$ elements of $\mathbb{P}(E[7])$ then $\sigma\in \ker\bar\varrho_{E,7}$.  This necessarily applies whenever $\bar\varrho_{E,7}(\sigma)$ lies in $\ker\phi$, since it must fix 6 elements, thus $\ker\bar\varrho_{E,7}=\ker(\phi\circ\bar\varrho_{E,7})$ and $K=F$.
\end{proof}

\begin{corollary}
Let $E$ be an elliptic curve over a number field $k$ with $j(E)\ne 0,1728$.  The minimal field $K$ over which all the $7$-isogenies of $E$ are defined depends only on $j(E)$.
\end{corollary}

\begin{remark}
The first part of the proof of Proposition~\ref{proposition: phi7} also applies when $j(E)=0,1728$, thus we always have $F\subseteq K$.  Equality does not hold in general, but a direct computation finds that $[K\!:\!F]$ must divide $6$ (resp.\ $2$) when $j(E)=0$ (resp.\ $1728$), and this occurs when $k=\Q$.
\end{remark}

We now fix $E$ as the elliptic curve $y^{2}=x^{3}+6x+7$ with Cremona label~\href{http://www.lmfdb.org/EllipticCurve/Q/144b1}{\texttt{144b1}}.
Note that $\varrho_{E,7}$ is surjective; this can be seen in the entry for this curve in the L-functions and Modular Forms Database \cite{LMFDB} and was determined by the algorithm in \cite{Sut16}.
It follows that $\Gal(\Q(E[7])/\Q)\simeq \GL_2(\F_7)$, and Proposition~\ref{proposition: modularK} implies that $H:=\Gal(K/\Q)\simeq \PGL_2(\F_7)$.

For our chosen curve $E$ we have $a=6$ and $b=7$.  Plugging these values into equation \eqref{eq: XE7} and applying the algorithm of \cite{Sut12} to simplify the result yields the curve listed in entry \#14 of Table~\ref{table: example curves Klein} for $H=G_{C_7^0}$.
To determine the field $K$, we apply Proposition~\ref{proposition: phi7}. 
Plugging $j(E)=48384$ into $\Phi_7(x,j(E))$ and using PARI/GP to simplify the resulting polynomial, we find that $K$ is the splitting field of the polynomial $\href{http://www.lmfdb.org/NumberField/8.2.153692888832.1}{x^8 + 4x^7 + 21x^4 + 18x + 9}$.

We applied the same procedure to obtain the equations for $C$ and the polynomials defining $K$ that are listed in rows \#4, \#6, \#9, \#10 of Table~\ref{table: example curves Klein} using the elliptic curves $E$ with Cremona labels \href{http://www.lmfdb.org/EllipticCurve/Q/2450ba1}{\texttt{2450ba1}}, \href{http://www.lmfdb.org/EllipticCurve/Q/64a4}{\texttt{64a4}}, \href{http://www.lmfdb.org/EllipticCurve/Q/784h1}{\texttt{784h1}}, \href{http://www.lmfdb.org/EllipticCurve/Q/36a1}{\texttt{36a1}}, respectively with appropriate adjustments for the cases with $j(E)=0,1728$ as indicated in Proposition~\ref{proposition: phi7}.  The curve in row \#11 is a base change of the curve in row \#6, and for the remaining 7 curves we used the parameterizations in \cite{Lor18}.

\subsection{Numerical computations}\label{section: computations}

In the previous sections we have described the explicit computation of several quantities related to twists $C$ of $C^0=C^0_d$, where $C^0_d$ is our fixed model over $\Q$ for the Fermat quartic ($d=1$) or the Klein quartic ($d=7$), with $\Jac(C^0)\sim (E^0)^3$, where $E^0=E^0_d$ is an elliptic curve over $\Q$ with CM by $M=\Q(\sqrt{-d})$ defined in \eqref{eq:E0}.  These include:
\begin{itemize}
\item Explicit equations for twists $C$ of $C^0$ corresponding to subgroups $H$ of $G_{C^0}$;
\item Defining polynomials for the minimal field $K$ for which $\End(\Jac(C)_K)=\End(\Jac(C)_\Qbar)$;
\item Independent and joint coefficient moments of the Sato-Tate groups $\ST(\Jac(C))$.
\end{itemize}
These computations are numerous and lengthy, leaving many opportunities for errors, both by human and machine.
We performed several numerical tests to verify our computations.

\subsubsection{\emph{\textbf{Na\"ive point-counting}}}
A simple but effective way to test the compatibility of a twist $C/k$ and endomorphism field $K$ is to verify that for the first several degree one primes $\p$ of $k$ of good reduction for $C$ that split completely in $K$, the reduction of $\Jac(C)$ modulo $\p$ is isogenous to the cube of the reduction of $E^0$ modulo the prime $p\coloneqq N(\p)$.
By a theorem of Tate, it suffices to check that $L_\p(\Jac(C),T)=L_p(E^0,T)^3$.
For this task we used optimized brute force point-counting methods adapted from \cite[\S 3]{KS08}.
The $L$-polynomial $L_\p(\Jac(C),T)$ is the numerator of the zeta function of $C$, a genus 3 curve, so it suffices to compute $\#C(\F_p)$, $\#C(\F_{p^2})$, $\#C(\F_{p^3})$, reducing the problem to counting points on smooth plane quartics and elliptic curves over finite~fields.

To count projective points $(x:y:z)$ on a smooth plane quartic $f(x,y,z)=0$ over a finite field~$\F_q$, one first counts affine points $(x:y:1)$ by iterating over $a\in \F_q$, computing the number~$r$ of distinct $\F_q$-rational roots of $g_a(x)\coloneqq f(x,a,1)$ via $r=\deg(\gcd(x^q-x,g_a(x)))$, and then determining the multiplicity of each rational root by determining the least $n\ge 1$ for which $\gcd(g_a(x),g^{(n)}_a(x))=1$, where $g^{(n)}_a$ denotes the $n$th derivative of $g_a\in \F_q[x]$; note that to compute $\gcd(x^q-x,g_a(x))$ one first computes $x^q\bmod g_a(x)$ using a square-and-multiply algorithm.
Having counted affine points $(x:y:1)$ one then counts the $\F_q$-rational roots of $f(x,1,0)$ and finally checks whether $(1:0:0)$ is a point on the curve.

To optimize this procedure one first seeks a linear transformation of $f(x,y,z)$ that ensures $g_a(x)=f(x,a,1)$ has degree at most 3 for all $a\in \F_q$; for this it suffices to translate a rational point to $(1:0:0)$, which is always possible for $q\ge 37$ (by the Weil bounds).
This yields an $O(p^3(\log p)^{2+o(1)})$-time algorithm to compute $L_\p(\Jac(C),T)$ that is quite practical for~$p$ up to $2^{12}$, enough to find several (possibly hundreds) of degree one
primes $\p$ of $k$ that split completely in $K$.

Having computed $L_\p(\Jac(C),T)$, one compares this to $L_p(E^0,T)^3$; note that the polynomial $L_p(E^0,T)=pT^2-a_pT+1$ is easily computed via $a_p=p+1-\#E^0(\F_p)$.
If this comparison fails, then either $C$ is not a twist of $C^0$, or not all of the endomorphisms of $\Jac(C)_\Qbar$ are defined over~$K$.  The converse is of course false, but if this comparison succeeds for many degree-1 primes $\p$ it gives one a high degree of confidence in the computations of $C$ and $K$.\footnote{The objective of this test is not to prove anything, it is simply a mechanism for catching mistakes, of which we found several; most were our own,  but some were due to minor errors in the literature, and at least one was caused by a defect in one of the computer algebra systems we used.}  Note that this test will succeed even when $K$ is not minimal, but we also check that $H\simeq \Gal(K/k)$, which means that so long as~$C$ is a twist of $C^0$  corresponding to the subgroup $H\subseteq G_{C_0}$, the field $K$ must be minimal.

\subsubsection{\emph{\textbf{An average polynomial-time algorithm}}}
In order to numerically test our computations of the the Sato-Tate groups $\ST(C)$, and to verify our computation of the coefficient moments, we also computed Sato-Tate statistics for all of our twists $C/k$ of the Fermat and Klein quartics.  This requires computing the $L$-polynomials $L_\p(\Jac(C),T)$ at primes $\p$ of good reduction for $C$ up to some bound $N$, and it suffices to consider only primes $\p$ of prime norm $p=N(\p)$, since nearly all the primes of norm less than $N$ are degree-1 primes.
In order to get statistics that are close to the values predicted by the Sato-Tate group one needs $N$ to be fairly large.  We used $N=2^{26}$, which is far too large for the naive $O(p^3(\log p)^{2+o(1)})$-time algorithm described above to be practical, even for a single prime $p\approx N$, let alone all good $p\le N$.

In \cite{HS}, Harvey and Sutherland give an average polynomial-time algorithm to count points on smooth plane quartics over $\Q$ that allows one to compute $L_p(\Jac(C),T)\bmod p$ for all good primes $p\le N$ in time $O(N(\log N)^{3+o(1)})$, which represents an average cost of $O((\log p)^{4+o(1)})$ per prime $p\le N$.
This is achieved by computing the Hasse-Witt matrices of the reductions of $C$ modulo $p$ using a generalization of the approach given in \cite{HS14,HS16} for hyperelliptic curves.
In \cite{HS} they also give an $O(\sqrt{p}(\log p)^{1+o(1)})$-time algorithm to compute $L_p(\Jac(C),T)\bmod p$ for a single good prime $p$, which allows one to handle reductions of smooth plane quartics $C$ defined over number fields at degree one primes; this increases the total running time for $p\le N$ to $O(N^{3/2+o(1)})$, which is still feasible with $N=2^{26}$.%\footnote{In principal the average polynomial-time algorithm in \cite{HS} generalizes to smooth plane quartics over number fields, but no implementation of this generalization was available to us.}

Having computed $L_p(\Jac(C),T)\bmod p$, we need to lift this polynomial for $(\Z/p\Z)[T]$ to $\Z[T]$, which is facilitated by Proposition~\ref{proposition: trthetaMEA} below.
It follows from the Weil bounds that the linear coefficient of $L_p(\Jac(C),T)$ is an integer of absolute value at most $6\sqrt{p}$.
For $p>144$ the value of this integer is uniquely determined by its value modulo $p$, and for $p<144$ we can apply the naive approach described above.
This uniquely determines the value in the column labeled $F_1(x)$ in Tables~\ref{table: Coefficient transformations Fermat} and~\ref{table: Coefficient transformations Klein} of Proposition~\ref{proposition: trthetaMEA} below, which then determines the values in columns $F_2(x)$ and $F_3(s)$, allowing the polynomial $L_p(\Jac(C),T)\in \Z[T]$ to be completely determined.

\begin{proposition}\label{proposition: trthetaMEA}  For $\tau$ in $\Gal(L/k)$, let $s=s(\tau)$ and $t=t(\tau)$ denote the orders of $\tau$ and the projection of~$\tau$ on $\Gal(kM/k)$, respectively. For $C$ a twist of $C^0$, the following hold:
\begin{enumerate}[\rm (i)]
\item The pair $(s,t)$ is one of the $9$ pairs listed on Table \ref{table: Coefficient transformations Fermat} if $C^0=C^0_1$, and one of the $9$ pairs listed on Table  \ref{table: Coefficient transformations Klein} if $C^0=C^0_7$.
\item For each pair $(s,t)$,  let
$$
F_{(s,t)}:=F_{1}\times F_{2}\times F_{3}\colon [-2,2]\rightarrow [-6,6]\times [-1,15]\times [-20,20]\subseteq \R^3
$$
be the map defined in Table \ref{table: Coefficient transformations Fermat} if $C_0=C_0^1$ or Table  \ref{table: Coefficient transformations Klein} if $C_0=C_0^7$.
For every prime $\p$ unramified in $K$ of good reduction for both $\Jac(C)$ and $E^0$ we have
\begin{equation}\label{equation: functions}
F_{(f_L(\p),f_{kM}(\p))}(a_1(E^0)(\p))=\bigl( a_1(\Jac(C))( \p), a_2(\Jac(C))(\p), a_3(\Jac(C))(\p) \bigr),
\end{equation}
 where $f_L(\p)$ (resp.\ $f_{kM}(\p)$) is the residue degree of $\p$ in $L$ (resp.\ $kM$).
\end{enumerate}

\begin{table}[htb]
\setlength{\extrarowheight}{1.5pt}
\begin{center}
\caption{For each possible pair $(s,t)$, the corresponding values of $F_{(s,t)}(x)$ if $C^0=C^0_1$. In the table below, $y$ denotes $\pm\sqrt{4-x^2}$.}
\label{table: Coefficient transformations Fermat}

\begin{tabular}{rrrr}
$(s,t)$ & $F_{1}(x)$ & $F_{2}(x)$ & $F_{3}(x)$ \\\toprule
$(1,1)$ & \begin{tabular}{r} $3x$ \end{tabular}     & \begin{tabular}{r}$3x^2+3$          \end{tabular} & \begin{tabular}{r}$x^3+6x$ \end{tabular} \\

$(2,1)$& or $\begin{cases} \\  \end{cases}$ \begin{tabular}{r} $-x$\\ $x$\end{tabular}                 & \begin{tabular}{r} $-x^2+3$\\ $-x^2+3$ \end{tabular}   & \begin{tabular}{r} $x^3-2x$ \\ $-x^3+2x$ \end{tabular}  \\

$(3,1)$   & \begin{tabular}{r} $0$ \end{tabular}     & \begin{tabular}{r}$0$          \end{tabular} & \begin{tabular}{r}$x^3-3x$ \end{tabular} \\

$(4,1)$ &or $\begin{cases} \\ \\  \end{cases} $ \begin{tabular}{r} $-x+2y $\\ $x$\\$-x$\end{tabular}                 & \begin{tabular}{r} $-x^2+7-2xy$ \\ $x^2-1$\\ $x^2-1$ \end{tabular}   & \begin{tabular}{r} $x^3-6x+4y$ \\ $x^3-2x$\\ $-x^3+2x$ \end{tabular}  \\

$(8,1)$   &  \begin{tabular}{r}  $y$\end{tabular}     & \begin{tabular}{r} $-xy+1$          \end{tabular} & \begin{tabular}{r} $x^3-4x$ \end{tabular} \\

$(2,2)$ & \begin{tabular}{r}$0$ \end{tabular}  &  \begin{tabular}{r}$3$ \end{tabular} & \begin{tabular}{r}$0$ \end{tabular} \\

$(4,2)$ & \begin{tabular}{r}$0$ \end{tabular}  &\begin{tabular}{r} $-1$ \end{tabular}  & \begin{tabular}{r}$0$ \end{tabular} \\

$(6,2)$ &\begin{tabular}{r} $0$ \end{tabular}  & \begin{tabular}{r}$0$ \end{tabular}  & \begin{tabular}{r}$0$  \end{tabular}\\

$(8,2)$ & \begin{tabular}{r}$0$ \end{tabular}  & \begin{tabular}{r}$1$ \end{tabular}  & \begin{tabular}{r}$0$ \end{tabular} \\\bottomrule
\end{tabular}
\end{center}
\end{table}

\begin{table}[htb]
\setlength{\extrarowheight}{1.5pt}
\begin{center}
\caption{For each possible pair $(s,t)$ the corresponding values of $F_{(s,t)}(x)$ if $C^0=C^0_7$. In the table below, $y$ denotes $\pm\sqrt{(4-x^2)/7}$.}
\label{table: Coefficient transformations Klein}
\begin{tabular}{rrrr}
$(s,t)$ & $F_{1}(x)$ & $F_{2}(x)$ & $F_{3}(x)$ \\\toprule
$(1,1)$ & \begin{tabular}{r} $3x$ \end{tabular}     & \begin{tabular}{r}$3x^2+3$          \end{tabular} & \begin{tabular}{r}$x^3+6x$ \end{tabular} \\

$(2,1)$&  \begin{tabular}{r} $-x$\end{tabular}     & \begin{tabular}{r} $-x^2+3$ \end{tabular}   & \begin{tabular}{r} $x^3-2x$  \end{tabular}  \\

$(3,1)$   & \begin{tabular}{r} $0$ \end{tabular}     & \begin{tabular}{r}$0$          \end{tabular} & \begin{tabular}{r}$x^3-3x$ \end{tabular} \\

$(4,1)$   & \begin{tabular}{r} $x$ \end{tabular}     & \begin{tabular}{r}$x^2-1$    \end{tabular} & \begin{tabular}{r}$x^3-2x$ \end{tabular} \\

$(7,1)$ & \begin{tabular}{r} $-\frac{x}{2}+ \frac{ 7}{2}y$\end{tabular}                 & \begin{tabular}{r} $-\frac{x^2}{2}+3-\frac{ 7}{2}xy$ \end{tabular}   & \begin{tabular}{r} $x^3-\frac{9}{2}x+\frac{7}{2}y$ \end{tabular}  \\

$(2,2)$ & \begin{tabular}{r}$0$ \end{tabular}  &  \begin{tabular}{r}$3$ \end{tabular} & \begin{tabular}{r}$0$ \end{tabular} \\

$(4,2)$ & \begin{tabular}{r}$0$ \end{tabular}  &\begin{tabular}{r} $-1$ \end{tabular}  & \begin{tabular}{r}$0$ \end{tabular} \\

$(6,2)$ &\begin{tabular}{r} $0$ \end{tabular}  & \begin{tabular}{r}$0$ \end{tabular}  & \begin{tabular}{r}$0$  \end{tabular}\\

$(8,2)$ & \begin{tabular}{r}$0$ \end{tabular}  & \begin{tabular}{r}$1$ \end{tabular}  & \begin{tabular}{r}$0$ \end{tabular} \\\bottomrule
\end{tabular}
\end{center}
\end{table}
\end{proposition}

\begin{proof} For every prime $\p$ unramified in $K$ of good reduction for both $\Jac(C)$ and $E^0$, write $x_\p:=a_1(E^0)(\p)$ and $y_\p:=\pm\sqrt{4-x_\p^2}$. It follows from the proof of Proposition \ref{proposition: import} that
\begin{align*}
a_1(\Jac(C))(\p) & = -\Rr(a_1(\p))x_\p+\Ii( a_1(\p))y_\p\,,\\[4pt]
a_2(\Jac(C)(\p)) & =\Rr(a_2(\p))(x_\p^2-2)- \Ii(a_2(\p))x_\p y_\p+|a_1(\p)|^2\,,\\[4pt] 
a_3(\Jac(C)(\p)) & = -a_3(\p)(x_\p^3-3x_\p)-\Rr(\overline a_1(\p)a_2(\p))x_\p+ \Ii(\overline a_1(\p)a_2(\p))y_\p \,,
\end{align*}
from which one can easily derive the assertion of the proposition.
\end{proof}

Tables~\ref{table: Fermat statistics} and~\ref{table: Klein statistics} show Sato-Tate statistics for the Fermat and Klein twists $C/k$ and their base changes $C_{kM}$.
In each row we list moment statistics $\overline\M_{101}, \overline\M_{030}, \overline\M_{202}$ for the three moments $\M_{101}, \M_{030}, \M_{202}$ that uniquely determine the Sato-Tate group $\ST(C)$, by Proposition~\ref{proposition: lowernumb}.
These were computed by averaging over all good primes of degree one and norm $p\le 2^{26}$.

For comparison we also list the actual value of each moment, computed using the method described in \S\ref{sec: joint moments}.  In every case the moment statistics agree with the corresponding moments of the Sato-Tate groups to within 1.5 percent, and in almost all cases, to within 0.5 percent.

\section{Tables}\label{section: tables}

In this final section we present tables of characters, curves, Sato-Tate distributions, and moment statistics referred to elsewhere in this article.
Let us briefly describe their contents.

Table~\ref{table: charactertable} lists characters of the automorphism groups of the Fermat and Klein quartics specified via conjugacy class representatives expressed using the generators $s,t,u$ defined in \eqref{equation: autos1} and \eqref{equation: autos7}.

Tables \ref{table: example curves Fermat}--\ref{table: example curves Klein} list explicit curve equations for twists $C$ of the Fermat and Klein quartics corresponding to subgroups~$H$ of 
$G_{C^0}\coloneqq \Aut(C^0_{M})\rtimes \Gal(M/\Q)$, as described in Remark~\ref{remark: pairs}.
The group $H_0\coloneqq H\cap \Aut(C^0_{M})$ is specified in terms of the generators $s,t,u$ listed in~\eqref{equation: autos1} and \eqref{equation: autos7}.
When $kM=k$ we have $H=H_0$, and otherwise $H$ is specified by listing an element $h\in \Aut(C_M^0)$ for which $H=H_0\cup H_0\cdot(h,\tau)$, where $\Gal(M/\Q)=\langle\tau\rangle$; see \eqref{equation: descH}.
The isomorphism classes of~$H$ and $H_0$ are specified by GAP identifiers $\mathrm{ID}(H)$ and $\mathrm{ID}(H_0)$.
The minimal field $K$ over which all endomorphisms of $\Jac(C_\Qbar)$ are defined is given as an explicit extension of~$\Q$, or as the splitting field $\Gal(f(x))$ of a monic $f\in \Z[x]$.
%; recall that we always have $\Gal(K/k)\simeq H$ and $\Gal(K/kM)\simeq H_0$.
In the last 2 columns of Tables~ \ref{table: example curves Fermat} and \ref{table: example curves Klein} we identify the Sato-Tate distributions of $\ST(C)$ and $\ST(C_{kM})$ by their row numbers in Table~\ref{table: distributions}.
Among twists with the same Sato-Tate group $\ST(C)$ (which is uniquely identified by its distribution), we list curves with isogenous Jacobians, per Remark~\ref{remark: isojac}.

In Table~\ref{table: distributions} we list the 60 Sato-Tate distributions that arise among twists of the Fermat and Klein quartics.
Each component group is identified by its GAP ID, and we list the joint moments $\M_{101}$, $\M_{030}$, $\M_{202}$ sufficient to uniquely determine the Sato-Tate distribution, along with the first two non-trivial independent coefficient moments for $a_1,a_2,a_3$.
We also list the proportion $z_{i,j}$ of components on which the coefficient $a_i$ takes the fixed integer value $j$; for $i=1,3$ we list only $z_1\coloneqq z_{1,0}$ and $z_3\coloneqq z_{3,0}$, and for $i=2$ we list the vector $z_2\coloneqq [z_{2,-1},z_{2,0},z_{2,1},z_{2,2},z_{2,3}]$; see Lemma~\ref{lemma: zvalues} for details.
There are 6 pairs of Sato-Tate distributions whose independent coefficient measures coincide; these pairs are identified by roman letters that appear in the last column.

Tables~\ref{table: Fermat statistics}--\ref{table: Klein statistics} list moment statistics for twists of the Fermat and Klein quartics computed over good primes $p\le 2^{26}$, along with the corresponding moment values.  Twists with isogenous Jacobians necessarily have the same moment statistics, so we list only one twist in each isogeny class.
\bigskip
\bigskip
\bigskip
\bigskip
\bigskip

\small
\begin{table}[htb!]
\setlength{\tabcolsep}{3pt}
\caption{Character tables of $\Aut(C^0_M)$.  See~\eqref{equation: autos1}~and~\eqref{equation: autos7} for a description of the generators $s,t,u$.}\label{table: charactertable}
\begin{subtable}{.5\linewidth}
\caption{$\Aut((C_1^0)_M)\simeq \langle 96,64\rangle$}
\label{table: charactertable1}
\centering
\begin{tabular}{lrrrrrrrrrr}
Class &  $1a$  & $2a$ & $2b$ &  $3a$ &  $4a$ & $4b$ &  $4c$ & $4d$ & $8a$ & $8b$\\\hline
Repr. &  $1$  & $u^2$ & $u^2t$ &  $s$ &  $u$ & $u^3$ &  $tut$ & $t$ & $tu$ & $tu^3$\\\hline
Order  & 1 &  2 & 2 & 3 &  4 & 4 &  4 & 4 & 8 & 8\\\hline
Size  & 1 &  3 & 12 & 32 &  3 & 3 &  6 & 12 & 12 & 12\\\hline
$\chi_1$   &   1 & 1 & 1 & 1 &  1&   1 & 1 & 1 & 1 & 1\\
$\chi_2$  & 1 & 1 & -1 & 1  &   1  &   1 & 1 &-1& -1& -1\\
$\chi_3$ &  2 & 2 & 0 &-1 &    2  &   2 & 2 & 0 & 0 & 0\\
$\chi_4$ &  3 & 3 & -1 & 0  &  -1  &  -1 &-1& -1 & 1 & 1\\
$\chi_5$  & 3 & 3 & 1 & 0  &  -1  &  -1 &-1 & 1 &-1 &-1\\
$\chi_6$  &  3 & -1 &  1&  0&$-1-2i$ &$-1+2i$ &  1 &-1 & $i$& $-i$\\
$\chi_7$  &   3 &-1 & 1 & 0&$-1+2i$& $-1-2i$ &  1 &-1& $-i$&  $i$\\
$\chi_8$   &   3 &-1 &-1 & 0 &$-1+2i$& $-1-2i$ &  1 & 1 & $i$ & $-i$\\
$\chi_9$  &   3 &-1 &-1 & 0 &$-1-2i$ & $-1+2i$ &  1 & 1 &$-i$ &  $i$\\
$\chi_{10}$ &   6 &-2 & 0 & 0  &   2   &  2 &-2 & 0 & 0  &0\\\hline
\end{tabular}
\end{subtable}
\hfill
\begin{subtable}{.4\linewidth}
\caption{$\Aut((C_7^0)_M)\simeq \langle 168,42\rangle$}
\label{table: charactertable7}
\centering
\begin{tabular}{lrrrrrr}
Class & $1a$ & $2a$ & $3a$ & $4a$  &  $7a$  &  $7b$\\\hline
Repr. & $1$ & $t$ & $s$ & $u^2tu^3tu^2$  &  $u$  &  $u^3$\\\hline
Order & 1 & 2 & 3 & 4  &  7  &  7\\\hline
Size  &   1 & 21 & 56 & 42 &  24 &  24\\\hline
$\chi_1$  &   1 & 1 & 1 & 1 &   1 &   1\\
$\chi_2$  &   3 & -1 &  0  & 1 &  $a$ &  $\overline a$\\
$\chi_3$  &  3 &-1 & 0 & 1 & $\overline a$ &  $a$\\
$\chi_4$  &   6  & 2 & 0 & 0  & -1 &  -1\\
$\chi_5$  &  7 & -1 & 1 &-1 &   0  &  0\\
$\chi_6$   & 8 & 0 &-1 & 0  &  1  &  1\\\hline
\end{tabular}
\end{subtable}
\end{table}
\normalsize

\begin{table}
\begin{center}
\small
\setlength{\extrarowheight}{2.8pt}
\caption{Twists of the Fermat quartic corresponding to subgroups $H \subseteq G_{C^0_1}$.
See \eqref{equation: autos1} for the definitions of $s,t,u$ and \eqref{equation: descH} for the definition of $h$.
We identify $\ST_k=\ST(C)$ and $\ST_{kM}=\ST(C_{kM})$ by row numbers in Table~\ref{table: distributions}; here $M=\Q(i)$.}\label{table: example curves Fermat}
\vspace{5pt}
\begin{tabular}{|r|c|c|c|c|c|l|c|c|}\hline
\# & Gen$(H_0)$ & $h$ & ID$(H)$ & ID$(H_0)$ & $k$ & $K$ & $\ST_k$ & $\ST_{kM}$\\\hline

1 & id & id & $\langle 2,1 \rangle$ & $\langle 1,1 \rangle$ & $\Q$ & $\Q(i)$ & 3 & 1\\\hline % 1
$\boldsymbol{C_1^0}$&\multicolumn{8}{l|}{$x^4+y^4+z^4$}\\\hline

2 & id & $u^2t$ & $\langle 2,1 \rangle$ & $\langle 1,1 \rangle$ & $\Q$ & $\Q(i)$ & 3 & 1\\\hline % 2
&\multicolumn{8}{l|}{$x^4-6x^2y^2+y^4-2z^4$}\\\hline

3 & id & $t^3utu$ & $\langle 2,1 \rangle$ & $\langle 1,1 \rangle$ & $\Q$ & $\Q(i)$ & 3 & 1\\\hline % 3
&\multicolumn{8}{l|}{$4x^4-y^4-z^4$}\\\hline

4 & $t^2$ & $t$ & $\langle 4,1 \rangle$ & $\langle 2,1 \rangle$ & $\Q(\sqrt{-5})$ & $\Gal(x^4\!-\!x^2\!-\!1)$ & 5 & 2\\\hline % 4
&\multicolumn{8}{l|}{$12x^4+40x^3y-100xy^3-75y^4-2z^4$}\\\hline

5 & $t^2$ & $t^3utu$ & $\langle 4,2 \rangle$ & $\langle 2,1 \rangle$ & $\Q$ & $\Q(\sqrt{3},i)$ & 9 & 2\\\hline % 5
&\multicolumn{8}{l|}{$9x^4+9y^4-4z^4$}\\\hline

6 & $t^2$ & $u^2t$ & $\langle 4,2 \rangle$ & $\langle 2,1 \rangle$ & $\Q$ & $\Q(\sqrt{3},i)$ & 9 & 2\\\hline % 6
&\multicolumn{8}{l|}{$9x^4-54x^2y^2+9y^4-2z^4$}\\\hline

7 & $t^2$ & id & $\langle 4,2 \rangle$ & $\langle 2,1 \rangle$ & $\Q$ & $\Q(\sqrt{3},i)$ & 9 & 2\\\hline % 7
&\multicolumn{8}{l|}{$9x^4+y^4+z^4$}\\\hline

8 & $t^2$ & $u$ & $\langle 4,2 \rangle$ & $\langle 2,1 \rangle$ & $\Q$ & $\Q(\sqrt{3},i)$ & 9 & 2\\\hline % 8
&\multicolumn{8}{l|}{$9x^4-4y^4+z^4$}\\\hline

9 & $u^2t$ & id & $\langle 4,2 \rangle$ & $\langle 2,1 \rangle$ & $\Q$ & $\Q(\sqrt{3},i)$ & 9 & 2\\\hline % 9
&\multicolumn{8}{l|}{$2x^4+36x^2y^2+18y^4+z^4$}\\\hline

10 & $u^2t$ & $t^3utu$ & $\langle 4,2 \rangle$ & $\langle 2,1 \rangle$ & $\Q$ & $\Q(\sqrt{3},i)$ & 9 & 2\\\hline % 10
&\multicolumn{8}{l|}{$9x^4+18x^2y^2+y^4-2z^4$}\\\hline

11 & $s$ & $u^2t$ & $\langle 6,1 \rangle$ & $\langle 3,1 \rangle$ & $\Q$ & $\Gal(x^3\!-\!3x\!-\!4)$ & 11 & 4\\\hline % 12
&\multicolumn{8}{l|}{$x^4+4x^3y+12x^2y^2-12x^2yz-6x^2z^2+36xyz^2+6y^4-36y^2z^2-12yz^3+9z^4$}\\\hline

12 & $s$ & id & $\langle 6,2 \rangle$ & $\langle 3,1 \rangle$ & $\Q$ & $\Gal(x^6\!+\!5x^4\!+\!6x^2\!+\!1)$ & 12 & 4\\\hline % 11
&\multicolumn{8}{l|}{$5x^4+8x^3y-4x^3z+6x^2y^2+12x^2z^2+12xyz^2+4xz^3+2y^4+4y^3z+6y^2z^2+4yz^3+2z^4$}\\\hline

13 & $t^3utu^3$ & $tu$ & $\langle 8,1 \rangle$ & $\langle 4,1 \rangle$ & $\Q(\sqrt{-5})$ & $\Gal(x^8\!-\!2x^4\!-\!4)$ & 14 & 6\\\hline % 30
&\multicolumn{8}{l|}{$x^4-10x^3z+30x^2z^2-2y^4-100z^4$}\\\hline

14 & $t$ & $t^3utu$ & $\langle 8,2 \rangle$ & $\langle 4,1 \rangle$ & $\Q$ & $\Gal(x^8\!+\!15x^4\!+\!25)$ & 16 & 6\\\hline % 24
&\multicolumn{8}{l|}{$3x^4-4x^3y+12x^2y^2+4xy^3+3y^4-5z^4$}\\\hline

15 & $t^3utu^3$ & $u^2t$ & $\langle 8,2 \rangle$ & $\langle 4,1 \rangle$ & $\Q$ & $\Gal(x^8\!+\!15x^4\!+\!25)$ & 16 & 6\\\hline % 27
&\multicolumn{8}{l|}{$12x^4+40x^3y-100xy^3-75y^4+10z^4$}\\\hline

16 & $t$ & id & $\langle 8,2 \rangle$ & $\langle 4,1 \rangle$ & $\Q$ & $\Gal(x^8\!+\!15x^4\!+\!25)$ & 16 & 6\\\hline % 29
&\multicolumn{8}{l|}{$3x^4+4x^3y+12x^2y^2-4xy^3+3y^4+20z^4$}\\\hline

17 & $u^2t,t^2$ & $t^3utu^3$ & $\langle 8,3 \rangle$ & $\langle 4,2 \rangle$ & $\Q$ & $\Gal(x^4\!-\!6x^2\!+\!10)$ & 19 & 8\\\hline % 21
&\multicolumn{8}{l|}{$11x^4+12x^3y+54x^2y^2-12xy^3+11y^4-2z^4$}\\\hline

18 & $t^2,u^2$ & $u^2t$ & $\langle 8,3 \rangle$ & $\langle 4,2 \rangle$ & $\Q$ & $\Gal(x^4\!-\!6x^2\!+\!10)$ & 19 & 8\\\hline % 22
&\multicolumn{8}{l|}{$x^4+5x^3y-25xy^3-25y^4+z^4$}\\\hline

19 & $u^2t,t^2$ & $u^2$ & $\langle 8,3 \rangle$ & $\langle 4,2 \rangle$ & $\Q$ & $\Gal(x^4\!-\!6x^2\!+\!10)$ & 19 & 8\\\hline % 28
&\multicolumn{8}{l|}{$19x^4-12x^3y+6x^2y^2+12xy^3+19y^4+2z^4$}\\\hline

20 & $t$ & $u^2$ & $\langle 8,3 \rangle$ & $\langle 4,1 \rangle$ & $\Q$ & $\Gal(x^4\!-\!6x^2\!+\!12)$ & 20 & 6\\\hline % 18
&\multicolumn{8}{l|}{$9x^4-18x^2y^2-12xy^3-2y^4+12z^4$}\\\hline

\end{tabular}
\end{center}
\end{table}

\begin{table}
{\textsc{Table 4}.  Twists of the Fermat quartic (page 2 of 3).}
\bigskip
\bigskip

\begin{center}
\small
\setlength{\extrarowheight}{3.3pt}
\begin{tabular}{|r|c|c|c|c|c|l|c|c|}\hline
\# & Gen$(H_0)$ & $h$ & ID$(H)$ & ID$(H_0)$ & $k$ & $K$ & $\ST_k$ & $\ST_{kM}$\\\hline

21 & $t$ & $t^3utu^3$ & $\langle 8,3 \rangle$ & $\langle 4,1 \rangle$ & $\Q$ & $\Gal(x^4\!-\!6x^2\!+\!12)$ & 20 & 6\\\hline % 19
&\multicolumn{8}{l|}{$9x^4-18x^2y^2-12xy^3-2y^4-3z^4$}\\\hline

22 & $t^3utu^3$ & $u$ & $\langle 8,3 \rangle$ & $\langle 4,1 \rangle$ & $\Q$ & $\Gal(x^4\!-\!6x^2\!+\!12)$ & 20 & 6\\\hline % 25
&\multicolumn{8}{l|}{$9x^4+3y^4-4z^4$}\\\hline

23 & $t^3utu^3$ & id & $\langle 8,3 \rangle$ & $\langle 4,1 \rangle$ & $\Q$ & $\Gal(x^4\!-\!6x^2\!+\!12)$ & 20 & 6\\\hline % 26
&\multicolumn{8}{l|}{$9x^4+3y^4+z^4$}\\\hline

24 & $t^3utu$ & $u^2t$ & $\langle 8,3 \rangle$ & $\langle 4,1 \rangle$ & $\Q$ & $\Gal(x^4\!-\!6x^2\!+\!12)$ & 21 & 7\\\hline % 16
&\multicolumn{8}{l|}{$x^4-6x^2y^2+y^4-6z^4$}\\\hline

25 & $t^3utu$ & $u$ & $\langle 8,3 \rangle$ & $\langle 4,1 \rangle$ & $\Q$ & $\Gal(x^4\!-\!6x^2\!+\!12)$ & 21 & 7\\\hline % 20
&\multicolumn{8}{l|}{$3x^4-4y^4+z^4$}\\\hline

26 & $t^3utu$ & id & $\langle 8,3 \rangle$ & $\langle 4,1 \rangle$ & $\Q$ & $\Gal(x^4\!-\!6x^2\!+\!12)$ & 21 & 7\\\hline % 31
&\multicolumn{8}{l|}{$3x^4+y^4+z^4$}\\\hline

27 & $t^3utu$ & $t$ & $\langle 8,4 \rangle$ & $\langle 4,1 \rangle$ & $\Q(\sqrt{-2})$ & $\Gal(x^8\!+\!9)$ & 23 & 7\\\hline % 14
&\multicolumn{8}{l|}{$3x^3y-3xy^3-2z^4$}\\\hline

28 & $t^2,u^2$ & id & $\langle 8,5 \rangle$ & $\langle 4,2 \rangle$ & $\Q$ & $\Q(\sqrt{3},\sqrt{5},i)$ & 24 & 8\\\hline % 13
&\multicolumn{8}{l|}{$9x^4+25y^4+z^4$}\\\hline

29 & $t^2,u^2$ & $u$ & $\langle 8,5 \rangle$ & $\langle 4,2 \rangle$ & $\Q$ & $\Q(\sqrt{3},\sqrt{5},i)$ & 24 & 8\\\hline % 15
&\multicolumn{8}{l|}{$9x^4+25y^4-4z^4$}\\\hline

30 & $u^2t,t^2$ & $t^3utu$ & $\langle 8,5 \rangle$ & $\langle 4,2 \rangle$ & $\Q$ & $\Q(\sqrt{3},\sqrt{5},i)$ & 24 & 8\\\hline % 17
&\multicolumn{8}{l|}{$x^4+30x^2y^2+25y^4-18z^4$}\\\hline

31 & $u^2t,t^2$ & id & $\langle 8,5 \rangle$ & $\langle 4,2 \rangle$ & $\Q$ & $\Q(\sqrt{3},\sqrt{5},i)$ & 24 & 8\\\hline % 23
&\multicolumn{8}{l|}{$2x^4+60x^2y^2+50y^4+9z^4$}\\\hline

32 & $s,u^2t$ & id & $\langle 12,4 \rangle$ & $\langle 6,1 \rangle$ & $\Q$ & $\Gal(x^6\!+\!2x^3\!+\!2)$ & 26 & 10\\\hline % 32
&\multicolumn{8}{l|}{$4x^3y-3x^2z^2+12xy^2z-2y^4-2yz^3$}\\\hline

33 & $t^3utu,u^2$ & $ut$ & $\langle 16,6 \rangle$ & $\langle 8,2 \rangle$ & $\Q(\sqrt{-5})$ & $\Gal(x^8\!-\!2x^4\!+\!5)$ & 29 & 17\\\hline % 33
&\multicolumn{8}{l|}{$x^4-30x^2y^2-80xy^3-55y^4-2z^4$}\\\hline

34 & $t^3utu^3,u^2t$ & $u$ & $\langle 16,7 \rangle$ & $\langle 8,3 \rangle$ & $\Q$ & $\Gal(x^8\!-\!6x^4\!-\!8x^2\!-\!1)$ & 31 & 18\\\hline % 38
&\multicolumn{8}{l|}{$x^4-12x^2y^2-32xy^3-28y^4+z^4$}\\\hline

35 & $tu^2tut$ & $tutu^2$ & $\langle 16,7 \rangle$ & $\langle 8,1 \rangle$ & $\Q$ & $\Gal(x^8\!-\!8x^4\!-\!2)$ & 32 & 15\\\hline % 41
&\multicolumn{8}{l|}{$2x^3y-xy^3-z^4$}\\\hline

36 & $u^2tu$ & $u$ & $\langle 16,8 \rangle$ & $\langle 8,1 \rangle$ & $\Q$ & $\Gal(x^8\!-\!2)$ & 34 & 15\\\hline % 37
&\multicolumn{8}{l|}{$x^3y+2xy^3+z^4$}\\\hline

37 & $t^3utu^3,t$ & $u$ & $\langle 16,8 \rangle$ & $\langle 8,4 \rangle$ & $\Q$ & $\Gal(x^8\!-\!10x^4\!-\!100)$ & 33 & 22\\\hline % 45
&\multicolumn{8}{l|}{$x^4+10x^3y+30x^2y^2-100y^4-10z^4$}\\\hline

38 & $t,u^2$ & $utu$ & $\langle 16,11 \rangle$ & $\langle 8,3 \rangle$ & $\Q$ & $\Gal(x^8\!-\!2x^4\!+\!9)$ & 35 & 18\\\hline % 40
&\multicolumn{8}{l|}{$4x^4+4x^3y+6x^2y^2-2xy^3+y^4-2z^4$}\\\hline

39 & $t,u^2$ & id & $\langle 16,11 \rangle$ & $\langle 8,3 \rangle$ & $\Q$ & $\Gal(x^8\!-\!2x^4\!+\!9)$ & 35 & 18\\\hline % 43
&\multicolumn{8}{l|}{$4x^4+4x^3y+6x^2y^2-2xy^3+y^4+2z^4$}\\\hline

40 & $t^3utu^3,u^2t$ & id & $\langle 16,11 \rangle$ & $\langle 8,3 \rangle$ & $\Q$ & $\Gal(x^8\!-\!2x^4\!+\!9)$ & 35 & 18\\\hline % 44
&\multicolumn{8}{l|}{$5x^4-8x^3y+12x^2y^2+16xy^3+20y^4+2z^4$}\\\hline

\end{tabular}
\end{center}
\end{table}

\begin{table}
\begin{center}
{\textsc{Table 4}.  Twists of the Fermat quartic (page 3 of 3).}
\bigskip
\bigskip

\small
\setlength{\extrarowheight}{4pt}
\begin{tabular}{|r|c|c|c|c|c|l|c|c|}\hline
\# & Gen$(H_0)$ & $h$ & ID$(H)$ &\hspace{-3pt} ID$(H_0)$\hspace{-4pt}\,  & $k$ & $K$ & $\ST_k$ & $\ST_{kM}$\\\hline

41 & $t^3utu,u^2$ & id & $\langle 16,11 \rangle$ & $\langle 8,2 \rangle$ & $\Q$ & $\Gal(x^8\!+\!5x^4\!+\!25)$ & 36 & 17\\\hline % 36
&\multicolumn{8}{l|}{$9x^4+5y^4+z^4$}\\\hline

42 & $t^3utu,u^2$ & $u$ & $\langle 16,11 \rangle$ & $\langle 8,2 \rangle$ & $\Q$ & $\Gal(x^8\!+\!5x^4\!+\!25)$ & 36 & 17\\\hline % 34
&\multicolumn{8}{l|}{$9x^4+5y^4-4z^4$}\\\hline

43 & $utu,t^3$ & $u^2$ & $\langle 16,11 \rangle$ & $\langle 8,2 \rangle$ & $\Q$ & $\Gal(x^8\!+\!5x^4\!+\!25)$ & 36 & 17\\\hline % 42
&\multicolumn{8}{l|}{$7x^4+8x^3y+6x^2y^2+8xy^3+7y^4+10z^4$}\\\hline

44 & $t^3utu,u^2$ & $u^2t$ & $\langle 16,13 \rangle$ & $\langle 8,2 \rangle$ & $\Q$ & $\Gal(x^8\!-\!8x^4\!+\!25)$ & 39 & 17\\\hline % 35
&\multicolumn{8}{l|}{$x^4+5x^3y-25xy^3-25y^4+2z^4$}\\\hline

45 & $t,utu$ & id & $\langle 16,13 \rangle$ & $\langle 8,2 \rangle$ & $\Q$ & $\Gal(x^8\!-\!8x^4\!+\!25)$ & 39 & 17\\\hline % 46
&\multicolumn{8}{l|}{$19x^4+32x^3y+21x^2y^2+8xy^3+3y^4+2y^3z+6y^2z^2+8yz^3+4z^4$}\\\hline

46 & $t^3utu^3,t$ & id & $\langle 16,13 \rangle$ & $\langle 8,4 \rangle$ & $\Q$ & $\Gal(x^8\!+\!12x^4\!+\!9)$ & 37 & 22\\\hline % 39
&\multicolumn{8}{l|}{$x^4-2x^3y+6x^2y^2+4xy^3+4y^4+3z^4$}\\\hline

47 & $s,u^2$ & $u^2t$ & $\langle 24,12 \rangle$ & $\langle 12,3 \rangle$ & $\Q$ & $\Gal(x^4\!-\!16x\!-\!24)$ & 42 & 25\\\hline % 48
&\multicolumn{8}{l|}{$x^4-3x^3z-12x^2yz+16xy^3-xz^3+9y^4+12y^3z+6y^2z^2$}\\\hline

48 & $s,u^2$ & id & $\langle 24,13 \rangle$ & $\langle 12,3 \rangle$ & $\Q$ & $\Gal(x^6\!-\!x^4\!-\!2x^2\!+\!1)$ & 43 & 25\\\hline % 47
&\multicolumn{8}{l|}{$3x^4+4x^3y+4x^3z+6x^2y^2+6x^2z^2+8xy^3+12xyz^2+5y^4+4y^3z+12y^2z^2+z^4$}\\\hline

49 & $tu^2tut,u^2$ & id & $\langle 32,7 \rangle$ & $\langle 16,6 \rangle$ & $\Q$ & $\Gal(x^8\!-\!10x^4\!+\!20)$ & 44 & 30\\\hline % 50
&\multicolumn{8}{l|}{$4x^4-8x^3y+12x^2y^2+2y^4+5z^4$}\\\hline

50 & $u,u^2tu^3t$ & $u^2t$ & $\langle 32,11 \rangle$ & $\langle 16,2 \rangle$ & $\Q$ & $\Gal(x^8\!-\!2x^4\!+\!5)$ & 45 & 28\\\hline % 52
&\multicolumn{8}{l|}{$x^4-30x^2y^2-80xy^3-55y^4-2z^4$}\\\hline

51 & $u,t^3ut$ & id & $\langle 32,34 \rangle$ & $\langle 16,2 \rangle$ & $\Q$ & $\Gal(x^{16}\!-\!4x^{12}\!+\!6x^8\!+\!20x^4\!+\!1)$ & 47 & 28\\\hline % 49
&\multicolumn{8}{l|}{$2x^4+3y^4+z^4$}\\\hline

52 & $t^3utu,u^2,t$ & $u$ & $\langle 32,43 \rangle$ & $\langle 16,13 \rangle$ & $\Q$ & $\Gal(x^8\!+\!6x^4\!-\!9)$ & 48 & 38\\\hline % 51
&\multicolumn{8}{l|}{$3x^4-36x^2y^2-96xy^3-84y^4+2z^4$}\\\hline

53 & $tu^2tut,u^2$ & $u$ & $\langle 32,43 \rangle$ & $\langle 16,6 \rangle$ & $\Q$ & $\Gal(x^8\!-\!10x^4\!+\!45)$ & 49 & 30\\\hline % 53
&\multicolumn{8}{l|}{$9x^4+36x^3y-24xy^3-4y^4-10z^4$}\\\hline

54 & $utu,u^2,t$ & id & $\langle 32,49 \rangle$ & $\langle 16,13 \rangle$ & $\Q$ & $\Gal(x^8\!+\!8x^4\!+\!9)$ & 50 & 38\\\hline % 54
&\multicolumn{8}{l|}{$x^4+x^3y+24x^2y^2+67xy^3+79y^4+2z^4$}\\\hline

55 & $s,t$ & id & $\langle 48,48 \rangle$ & $\langle 24,12 \rangle$ & $\Q$ & $\Gal(x^6\!-\!x^4\!+\!5x^2\!+\!1)$ & 53 & 41\\\hline % 55
&\multicolumn{8}{l|}{$3x^4+2x^3z+6x^2yz+12xy^3-30xy^2z+2xz^3-27y^4+38y^3z+18y^2z^2-10yz^3$}\\\hline

56 & $t,u$ & id & $\langle 64,134 \rangle$ & $\langle 32,11 \rangle$ & $\Q$ & $\Gal(x^8\!-\!4x^4\!-\!14)$ & 54 & 46\\\hline % 56
&\multicolumn{8}{l|}{$x^4-42x^2y^2-168xy^3-203y^4+z^4$}\\\hline

57 & $s,u$ & $u^2t$ & $\langle 96,64 \rangle$ & $\langle 48,3 \rangle$ & $\Q$ & $\Gal(x^{12}\!+\!6x^4\!+\!4)$ & 55 & 52\\\hline % 57
&\multicolumn{8}{l|}{$6x^3z+3x^2y^2+27x^2z^2+6xy^3+18xyz^2+4y^4+2y^3z+18yz^3-36z^4$}\\\hline

58 & $s,u$ & id & $\langle 96,72 \rangle$ & $\langle 48,3 \rangle$ & $\Q$ & $\Gal(x^{12}\!+\!x^8\!-\!2x^4\!-\!1)$ & 57 & 52\\\hline % 58
&\multicolumn{8}{l|}{$x^3y-x^3z+3x^2y^2+28xy^3-84xy^2z+84xyz^2+28y^4-56y^3z+98z^4$}\\\hline

59 & $s,t,u$ & id & $\langle 192,956 \rangle$ & $\langle 96,64 \rangle$ & $\Q$ & $\Gal(x^{12}\!+\!48x^4\!+\!64)$ & 59 & 56\\\hline % 59
&\multicolumn{8}{l|}{$44x^4+120x^3y+36x^3z+60x^2yz+9x^2z^2-200xy^3+xz^3-150y^4-15y^2z^2$}\\\hline

\end{tabular}
\end{center}
\end{table}

\begin{table}
\begin{center}
\small
\caption{Twists of the Klein quartic corresponding to subgroups $H\subseteq G_{C^0_7}$.
See \eqref{equation: autos7} for the definitions of $s,t,u$ and see \eqref{equation: descH} for the definition of~$h$.
We identify $\ST_k=\ST(C)$ and $\ST_{kM}=\ST(C_{kM})$ by row numbers in Table~\ref{table: distributions}; here $M=\Q(\sqrt{-7})$ and $a\coloneqq (-1+\sqrt{-7})/2$.
}\label{table: example curves Klein}
\bigskip

\setlength{\extrarowheight}{4pt}
\begin{tabular}{|c|c|c|c|c|c|l|c|c|}\hline
\# & Gen$(H_0)$ & $h$ & ID$(H)$ &\hspace{-3pt} ID$(H_0)$\hspace{-4pt}\, & $k$ & $K$ &\hspace{-2pt} $\ST_k$\hspace{-3pt}\, & \hspace{-2pt}$\ST_{kM}$\hspace{-5pt}\,\\\hline

1 & id & id & $\langle 2,1 \rangle$ & $\langle 1,1 \rangle$ & $\Q$ & $\Q(a)$ & 3 & 1\\\hline % 1
$\boldsymbol{C_7^0}$&\multicolumn{8}{l|}{$x^4+y^4+z^4+6(xy^3+yz^3+zx^3)-3(x^2y^2+y^2z^2+z^2x^2)+3xyz(x+y+z)$}\\\hline

2 & $t$ & id & $\langle 4,2 \rangle$ & $\langle 2,1 \rangle$ & $\Q$ & $\Q(a,i)$ & 9 & 2\\\hline % 2
&\multicolumn{8}{l|}{$3x^4+28x^3y+105x^2y^2-21x^2z^2+196xy^3+147y^4+147y^2z^2-49z^4$}\\\hline

3 & $ustu^6,sutu^6s^2$ & - & $\langle 4,2 \rangle$ & $\langle 4,2 \rangle$ & $\Q(a)$ & $\Q(\sqrt{2},\sqrt{3},a)$ & 8 & 8\\\hline % 12
&\multicolumn{8}{l|}{$x^4+9ax^2y^2+6ax^2z^2+9y^4+18ay^2z^2+4z^4$}\\\hline

%\textit{4} & $s$ & $t$ & $\langle 6,1 \rangle$ & $\langle 3,1 \rangle$ & $\Q$ & $\Gal(x^3\!-\!3x\!-\!5)$ & 11 & 4\\\hline % 3
%&\multicolumn{8}{l|}{$x^4-6x^3y+7x^3z+3x^2y^2+21x^2yz+3x^2z^2+18xy^3-9xyz^2+9y^4+21y^3z-9y^2z^2-3z^4$}\\\hline
4 & $s$ & $t$ & $\langle 6,1 \rangle$ & $\langle 3,1 \rangle$ & $\Q$ & $\Gal(x^3\!-\!x^2\!+\!2x\!-\!3)$ & 11 & 4\\\hline % 3
&\multicolumn{8}{l|}{$x^4+3x^3y-9x^3z+9x^2y^2-6x^2z^2+18xy^3+3xy^2z-3xyz^2+y^4+4y^3z-3y^2z^2+7yz^3$}\\\hline

5 & $s$ & id & $\langle 6,2 \rangle$ & $\langle 3,1 \rangle$ & $\Q$ & $\Q(\zeta_7)$ & 12 & 4\\\hline % 4
&\multicolumn{8}{l|}{$x^3y+xz^3+y^3z$}\\\hline

%\textit{6} & $u^2tu^3tu^2$ & \hspace{-3pt}$u^5tu^2$\hspace{-4pt}\, & $\langle 8,1 \rangle$ & $\langle 4,1 \rangle$ &\hspace{-4pt} $\Q(\sqrt{2},i)$ \hspace{-6pt}\,& $\Gal(x^8\!+\!7)$ & 14 & 6\\\hline % 6
%&\multicolumn{8}{l|}{$x^4-6x^2y^2+(24d+12)xyz^2-7y^4+(4d+9)z^4$}\\\hline
6 & $u^2tu^3tu^2$ & $u^5tu^2$ & $\langle 8,1 \rangle$ & $\langle 4,1 \rangle$ & $\Q(i)$ & $\Gal(x^8\!+\!2x^7\!-\!14x^4\!+\!16x\!+\!4)$ & 14 & 6\\\hline % 6
&\multicolumn{8}{l|}{$x^4+3x^2y^2-3x^2z^2+2y^3z+3y^2z^2+2yz^3$}\\\hline

7 & $u^2tu^3tu^2$ & id & $\langle 8,3 \rangle$ & $\langle 4,1 \rangle$ & $\Q$ & $\Gal(x^4\!-\!4x^2\!-\!14)$ & 20 & 6\\\hline % 7
&\multicolumn{8}{l|}{$12x^4-80x^3y+60x^2y^2-24x^2z^2-104xy^3+24xyz^2+83y^4+36y^2z^2-2z^4$}\\\hline

8 & $su,tu$ & - & $\langle 12,3 \rangle$ & $\langle 12,3 \rangle$ & $\Q(a)$ & $\Gal(x^6\!-\!147x^2\!+\!343)$ & 25 & 25\\\hline % 13
&\multicolumn{8}{l|}{$3x^4+(-18a+12)x^3y+(12a+4)x^3z+(-27a+36)x^2y^2+(9a+6)x^2z^2+36xy^2z+27y^4$}\\
&\multicolumn{8}{l|}{$+\,(54a-36)y^3z+(-54a+36)y^2z^2+(18a-12)yz^3+(-3a+2)z^4$}\\\hline

%\textit{9} & $s,t$ & id & $\langle 12,4 \rangle$ & $\langle 6,1 \rangle$ & $\Q$ & $\Gal(x^3\!-\!x^2+1)\cdot\Q(a)$ & 26 & 10\\\hline % 8
%&\multicolumn{8}{l|}{$1587x^4-1449x^2y^2+4347x^2yz-483x^2z^2+3381xy^3-3381xyz^2+3381xz^3-441y^4+2646y^3z$}\\
%&\multicolumn{8}{l|}{$-4263y^2z^2+882yz^3-49z^4$}\\\hline
9 & $s,t$ & id & $\langle 12,4 \rangle$ & $\langle 6,1 \rangle$ & $\Q$ & $\Gal(x^3\!-\!x^2\!+\!5x\!+\! 1)\cdot\Q(a)$ & 26 & 10\\\hline % 8
&\multicolumn{8}{l|}{$7x^3z+3x^2y^2-6xyz^2+2y^3z-4z^4$}\\\hline

10 & $u$ & id & $\langle 14,1 \rangle$ & $\langle 7,1 \rangle$ & $\Q$ & $\Gal(x^7\!+\!7x^3\!-\!7x^2\!+\!7x\!+\!1)$ & 27 & 13\\\hline % 5
&\multicolumn{8}{l|}{$x^3y-21x^2z^2+xy^3-42xyz^2-147xz^3+2y^4+21y^3z+63y^2z^2-196z^4$}\\\hline

%\textit{11} & $u^2tu^3tu^2,u^5tu^2$ & id & $\langle 16,7 \rangle$ & $\langle 8,3 \rangle$ & $\Q(\sqrt{2})$ & $\Gal(x^8\!+\!7)$ & 31 & 18\\\hline % 10
%&\multicolumn{8}{l|}{$x^4-6x^2y^2+(24d+12)xyz^2-7y^4+(4d+9)z^4$}\\\hline
11 & $u^2tu^3tu^2,u^5tu^2$ & id & $\langle 16,7 \rangle$ & $\langle 8,3 \rangle$ & $\Q$ & $\Gal(x^8\!+\!2x^7\!-\!14x^4\!+\!16x\!+\!4)$ & 31 & 18\\\hline % 10
&\multicolumn{8}{l|}{$x^4+3x^2y^2-3x^2z^2+2y^3z+3y^2z^2+2yz^3$}\\\hline

12 & $sust,su^6s^2tu^2$ & - & $\langle 24,12 \rangle$ & $\langle 24,12 \rangle$ & $\Q(a)$ & $\Gal(x^4\!+\! 2x^3\!+\!6x^2\!-\!6)\cdot\Q(a)$ & 41 & 41\\\hline % 14
&\multicolumn{8}{l|}{$(3a-2)x^4+(30a-20)x^3y+(90a-60)x^2y^2+(9a+6)x^2z^2+(150a-100)xy^3+60xy^2z$}\\
&\multicolumn{8}{l|}{$+\,(12a+4)xz^3+(150a-25)y^4+(-45a+60)y^2z^2+(30a-20)yz^3+3z^4$}\\\hline

13 & $u,s$ & id & $\langle 42,1 \rangle$ & $\langle 21,1 \rangle$ & $\Q$ & $\Gal(x^7\!-\!2)$ & 51 & 40\\\hline % 9
&\multicolumn{8}{l|}{$2x^3y+xz^3+y^3z$}\\\hline

14 & $t,u,s$ & id & \hspace{-2pt}$\langle 336,208 \rangle$\hspace{-4pt}\, &\hspace{-5pt} $\langle 168,42 \rangle$ \hspace{-7pt}\, & $\Q$ & $\Gal(x^8\!+\!4x^7\!+\!21x^4\!+\!18x\!+\!9)$ & 60 & 58\\\hline % 11
&\multicolumn{8}{l|}{$2x^3y-2x^3z-3x^2z^2-2xy^3-2xz^3-4y^3z+3y^2z^2-yz^3$}\\\hline
\end{tabular}
\end{center}
\end{table}

\begin{table}
\begin{center}
\footnotesize
\setlength{\extrarowheight}{0.1pt}
\caption{The 60 Sato-Tate distributions arising for Fermat and Klein twists.}\label{table: distributions}
\vspace{-12pt}
\begin{tabular}{rlrrrrrrrrrcccccccc}
$\#$ & ID &\hspace{-20pt} $\M_{101}$\hspace{-1pt} & \hspace{-5pt}$\M_{030}$ & $\M_{202}$ & \hspace{-4pt}$\M_{200}$\hspace{-4pt}\, & $\M_{400}$\hspace{-2pt} &\hspace{-4pt} $\M_{010}$\hspace{-4pt} & $\M_{020}$\hspace{-4pt} & $\M_{002}$\hspace{-4pt} & $\M_{004}$ & $z_1$  &&&\hspace{-5pt}$z_2$\hspace{4pt}{\,}&& & $z_3$ & \\\toprule
1 & $\langle1,1\rangle$ & 54 & 1215 & 4734 & 18 & 486 & 9 & 99 & 164 & 47148 & 0 & 0\hspace{-8pt}&\hspace{-8pt}0\hspace{-8pt}&\hspace{-8pt}0\hspace{-8pt}&\hspace{-8pt}0\hspace{-8pt}&\hspace{-8pt}0 & 0\\
2 & $\langle2,1\rangle$ & 26 & 611 & 2374 & 10 & 246 & 5 & 51 & 84 & 23596 & 0 & 0\hspace{-8pt}&\hspace{-8pt}0\hspace{-8pt}&\hspace{-8pt}0\hspace{-8pt}&\hspace{-8pt}0\hspace{-8pt}&\hspace{-8pt}0 & 0\\
3 & $\langle2,1\rangle$ & 27 & 621 & 2367 & 9 & 243 & 6 & 54 & 82 & 23574 & $\nicefrac{1}{2}$ & 0\hspace{-8pt}&\hspace{-8pt}0\hspace{-8pt}&\hspace{-8pt}0\hspace{-8pt}&\hspace{-8pt}0\hspace{-8pt}&\hspace{-8pt}$\nicefrac{1}{2}$ & $\nicefrac{1}{2}$\\
4 & $\langle3,1\rangle$ & 18 & 405 & 1578 & 6 & 162 & 3 & 33 & 56 & 15720 & $\nicefrac{2}{3}$ & 0\hspace{-8pt}&\hspace{-8pt}$\nicefrac{2}{3}$\hspace{-8pt}&\hspace{-8pt}0\hspace{-8pt}&\hspace{-8pt}0\hspace{-8pt}&\hspace{-8pt}0 & 0\\
5 & $\langle4,1\rangle$ & 13 & 305 & 1187 & 5 & 123 & 2 & 26 & 42 & 11798 & $\nicefrac{1}{2}$ & $\nicefrac{1}{2}$\hspace{-8pt}&\hspace{-8pt}0\hspace{-8pt}&\hspace{-8pt}0\hspace{-8pt}&\hspace{-8pt}0\hspace{-8pt}&\hspace{-8pt}0 & $\nicefrac{1}{2}$\\
6 & $\langle4,1\rangle$ & 14 & 309 & 1194 & 6 & 126 & 3 & 27 & 44 & 11820 & 0 & 0\hspace{-8pt}&\hspace{-8pt}0\hspace{-8pt}&\hspace{-8pt}0\hspace{-8pt}&\hspace{-8pt}0\hspace{-8pt}&\hspace{-8pt}0 & 0 & \!a\\
7 & $\langle4,1\rangle$ & 24 & 443 & 1614 & 10 & 198 & 5 & 43 & 68 & 14444 & 0 & 0\hspace{-8pt}&\hspace{-8pt}0\hspace{-8pt}&\hspace{-8pt}0\hspace{-8pt}&\hspace{-8pt}0\hspace{-8pt}&\hspace{-8pt}0 & 0\\
8 & $\langle4,2\rangle$ & 12 & 309 & 1194 & 6 & 126 & 3 & 27 & 44 & 11820 & 0 & 0\hspace{-8pt}&\hspace{-8pt}0\hspace{-8pt}&\hspace{-8pt}0\hspace{-8pt}&\hspace{-8pt}0\hspace{-8pt}&\hspace{-8pt}0 & 0 & \!a\\
9 & $\langle4,2\rangle$ & 13 & 319 & 1187 & 5 & 123 & 4 & 30 & 42 & 11798 & $\nicefrac{1}{2}$ & 0\hspace{-8pt}&\hspace{-8pt}0\hspace{-8pt}&\hspace{-8pt}0\hspace{-8pt}&\hspace{-8pt}0\hspace{-8pt}&\hspace{-8pt}$\nicefrac{1}{2}$ & $\nicefrac{1}{2}$\\
10 & $\langle6,1\rangle$ & 8 & 206 & 796 & 4 & 84 & 2 & 18 & 30 & 7882 & $\nicefrac{1}{3}$ & 0\hspace{-8pt}&\hspace{-8pt}$\nicefrac{1}{3}$\hspace{-8pt}&\hspace{-8pt}0\hspace{-8pt}&\hspace{-8pt}0\hspace{-8pt}&\hspace{-8pt}0 & 0\\
11 & $\langle6,1\rangle$ & 9 & 216 & 789 & 3 & 81 & 3 & 21 & 28 & 7860 & $\nicefrac{5}{6}$ & 0\hspace{-8pt}&\hspace{-8pt}$\nicefrac{1}{3}$\hspace{-8pt}&\hspace{-8pt}0\hspace{-8pt}&\hspace{-8pt}0\hspace{-8pt}&\hspace{-8pt}$\nicefrac{1}{2}$ & $\nicefrac{1}{2}$\\
12 & $\langle6,2\rangle$ & 9 & 207 & 789 & 3 & 81 & 2 & 18 & 28 & 7860 & $\nicefrac{5}{6}$ & 0\hspace{-8pt}&\hspace{-8pt}$\nicefrac{2}{3}$\hspace{-8pt}&\hspace{-8pt}0\hspace{-8pt}&\hspace{-8pt}0\hspace{-8pt}&\hspace{-8pt}$\nicefrac{1}{6}$ & $\nicefrac{1}{2}$\\
13 & $\langle7,1\rangle$ & 12 & 201 & 732 & 6 & 90 & 3 & 21 & 32 & 6936 & 0 & 0\hspace{-8pt}&\hspace{-8pt}0\hspace{-8pt}&\hspace{-8pt}0\hspace{-8pt}&\hspace{-8pt}0\hspace{-8pt}&\hspace{-8pt}0 & 0\\
14 & $\langle8,1\rangle$ & 7 & 155 & 597 & 3 & 63 & 2 & 14 & 22 & 5910 & $\nicefrac{1}{2}$ & 0\hspace{-8pt}&\hspace{-8pt}0\hspace{-8pt}&\hspace{-8pt}$\nicefrac{1}{2}$\hspace{-8pt}&\hspace{-8pt}0\hspace{-8pt}&\hspace{-8pt}0 & $\nicefrac{1}{2}$\\
15 & $\langle8,1\rangle$ & 12 & 225 & 812 & 6 & 102 & 3 & 23 & 36 & 7236 & 0 & 0\hspace{-8pt}&\hspace{-8pt}0\hspace{-8pt}&\hspace{-8pt}0\hspace{-8pt}&\hspace{-8pt}0\hspace{-8pt}&\hspace{-8pt}0 & 0\\
16 & $\langle8,2\rangle$ & 7 & 161 & 597 & 3 & 63 & 2 & 16 & 22 & 5910 & $\nicefrac{1}{2}$ & $\nicefrac{1}{4}$\hspace{-8pt}&\hspace{-8pt}0\hspace{-8pt}&\hspace{-8pt}0\hspace{-8pt}&\hspace{-8pt}0\hspace{-8pt}&\hspace{-8pt}$\nicefrac{1}{4}$ & $\nicefrac{1}{2}$ & \!b\\
17 & $\langle8,2\rangle$ & 12 & 225 & 814 & 6 & 102 & 3 & 23 & 36 & 7244 & 0 & 0\hspace{-8pt}&\hspace{-8pt}0\hspace{-8pt}&\hspace{-8pt}0\hspace{-8pt}&\hspace{-8pt}0\hspace{-8pt}&\hspace{-8pt}0 & 0\\
18 & $\langle8,3\rangle$ & 6 & 158 & 604 & 4 & 66 & 2 & 15 & 24 & 5932 & 0 & 0\hspace{-8pt}&\hspace{-8pt}0\hspace{-8pt}&\hspace{-8pt}0\hspace{-8pt}&\hspace{-8pt}0\hspace{-8pt}&\hspace{-8pt}0 & 0 & c\\
19 & $\langle8,3\rangle$ & 6 & 161 & 597 & 3 & 63 & 2 & 16 & 22 & 5910 & $\nicefrac{1}{2}$ & $\nicefrac{1}{4}$\hspace{-8pt}&\hspace{-8pt}0\hspace{-8pt}&\hspace{-8pt}0\hspace{-8pt}&\hspace{-8pt}0\hspace{-8pt}&\hspace{-8pt}$\nicefrac{1}{4}$ & $\nicefrac{1}{2}$ & \!b\\
20 & $\langle8,3\rangle$ & 7 & 168 & 597 & 3 & 63 & 3 & 18 & 22 & 5910 & $\nicefrac{1}{2}$ & 0\hspace{-8pt}&\hspace{-8pt}0\hspace{-8pt}&\hspace{-8pt}0\hspace{-8pt}&\hspace{-8pt}0\hspace{-8pt}&\hspace{-8pt}$\nicefrac{1}{2}$ & $\nicefrac{1}{2}$ & \!d\\
21 & $\langle8,3\rangle$ & 12 & 235 & 807 & 5 & 99 & 4 & 26 & 34 & 7222 & $\nicefrac{1}{2}$ & 0\hspace{-8pt}&\hspace{-8pt}0\hspace{-8pt}&\hspace{-8pt}0\hspace{-8pt}&\hspace{-8pt}0\hspace{-8pt}&\hspace{-8pt}$\nicefrac{1}{2}$ & $\nicefrac{1}{2}$\\
22 & $\langle8,4\rangle$ & 8 & 158 & 604 & 4 & 66 & 2 & 15 & 24 & 5932 & 0 & 0\hspace{-8pt}&\hspace{-8pt}0\hspace{-8pt}&\hspace{-8pt}0\hspace{-8pt}&\hspace{-8pt}0\hspace{-8pt}&\hspace{-8pt}0 & 0 & c\\
23 & $\langle8,4\rangle$ & 12 & 221 & 807 & 5 & 99 & 2 & 22 & 34 & 7222 & $\nicefrac{1}{2}$ & $\nicefrac{1}{2}$\hspace{-8pt}&\hspace{-8pt}0\hspace{-8pt}&\hspace{-8pt}0\hspace{-8pt}&\hspace{-8pt}0\hspace{-8pt}&\hspace{-8pt}0 & $\nicefrac{1}{2}$\\
24 & $\langle8,5\rangle$ & 6 & 168 & 597 & 3 & 63 & 3 & 18 & 22 & 5910 & $\nicefrac{1}{2}$ & 0\hspace{-8pt}&\hspace{-8pt}0\hspace{-8pt}&\hspace{-8pt}0\hspace{-8pt}&\hspace{-8pt}0\hspace{-8pt}&\hspace{-8pt}$\nicefrac{1}{2}$ & $\nicefrac{1}{2}$ &\!d\\
25 & $\langle12,3\rangle$ & 4 & 103 & 398 & 2 & 42 & 1 & 9 & 16 & 3944 & $\nicefrac{2}{3}$ & 0\hspace{-8pt}&\hspace{-8pt}$\nicefrac{2}{3}$\hspace{-8pt}&\hspace{-8pt}0\hspace{-8pt}&\hspace{-8pt}0\hspace{-8pt}&\hspace{-8pt}0 & 0\\
26 & $\langle12,4\rangle$ & 4 & 112 & 398 & 2 & 42 & 2 & 12 & 15 & 3941 & $\nicefrac{2}{3}$ & 0\hspace{-8pt}&\hspace{-8pt}$\nicefrac{1}{3}$\hspace{-8pt}&\hspace{-8pt}0\hspace{-8pt}&\hspace{-8pt}0\hspace{-8pt}&\hspace{-8pt}$\nicefrac{1}{3}$ & $\nicefrac{1}{2}$\\
27 & $\langle14,1\rangle$ & 6 & 114 & 366 & 3 & 45 & 3 & 15 & 16 & 3468 & $\nicefrac{1}{2}$ & 0\hspace{-8pt}&\hspace{-8pt}0\hspace{-8pt}&\hspace{-8pt}0\hspace{-8pt}&\hspace{-8pt}0\hspace{-8pt}&\hspace{-8pt}$\nicefrac{1}{2}$ & $\nicefrac{1}{2}$\\
28 & $\langle16,2\rangle$ & 12 & 183 & 624 & 6 & 90 & 3 & 21 & 32 & 4956 & 0 & 0\hspace{-8pt}&\hspace{-8pt}0\hspace{-8pt}&\hspace{-8pt}0\hspace{-8pt}&\hspace{-8pt}0\hspace{-8pt}&\hspace{-8pt}0 & 0\\
29 & $\langle16,6\rangle$ & 6 & 113 & 407 & 3 & 51 & 2 & 12 & 18 & 3622 & $\nicefrac{1}{2}$ & 0\hspace{-8pt}&\hspace{-8pt}0\hspace{-8pt}&\hspace{-8pt}$\nicefrac{1}{2}$\hspace{-8pt}&\hspace{-8pt}0\hspace{-8pt}&\hspace{-8pt}0 & $\nicefrac{1}{2}$\\
30 & $\langle16,6\rangle$ & 6 & 116 & 412 & 4 & 54 & 2 & 13 & 20 & 3636 & 0 & 0\hspace{-8pt}&\hspace{-8pt}0\hspace{-8pt}&\hspace{-8pt}0\hspace{-8pt}&\hspace{-8pt}0\hspace{-8pt}&\hspace{-8pt}0 & 0\\
31 & $\langle16,7\rangle$ & 3 & 86 & 302 & 2 & 33 & 2 & 10 & 12 & 2966 & $\nicefrac{1}{2}$ & 0\hspace{-8pt}&\hspace{-8pt}0\hspace{-8pt}&\hspace{-8pt}$\nicefrac{1}{4}$\hspace{-8pt}&\hspace{-8pt}0\hspace{-8pt}&\hspace{-8pt}$\nicefrac{1}{4}$ & $\nicefrac{1}{2}$ & \!e\\
32 & $\langle16,7\rangle$ & 6 & 126 & 406 & 3 & 51 & 3 & 16 & 18 & 3618 & $\nicefrac{1}{2}$ & 0\hspace{-8pt}&\hspace{-8pt}0\hspace{-8pt}&\hspace{-8pt}0\hspace{-8pt}&\hspace{-8pt}0\hspace{-8pt}&\hspace{-8pt}$\nicefrac{1}{2}$ & $\nicefrac{1}{2}$\\
33 & $\langle16,8\rangle$ & 4 & 86 & 302 & 2 & 33 & 2 & 10 & 12 & 2966 & $\nicefrac{1}{2}$ & 0\hspace{-8pt}&\hspace{-8pt}0\hspace{-8pt}&\hspace{-8pt}$\nicefrac{1}{4}$\hspace{-8pt}&\hspace{-8pt}0\hspace{-8pt}&\hspace{-8pt}$\nicefrac{1}{4}$ & $\nicefrac{1}{2}$& \!e\\
34 & $\langle16,8\rangle$ & 6 & 119 & 406 & 3 & 51 & 2 & 14 & 18 & 3618 & $\nicefrac{1}{2}$ & $\nicefrac{1}{4}$\hspace{-8pt}&\hspace{-8pt}0\hspace{-8pt}&\hspace{-8pt}0\hspace{-8pt}&\hspace{-8pt}0\hspace{-8pt}&\hspace{-8pt}$\nicefrac{1}{4}$ & $\nicefrac{1}{2}$\\
35 & $\langle16,11\rangle$ & 3 & 89 & 302 & 2 & 33 & 2 & 11 & 12 & 2966 & $\nicefrac{1}{2}$ & $\nicefrac{1}{8}$\hspace{-8pt}&\hspace{-8pt}0\hspace{-8pt}&\hspace{-8pt}0\hspace{-8pt}&\hspace{-8pt}0\hspace{-8pt}&\hspace{-8pt}$\nicefrac{3}{8}$ & $\nicefrac{1}{2}$& \!f\\
36 & $\langle16,11\rangle$ & 6 & 126 & 407 & 3 & 51 & 3 & 16 & 18 & 3622 & $\nicefrac{1}{2}$ & 0\hspace{-8pt}&\hspace{-8pt}0\hspace{-8pt}&\hspace{-8pt}0\hspace{-8pt}&\hspace{-8pt}0\hspace{-8pt}&\hspace{-8pt}$\nicefrac{1}{2}$ & $\nicefrac{1}{2}$\\
37 & $\langle16,13\rangle$ & 4 & 89 & 302 & 2 & 33 & 2 & 11 & 12 & 2966 & $\nicefrac{1}{2}$ & $\nicefrac{1}{8}$\hspace{-8pt}&\hspace{-8pt}0\hspace{-8pt}&\hspace{-8pt}0\hspace{-8pt}&\hspace{-8pt}0\hspace{-8pt}&\hspace{-8pt}$\nicefrac{3}{8}$ & $\nicefrac{1}{2}$ & \!f\\
38 & $\langle16,13\rangle$ & 6 & 116 & 414 & 4 & 54 & 2 & 13 & 20 & 3644 & 0 & 0\hspace{-8pt}&\hspace{-8pt}0\hspace{-8pt}&\hspace{-8pt}0\hspace{-8pt}&\hspace{-8pt}0\hspace{-8pt}&\hspace{-8pt}0 & 0\\
39 & $\langle16,13\rangle$ & 6 & 119 & 407 & 3 & 51 & 2 & 14 & 18 & 3622 & $\nicefrac{1}{2}$ & $\nicefrac{1}{4}$\hspace{-8pt}&\hspace{-8pt}0\hspace{-8pt}&\hspace{-8pt}0\hspace{-8pt}&\hspace{-8pt}0\hspace{-8pt}&\hspace{-8pt}$\nicefrac{1}{4}$ & $\nicefrac{1}{2}$\\
40 & $\langle21,1\rangle$ & 4 & 67 & 244 & 2 & 30 & 1 & 7 & 12 & 2316 & $\nicefrac{2}{3}$ & 0\hspace{-8pt}&\hspace{-8pt}$\nicefrac{2}{3}$\hspace{-8pt}&\hspace{-8pt}0\hspace{-8pt}&\hspace{-8pt}0\hspace{-8pt}&\hspace{-8pt}0 & 0\\
41 & $\langle24,12\rangle$ & 2 & 55 & 206 & 2 & 24 & 1 & 6 & 10 & 1994 & $\nicefrac{1}{3}$ & 0\hspace{-8pt}&\hspace{-8pt}$\nicefrac{1}{3}$\hspace{-8pt}&\hspace{-8pt}0\hspace{-8pt}&\hspace{-8pt}0\hspace{-8pt}&\hspace{-8pt}0 & 0\\
42 & $\langle24,12\rangle$ & 2 & 58 & 199 & 1 & 21 & 1 & 7 & 8 & 1972 & $\nicefrac{5}{6}$ & $\nicefrac{1}{4}$\hspace{-8pt}&\hspace{-8pt}$\nicefrac{1}{3}$\hspace{-8pt}&\hspace{-8pt}0\hspace{-8pt}&\hspace{-8pt}0\hspace{-8pt}&\hspace{-8pt}$\nicefrac{1}{4}$ & $\nicefrac{1}{2}$\\
43 & $\langle24,13\rangle$ & 2 & 56 & 199 & 1 & 21 & 1 & 6 & 8 & 1972 & $\nicefrac{5}{6}$ & 0\hspace{-8pt}&\hspace{-8pt}$\nicefrac{2}{3}$\hspace{-8pt}&\hspace{-8pt}0\hspace{-8pt}&\hspace{-8pt}0\hspace{-8pt}&\hspace{-8pt}$\nicefrac{1}{6}$ & $\nicefrac{1}{2}$\\
44 & $\langle32,7\rangle$ & 3 & 65 & 206 & 2 & 27 & 2 & 9 & 10 & 1818 & $\nicefrac{1}{2}$ & 0\hspace{-8pt}&\hspace{-8pt}0\hspace{-8pt}&\hspace{-8pt}$\nicefrac{1}{4}$\hspace{-8pt}&\hspace{-8pt}0\hspace{-8pt}&\hspace{-8pt}$\nicefrac{1}{4}$ & $\nicefrac{1}{2}$\\
45 & $\langle32,11\rangle$ & 6 & 95 & 312 & 3 & 45 & 2 & 12 & 16 & 2478 & $\nicefrac{1}{2}$ & $\nicefrac{1}{8}$\hspace{-8pt}&\hspace{-8pt}0\hspace{-8pt}&\hspace{-8pt}$\nicefrac{1}{4}$\hspace{-8pt}&\hspace{-8pt}0\hspace{-8pt}&\hspace{-8pt}$\nicefrac{1}{8}$ & $\nicefrac{1}{2}$\\
46 & $\langle32,11\rangle$ & 6 & 95 & 318 & 4 & 48 & 2 & 12 & 18 & 2496 & 0 & 0\hspace{-8pt}&\hspace{-8pt}0\hspace{-8pt}&\hspace{-8pt}0\hspace{-8pt}&\hspace{-8pt}0\hspace{-8pt}&\hspace{-8pt}0 & 0\\
47 & $\langle32,34\rangle$ & 6 & 105 & 312 & 3 & 45 & 3 & 15 & 16 & 2478 & $\nicefrac{1}{2}$ & 0\hspace{-8pt}&\hspace{-8pt}0\hspace{-8pt}&\hspace{-8pt}0\hspace{-8pt}&\hspace{-8pt}0\hspace{-8pt}&\hspace{-8pt}$\nicefrac{1}{2}$ & $\nicefrac{1}{2}$\\
48 & $\langle32,43\rangle$ & 3 & 65 & 207 & 2 & 27 & 2 & 9 & 10 & 1822 & $\nicefrac{1}{2}$ & 0\hspace{-8pt}&\hspace{-8pt}0\hspace{-8pt}&\hspace{-8pt}$\nicefrac{1}{4}$\hspace{-8pt}&\hspace{-8pt}0\hspace{-8pt}&\hspace{-8pt}$\nicefrac{1}{4}$ & $\nicefrac{1}{2}$\\
49 & $\langle32,43\rangle$ & 3 & 68 & 206 & 2 & 27 & 2 & 10 & 10 & 1818 & $\nicefrac{1}{2}$ & $\nicefrac{1}{8}$\hspace{-8pt}&\hspace{-8pt}0\hspace{-8pt}&\hspace{-8pt}0\hspace{-8pt}&\hspace{-8pt}0\hspace{-8pt}&\hspace{-8pt}$\nicefrac{3}{8}$ & $\nicefrac{1}{2}$\\
50 & $\langle32,49\rangle$ & 3 & 68 & 207 & 2 & 27 & 2 & 10 & 10 & 1822 & $\nicefrac{1}{2}$ & $\nicefrac{1}{8}$\hspace{-8pt}&\hspace{-8pt}0\hspace{-8pt}&\hspace{-8pt}0\hspace{-8pt}&\hspace{-8pt}0\hspace{-8pt}&\hspace{-8pt}$\nicefrac{3}{8}$ & $\nicefrac{1}{2}$\\
51 & $\langle42,1\rangle$ & 2 & 38 & 122 & 1 & 15 & 1 & 5 & 6 & 1158 & $\nicefrac{5}{6}$ & 0\hspace{-8pt}&\hspace{-8pt}$\nicefrac{2}{3}$\hspace{-8pt}&\hspace{-8pt}0\hspace{-8pt}&\hspace{-8pt}0\hspace{-8pt}&\hspace{-8pt}$\nicefrac{1}{6}$ & $\nicefrac{1}{2}$\\
52 & $\langle48,3\rangle$ & 4 & 61 & 208 & 2 & 30 & 1 & 7 & 12 & 1656 & $\nicefrac{2}{3}$ & 0\hspace{-8pt}&\hspace{-8pt}$\nicefrac{2}{3}$\hspace{-8pt}&\hspace{-8pt}0\hspace{-8pt}&\hspace{-8pt}0\hspace{-8pt}&\hspace{-8pt}0 & 0\\
53 & $\langle48,48\rangle$ & 1 & 33 & 103 & 1 & 12 & 1 & 5 & 5 & 997 & $\nicefrac{2}{3}$ & $\nicefrac{1}{8}$\hspace{-8pt}&\hspace{-8pt}$\nicefrac{1}{3}$\hspace{-8pt}&\hspace{-8pt}0\hspace{-8pt}&\hspace{-8pt}0\hspace{-8pt}&\hspace{-8pt}$\nicefrac{5}{24}$ & $\nicefrac{1}{2}$\\
54 & $\langle64,134\rangle$ & 3 & 56 & 159 & 2 & 24 & 2 & 9 & 9 & 1248 & $\nicefrac{1}{2}$ & $\nicefrac{1}{16}$\hspace{-8pt}&\hspace{-8pt}0\hspace{-8pt}&\hspace{-8pt}$\nicefrac{1}{8}$\hspace{-8pt}&\hspace{-8pt}0\hspace{-8pt}&\hspace{-8pt}$\nicefrac{5}{16}$ & $\nicefrac{1}{2}$\\
55 & $\langle96,64\rangle$ & 2 & 34 & 104 & 1 & 15 & 1 & 5 & 6 & 828 & $\nicefrac{5}{6}$ & $\nicefrac{1}{8}$\hspace{-8pt}&\hspace{-8pt}$\nicefrac{1}{3}$\hspace{-8pt}&\hspace{-8pt}$\nicefrac{1}{4}$\hspace{-8pt}&\hspace{-8pt}0\hspace{-8pt}&\hspace{-8pt}$\nicefrac{1}{8}$ & $\nicefrac{1}{2}$\\
56 & $\langle96,64\rangle$ & 2 & 34 & 110 & 2 & 18 & 1 & 5 & 8 & 846 & $\nicefrac{1}{3}$ & 0\hspace{-8pt}&\hspace{-8pt}$\nicefrac{1}{3}$\hspace{-8pt}&\hspace{-8pt}0\hspace{-8pt}&\hspace{-8pt}0\hspace{-8pt}&\hspace{-8pt}0 & 0\\
57 & $\langle96,72\rangle$ & 2 & 35 & 104 & 1 & 15 & 1 & 5 & 6 & 828 & $\nicefrac{5}{6}$ & 0\hspace{-8pt}&\hspace{-8pt}$\nicefrac{2}{3}$\hspace{-8pt}&\hspace{-8pt}0\hspace{-8pt}&\hspace{-8pt}0\hspace{-8pt}&\hspace{-8pt}$\nicefrac{1}{6}$ & $\nicefrac{1}{2}$\\
58 & $\langle168,42\rangle$ & 2 & 19 & 52 & 2 & 12 & 1 & 4 & 6 & 366 & $\nicefrac{1}{3}$ & 0\hspace{-8pt}&\hspace{-8pt}$\nicefrac{1}{3}$\hspace{-8pt}&\hspace{-8pt}0\hspace{-8pt}&\hspace{-8pt}0\hspace{-8pt}&\hspace{-8pt}0 & 0\\
59 & $\langle192,956\rangle$\hspace{-6pt}\, & 1 & 21 & 55 & 1 & 9 & 1 & 4 & 4 & 423 & $\nicefrac{2}{3}$ & $\nicefrac{1}{16}$\hspace{-8pt}&\hspace{-8pt}$\nicefrac{1}{3}$\hspace{-8pt}&\hspace{-8pt}$\nicefrac{1}{8}$\hspace{-8pt}&\hspace{-8pt}0\hspace{-8pt}&\hspace{-8pt}$\nicefrac{7}{48}$ & $\nicefrac{1}{2}$\\
60 & $\langle336,208\rangle$\hspace{-6pt}\, & 1 & 12 & 26 & 1 & 6 & 1 & 3 & 3 & 183 & $\nicefrac{2}{3}$ & 0\hspace{-8pt}&\hspace{-8pt}$\nicefrac{1}{3}$\hspace{-8pt}&\hspace{-8pt}$\nicefrac{1}{4}$\hspace{-8pt}&\hspace{-8pt}0\hspace{-8pt}&\hspace{-8pt}$\nicefrac{1}{12}$ & $\nicefrac{1}{2}$\\
\bottomrule
\end{tabular}
\end{center}
\end{table}

\begin{table}[htb]
\begin{center}
\footnotesize
\setlength{\extrarowheight}{0.3pt}
\setlength{\tabcolsep}{4pt}
\caption{Sato-Tate statistics for Fermat twists over degree one primes $p\le 2^{26}$.}\label{table: Fermat statistics}
\begin{tabular}{rrrrrrrcrrrrrr}\vspace{5pt}
&\multicolumn{6}{c}{$\ST(C_k)$} &&\multicolumn{6}{c}{$\ST(C_{kM})$}\\
$\#$ & \hspace{8pt}$\overline\M_{101}$ & \hspace{-4pt}$\M_{101}$ & \hspace{6pt}$\overline\M_{030}$ & \hspace{-4pt}$\M_{030}$ & \hspace{10pt}$\overline\M_{202}$ & \hspace{-4pt}$\M_{202}$ && \hspace{8pt}$\overline\M_{101}$ &\hspace{-4pt} $\M_{101}$ & \hspace{10pt}$\overline\M_{030}$ &\hspace{-4pt} $\M_{030}$ & \hspace{12pt}$\overline\M_{202}$ &\hspace{-4pt} $\M_{202}$  \\\toprule
1 & 26.99 & 27 & 620.72 & 621 & 2365.83 & 2367 && 53.99 & 54 & 1214.69 & 1215 & 4732.64 & 4734\\
4 & 12.98 & 13 & 304.55 & 305 & 1185.12 & 1187 && 25.98 & 26 & 610.36 & 611 & 2371.23 & 2374\\
5 & 12.99 & 13 & 318.75 & 319 & 1185.94 & 1187 && 25.99 & 26 & 610.61 & 611 & 2372.38 & 2374\\
11 & 8.99 & 9 & 215.63 & 216 & 787.39 & 789 && 17.98 & 18 & 404.33 & 405 & 1575.11 & 1578\\
12 & 9.00 & 9 & 206.88 & 207 & 788.45 & 789 && 18.00 & 18 & 404.84 & 405 & 1577.24 & 1578\\
13 & 6.98 & 7 & 154.56 & 155 & 595.23 & 597 && 13.97 & 14 & 308.25 & 309 & 1190.96 & 1194\\
14 & 6.99 & 7 & 160.84 & 161 & 596.37 & 597 && 13.99 & 14 & 308.75 & 309 & 1192.99 & 1194\\
17 & 5.99 & 6 & 160.80 & 161 & 596.20 & 597 && 11.99 & 12 & 308.67 & 309 & 1192.65 & 1194\\
20 & 6.99 & 7 & 167.74 & 168 & 595.89 & 597 && 13.98 & 14 & 308.54 & 309 & 1192.02 & 1194\\
24 & 11.99 & 12 & 234.80 & 235 & 806.10 & 807 && 23.99 & 24 & 442.70 & 443 & 1612.54 & 1614\\
27 & 11.99 & 12 & 220.64 & 221 & 805.56 & 807 && 23.98 & 24 & 442.43 & 443 & 1611.64 & 1614\\
28 & 5.99 & 6 & 167.77 & 168 & 596.03 & 597 && 11.98 & 12 & 308.60 & 309 & 1192.31 & 1194\\
32 & 4.00 & 4 & 111.98 & 112 & 397.87 & 398 && 8.00 & 8 & 205.99 & 206 & 795.91 & 796\\
33 & 6.00 & 6 & 112.89 & 113 & 406.55 & 407 && 12.00 & 12 & 224.88 & 225 & 813.43 & 814\\
34 & 3.00 & 3 & 85.98 & 86 & 301.92 & 302 && 6.00 & 6 & 157.99 & 158 & 603.96 & 604\\
35 & 5.99 & 6 & 125.77 & 126 & 405.04 & 406 && 11.98 & 12 & 224.57 & 225 & 810.25 & 812\\
36 & 5.99 & 6 & 118.74 & 119 & 404.95 & 406 && 11.98 & 12 & 224.52 & 225 & 810.06 & 812\\
37 & 4.00 & 4 & 85.99 & 86 & 301.98 & 302 && 8.00 & 8 & 158.00 & 158 & 604.09 & 604\\
38 & 2.99 & 3 & 88.81 & 89 & 301.24 & 302 && 5.99 & 6 & 157.65 & 158 & 602.60 & 604\\
41 & 5.99 & 6 & 125.74 & 126 & 405.90 & 407 && 11.98 & 12 & 224.53 & 225 & 811.97 & 814\\
44 & 6.00 & 6 & 118.88 & 119 & 406.47 & 407 && 11.99 & 12 & 224.80 & 225 & 813.12 & 814\\
46 & 3.99 & 4 & 88.73 & 89 & 300.88 & 302 && 7.98 & 8 & 157.50 & 158 & 601.89 & 604\\
47 & 2.00 & 2 & 57.88 & 58 & 198.48 & 199 && 3.99 & 4 & 102.77 & 103 & 397.05 & 398\\
48 & 1.99 & 2 & 55.81 & 56 & 198.20 & 199 && 3.99 & 4 & 102.64 & 103 & 396.49 & 398\\
49 & 2.99 & 3 & 64.85 & 65 & 205.42 & 206 && 5.99 & 6 & 115.72 & 116 & 410.92 & 412\\
50 & 6.00 & 6 & 94.92 & 95 & 311.67 & 312 && 12.00 & 12 & 182.87 & 183 & 623.48 & 624\\
51 & 5.99 & 6 & 104.72 & 105 & 310.80 & 312 && 11.98 & 12 & 182.47 & 183 & 621.72 & 624\\
52 & 3.00 & 3 & 64.94 & 65 & 206.72 & 207 && 6.00 & 6 & 115.89 & 116 & 413.52 & 414\\
53 & 2.99 & 3 & 67.84 & 68 & 205.30 & 206 && 5.99 & 6 & 115.71 & 116 & 410.68 & 412\\
54 & 3.00 & 3 & 67.88 & 68 & 206.50 & 207 && 5.99 & 6 & 115.78 & 116 & 413.08 & 414\\
55 & 1.00 & 1 & 32.93 & 33 & 102.73 & 103 && 1.99 & 2 & 54.86 & 55 & 205.50 & 206\\
56 & 2.99 & 3 & 55.77 & 56 & 158.06 & 159 && 5.98 & 6 & 94.56 & 95 & 316.20 & 318\\
57 & 1.99 & 2 & 33.87 & 34 & 103.46 & 104 && 3.99 & 4 & 60.76 & 61 & 206.96 & 208\\
58 & 2.00 & 2 & 34.93 & 35 & 103.71 & 104 && 4.00 & 4 & 60.87 & 61 & 207.45 & 208\\
59 & 1.00 & 1 & 20.99 & 21 & 54.95 & 55 && 2.00 & 2 & 33.98 & 34 & 109.92 & 110\\\bottomrule
\end{tabular}
\vspace{10pt}
\end{center}
\end{table}

\begin{table}[htb]
\begin{center}
\footnotesize
\setlength{\extrarowheight}{0.3pt}
\setlength{\tabcolsep}{4pt}
\caption{Sato-Tate statistics for Klein twists over degree one primes $p\le 2^{26}$.}\label{table: Klein statistics}
\begin{tabular}{rrrrrrrcrrrrrr}\vspace{5pt}
&\multicolumn{6}{c}{$\ST(C_k)$} &&\multicolumn{6}{c}{$\ST(C_{kM})$}\\
$\#$ & \hspace{8pt}$\overline\M_{101}$ & \hspace{-4pt}$\M_{101}$ & \hspace{6pt}$\overline\M_{030}$ & \hspace{-4pt}$\M_{030}$ & \hspace{10pt}$\overline\M_{202}$ & \hspace{-4pt}$\M_{202}$ && \hspace{8pt}$\overline\M_{101}$ &\hspace{-4pt} $\M_{101}$ & \hspace{10pt}$\overline\M_{030}$ &\hspace{-4pt} $\M_{030}$ & \hspace{12pt}$\overline\M_{202}$ &\hspace{-4pt} $\M_{202}$  \\\toprule
1 & 26.99 & 27 & 620.78 & 621 & 2366.04 & 2367 && 53.99 & 54 & 1214.67 & 1215 & 4732.50 & 4734\\
2 & 12.99 & 13 & 318.73 & 319 & 1185.85 & 1187 && 25.98 & 26 & 610.52 & 611 & 2371.92 & 2374\\
3 & 11.99 & 12 & 308.67 & 309 & 1192.59 & 1194 && 11.99 & 12 & 308.67 & 309 & 1192.59 & 1194\\
4 & 8.99 & 9 & 215.76 & 216 & 787.97 & 789 && 17.98 & 18 & 404.55 & 405 & 1576.08 & 1578\\
5 & 9.00 & 9 & 206.95 & 207 & 788.79 & 789 && 18.00 & 18 & 404.94 & 405 & 1577.71 & 1578)\\
6 & 7.00 & 7 & 154.90 & 155 & 596.58 & 597 && 14.00 & 14 & 308.88 & 309 & 1193.45 & 1194\\
7 & 6.99 & 7 & 167.80 & 168 & 596.13 & 597 && 13.99 & 14 & 308.63 & 309 & 1192.36 & 1194\\
8 & 4.01 & 4 & 103.36 & 103 & 399.55 & 398 && 4.01 & 4 & 103.36 & 103 & 399.55 & 398\\
9 & 3.99 & 4 & 111.68 & 112 & 396.63 & 398 && 7.98 & 8 & 205.37 & 206 & 793.34 & 796\\
10 & 5.99 & 6 & 113.83 & 114 & 365.24 & 366 && 11.99 & 12 & 200.67 & 201 & 730.55 & 732\\
11 & 3.00 & 3 & 85.94 & 86 & 301.73 & 302 && 6.00 & 6 & 157.89 & 158 & 603.51 & 604\\
12 & 2.01 & 2 & 55.11 & 55 & 206.43 & 206 && 2.01 & 2 & 55.11 & 55 & 206.43 & 206\\
13 & 2.00 & 2 & 37.97 & 38 & 121.81 & 122 && 4.00 & 4 & 66.94 & 67 & 243.65 & 244\\
14 & 1.00 & 1 & 11.94 & 12 & 25.70 & 26 && 2.00 & 2 & 18.87 & 19 & 51.40 & 52\\\bottomrule
\end{tabular}
\end{center}
\end{table}

\end{document}